\documentclass[preprint,3p]{elsarticle}

\usepackage{geometry}
\date{}
\makeatletter
\def\ps@pprintTitle{%
  \let\@oddhead\@empty
  \let\@evenhead\@empty
  \let\@oddfoot\@empty
  \let\@evenfoot\@empty
}
\makeatother

\usepackage{xcolor}

\usepackage{amsmath}
\usepackage{amssymb}
\usepackage{amsfonts}
\usepackage{amsthm}
\usepackage{mathtools} 
\newtheorem{corollary}{Corollary}
\newtheorem{theorem}{Theorem}

\newtheorem{remark}{Remark}

\usepackage{enumerate}

\usepackage{graphicx}
\usepackage{subcaption}

\usepackage{hyperref}
\hypersetup{
    colorlinks=true, 
    urlcolor=blue,   
    linkcolor=blue,  
    citecolor=blue   
}

\usepackage[capitalize]{cleveref} 

\usepackage{natbib}

\usepackage{algorithm}
\usepackage{algpseudocode}

\usepackage[markup=underlined]{changes}
\definechangesauthor[name={Alejandro}, color=blue]{AS}

\newcommand{\bfq}{\boldsymbol{q}}
\newcommand{\bfx}{\boldsymbol{x}}
\newcommand{\bfPhi}{\boldsymbol{\Phi}}

\DeclareMathOperator{\Div}{div}

\newcommand{\btheta}{\boldsymbol{\theta}}

\newcommand{\vertiii}[1]{{\left\vert\kern-0.25ex\left\vert\kern-0.25ex\left\vert #1 
    \right\vert\kern-0.25ex\right\vert\kern-0.25ex\right\vert}}

\usepackage{tikz}

\usetikzlibrary{positioning, arrows.meta, decorations.pathreplacing, calc}
\usepackage{etoolbox} 
\usepackage{listofitems} 
\tikzset{>=latex} 
\colorlet{myred}{red!80!black}
\colorlet{myblue}{blue!80!black}
\colorlet{mygreen}{green!60!black}
\colorlet{mydarkred}{myred!40!black}
\colorlet{mydarkblue}{myblue!40!black}
\colorlet{mydarkgreen}{mygreen!40!black}

\tikzstyle{node}=[very thick,circle,draw=myblue,minimum size=22,inner sep=0.5,outer sep=0.6]
\tikzstyle{connect}=[->,thick,mydarkblue,shorten >=1]

\tikzset{ 
  node 1/.style={node,mydarkgreen,draw=mygreen,fill=white},
  node 2/.style={node,mydarkblue,draw=myblue,fill=white},
  node 3/.style={node,mydarkred,draw=myred,fill=white},
}

\tikzset{
  tanhicon/.pic={
    \draw[gray, thick, -{Latex[length=3pt]}] (-0.4,0) -- (0.4,0);
    \draw[gray, thick, -{Latex[length=3pt]}] (0,-0.4) -- (0,0.4);
    \draw[very thick, red!70!black, rounded corners=0.5pt] 
      plot[smooth, samples=150, domain=-0.3:0.3] (\x,{0.4*tanh(7*\x)});
  }
}

\usepackage{changes}
\definechangesauthor[name=DP, color=red]{DP}
\definechangesauthor[name=CU, color=blue]{CU}
\definechangesauthor[name=PS, color=orange]{PS}
\definechangesauthor[name=ADS, color=gray]{ADS}

\begin{document}

\begin{frontmatter}

\title{Robust Deep FOSLS for Transmission Problems}

\author[label1]{Alejandro Duque-Salazar}
\ead{alejandro.duque@ehu.eus}

\author[label2]{Paulina Sep\'ulveda}
\ead{paulina.sepulveda@pucv.cl}

\author[label3,label4]{Carlos Uriarte}
\ead{curiarte@bcamath.org}
\ead{carlos.uriarte@curtin.edu.au}

\author[label5]{Jamie M. Taylor}
\ead{jamie.taylor@cunef.edu}

\author[label1,label3,label6]{David Pardo}
\ead{david.pardo@ehu.eus}
\ead{dpardo@bcamath.org}

\address[label1]{University of the Basque Country (UPV/EHU), Leioa, Spain}
\address[label2]{Pontificia Universidad Católica de Valparaíso (PUCV), Chile}
\address[label3]{Basque Center for Applied Mathematics (BCAM), Bilbao, Spain}
\address[label4]{Curtin University, Perth, Australia}
\address[label5]{Department of Mathematics, CUNEF Universidad, Madrid, Spain}
\address[label6]{Basque Foundation for Science (Ikerbasque), Bilbao, Spain}

\begin{abstract}
This work presents a robust, energy-based deep learning framework for solving transmission problems in heterogeneous media, including cases with discontinuous material scenarios. We introduce a weighted First-Order System Least-Squares (FOSLS) formulation involving an energy-norm Poincar\'e constant and prove its equivalence to a natural energy norm of the underlying equations, with constants independent of material parameters. As a result, the optimization landscape remains aligned with a meaningful error approximation even under high material contrast, where standard neural network losses often deteriorate. We further prove that the FOSLS formulation, together with its integral-loss representation, exhibits a \emph{passive variance reduction} property, whereby the gradient variance progressively decreases as the loss diminishes, in contrast to methods such as VPINNs and Deep Ritz. 
From a numerical standpoint, we adopt a reduced-order perspective by constructing a low-dimensional space described by a neural network. The optimal coefficients are computed via a least-squares solver, and the space is subsequently improved through gradient-based updates.
By selecting the activation function $\operatorname{ReQU}$, the method mitigates the spurious overshoots typically observed in smooth networks when approximating discontinuities. Numerical experiments in 1D and 2D interface settings corroborate these findings.

\end{abstract}


\end{frontmatter}

\section{Introduction}

Transmission problems in heterogeneous media occur in many applications, including heat conduction and wave propagation. Partial differential equations (PDEs) describe these phenomena, and their coefficients encode media properties. Where materials are discontinuous, coefficients vary and reduce PDE solution regularity, creating challenges for numerical methods \cite{bonito2013adaptive}. Capturing the solution accurately near material interfaces is essential. Discontinuous material coefficients may cause jumps in the gradient or tangential flux component, leading standard numerical approaches to spurious oscillations or accuracy loss \cite{gottlieb1997gibbs}. These challenges make alternative methods necessary to solve such PDEs accurately.

Physics-Informed Neural Networks (PINNs) \cite{raissi2019physics,jiao2021rate,cai2021physics} approximate the solution of the PDE via a neural network by minimizing a loss function that inherently relies on the strong form of the equation. However, this strong formulation imposes restrictive regularity requirements on the target solution. In problems with discontinuous coefficients, the residual can exhibit Dirac  delta-type singularities, violating the $L^2$-regularity assumption of the residual and making the strong form unsuitable \cite{taylor2023deep}. Consequently, these problems naturally motivate the use of weak or variational formulations.  

A mathematically natural way to relax these regularity requirements is to move from strong to weak formulations based on variational principles. This has motivated the development of Variational Physics-Informed Neural Networks (VPINNs) \cite{kharazmi2019variational,berrone2022variational, rojas2024robust} and Deep Ritz methods \cite{yu2018deep, uriarte2023deep}, which seek approximate solutions by minimizing functionals arising from weak formulations of the PDE. These approaches are attractive because they are consistent with classical finite element theory and naturally accommodate discontinuous coefficients and weak solutions. However, a major practical challenge remains in optimizing the resulting ariational integrals:  existing VPINNs and Deep Ritz formulations lack  a {\emph{ passive variance reduction}} property during training. In practice, the variance of the stochastic loss gradient often remains large even as the approximation improves, since quadrature noise is not inherently coupled to the PDE error. This can affect convergence unless  an excessively large number of quadrature points are used for approximating the gradient. 

To overcome these challenges, we turn to the First-Order System Least Squares (FOSLS) methodology, which reformulates the second-order PDE as an equivalent system of first-order equations \cite{cai1994first, bochev2009least,munzenmaier2011first}. 
As a result, each equation contributes an $L^2$-residual, avoiding the need to evaluate higher derivatives of the solution. Its components remain in their natural spaces $H^1$ for the primal variable and $H(\Div)$ for the flux.  
From a physical standpoint, the first-order formulation is especially meaningful because it treats the primary variable and its associated flux as independent functions. This separation aligns with fundamental physics laws such as Fourier's law of heat conduction \cite{fuhrer2021space} or Darcy’s law in porous media \cite{bochev2008locally}, where the flux represents a measurable physical quantity. Furthermore, the first-order system makes conservation principles explicit through balance equations, providing a more transparent representation of the underlying physics. In heterogeneous and composite materials, this formulation naturally accommodates discontinuities in material properties and allows for a physically consistent treatment of interface conditions across material boundaries \cite{doi:10.1137/S0036142903427688,doi:10.1137/S003614290342769X}. Another advantage of the FOSLS formulation is that it exhibits a passive variance reduction property: as the neural approximation approaches the exact solution, the variance in the stochastic approximation of the gradient of the loss naturally decreases, which accelerates stochastic gradient descent-based optimization.

While recent “Deep FOSLS” methods have explored this approach \cite{cai2020deep,bersetche2023deep,meissner2025deep, qiu2025variationally, 10.1093/imanum/draf073, castillo2025dpg}, a limitation remains: the equivalence constants relating the loss functional to the true error are influenced by physical coefficients, which may reduce accuracy in problems with high material contrast: the loss and the true error effectively decouple, so that a small loss no longer guarantees an accurate solution. To overcome this, we introduce a physically meaningful energy norm for the solution and its flux, along with a loss functional for the residual that is both physically motivated and mathematically balanced through the use of the energy-norm Poincaré constant. As a result, we prove that the equivalence constants  between the error and the loss become independent of the physical coefficients, ensuring robust convergence even in the presence of extreme parameter discontinuities.

Our method proposes a hybrid solver-in-the-loop strategy that employs a vector-valued neural network as an adaptive generator of low-dimensional approximation spaces in the spirit of reduced-order modeling \cite{quarteroni2015reduced}. The optimal solution is computed on these spaces via a linear least-squares solver, while the spaces themselves are updated using gradient-based optimization. That is, we aim to find a low-dimensional subspace so that its least-squares minimizer ---which may be readily computed--- is an accurate solution to our PDE.

Discrete neural approximations of PDE solutions may suffer from the quasi-Gibbs phenomenon, namely, localized oscillations that appear near discontinuities in the trained approximation when globally smooth neural networks attempt to represent non-smooth solutions through extremely steep transition layers  \cite{gottlieb1997gibbs, cai2021least}. Although the network remains smooth, these narrow layers can produce spurious overshoots or undershoots in the solution or its gradient near material interfaces, an effect that becomes particularly pronounced in high-contrast transmission problems and when highly regular activation functions such as $\tanh$ are employed. Several approaches have been proposed to alleviate this behavior. Adaptive sampling and integration-based strategies concentrate collocation points in regions with sharp variations, but they do not modify the intrinsic approximation space of the network and may therefore remain susceptible to quasi-Gibbs oscillations  \cite{gao2023failure}. 
Other strategies rely on localized refinement or material-dependent activations, mimicking adaptive finite-element techniques, yet they often require additional error estimators and may compromise the scalability of neural network approaches \cite{taylor2024adaptive, badia2025adaptive, sarma2024interface}. In contrast, our work addresses quasi-Gibbs behavior by choosing $\operatorname{ReQU}$ as the activation function. This choice is informed by approximation results showing that ReLU networks can represent discontinuities without exhibiting Gibbs-type oscillations \cite{cai2021least, cai2024least}. In the one-dimensional, single-hidden-layer case, this property admits an equivalent interpretation in terms of approximating jump discontinuities in the derivative. Guided by this observation, we improve the approximation of transmission problems without resorting to domain decomposition or local refinement strategies.

Despite the advantages, the current approach presents several challenges. First, the formulation introduces a new vector variable, increasing the number of unknowns. In practice, however, this extra cost is moderate in neural network settings, since the solution is represented through low-dimensional parametric representations. 
Second, we have not explored the use of different neural networks for the solution and the flux, and since these variables belong to different functional spaces, they may exhibit different approximation requirements near discontinuities. We leave the possibility of employing multiple networks for future work.
Finally, other architectural choices, hyperparameter settings, and the hybrid use of different losses for updating the spaces remain to be studied. Addressing these aspects represents a natural direction for future research.

The remainder of this manuscript is organized as follows.
\cref{sec:section2} establishes the theoretical foundation of the framework by introducing a weighted FOSLS functional and proving its robustness equivalence to the energy norm of the error.
\cref{sec:discrete} describes the finite-dimensional approximation, a procedure for estimating the energy-norm Poincaré constant on the discrete spaces required for the weighted functional, and a strategy for updating these spaces.
\cref{sec:section3} presents a neural network architecture that induces discrete approximation spaces, and discusses their initialization and regularity properties.
\cref{sec:integration} discusses the discretization of the loss functional through  stochastic integration, and states a theoretical result on passive variance reduction.
Finally, \cref{sec:section6} presents numerical experiments in one and two dimensions that validate the proposed methodology, compare it with standard non-robust schemes, illustrating its advantages over Ritz formulations and tanh-based networks, in particular, by mitigating the quasi-Gibbs phenomenon in problems with discontinuous coefficients, high material contrast, and coarse quadrature rules.

\section{Problem Statement}\label{sec:section2}

Let $\Omega \subset \mathbb{R}^d$, for $d \in \{1,2,3\}$, be a bounded domain with boundary $\partial \Omega$. We consider the diffusion equation with homogeneous Dirichlet boundary conditions given by the following formulation: find $u$ such that
\begin{equation}\label{eq: diffusion}
    \left\{ \begin{array}{rcll}
         -\Div (\kappa \nabla u) &=& f& \text{in } \Omega, \\
         u &=&0 & \text{on } \partial \Omega,
    \end{array} \right.
\end{equation}
where $\kappa \in L^{\infty}(\Omega)$ satisfies $\kappa\geq\kappa_0>0$ 
in $\Omega$ and  \(f \in L^2(\Omega) \). 

\begin{remark}
In the case of inhomogeneous Dirichlet boundary conditions, let ${u}$ denote the corresponding solution. Introducing a lift $u_D$ of the boundary data, i.e., a sufficiently regular function coinciding with the prescribed values on $\partial \Omega$, the decomposition $u - u_D$ reduces the problem to the homogeneous form \eqref{eq: diffusion}, with modified source term.
\end{remark}

Let  $H^1_0(\Omega):= \{ u\in L^2(\Omega):  \nabla u \in (L^2(\Omega))^d, u|_{\partial \Omega} = 0\}$. The weak formulation of problem~\eqref{eq: diffusion} reads as: find $u \in H_0^1(\Omega)$ such that
\begin{equation}\label{eq: weak_form}
 b(u,v):=\int_\Omega \kappa\nabla u\cdot\nabla v \, = \int_\Omega f v=:\ell(v), \qquad \forall v \in H^1_0(\Omega).
\end{equation}

The Lax-Milgram theorem~\cite{Ern2021} guarantees the existence and uniqueness of the solution of~\eqref{eq: weak_form}. Since the bilinear form $b$ is symmetric and positive definite, it induces an \emph{energy-type} inner-product and norm in $H^1_0(\Omega)$ defined as
\begin{equation}\label{eq: energy_norm}
(u,v)_{H^1_{0,\kappa}} := b(u,v) = \int_\Omega \kappa\nabla u\cdot\nabla v,\qquad\|u\|_{H^1_{0,\kappa}}^2 := b(u,u)= \int_\Omega \kappa |\nabla u|^2,\qquad u,v\in H^1_0(\Omega).
\end{equation}
This norm is more physically and mathematically meaningful than the standard $\kappa$-independent norms in $H^1(\Omega)$ and $H^1_0(\Omega)$, namely,
\begin{align*}
    \Vert\cdot\Vert_{H^1} := (\Vert\cdot\Vert_{L^2}^2+\Vert\nabla\cdot\Vert_{L^2}^2)^{1/2},\qquad
    \Vert\cdot\Vert_{H^1_0} := \Vert\nabla\cdot\Vert_{L^2}.
\end{align*}
 
For instance, when Eq.~\eqref{eq: diffusion} models the steady heat equation, where $u$ denotes the temperature field and $\kappa$ the thermal conductivity, the energy norm $\Vert\cdot\Vert_{H^1_{0,\kappa}}$ represents the total energy dissipated by heat conduction in maintaining the steady state. In this sense, it plays an analogous role to mechanical work in continuum mechanics, where the relevant energy norm quantifies the work required to sustain a deformation.

\subsection{First-Order System Least Squares (FOSLS)}\label{sec: FOSLS}
Let $\bfq := - \kappa \nabla u$ denote the flux. The diffusion problem \eqref{eq: diffusion} can be equivalently rewritten as the first-order system of equations
\begin{equation}\label{eq: FOS_PDE}
    \left\{\begin{array}{rcll}
         \Div \bfq &=& f,& \text{in } \Omega, \\
         \bfq + \kappa \nabla u&=&  0,& \text{in } \Omega, \\
         u &=&0, & \text{on } \partial \Omega .
    \end{array}\right.
\end{equation}
Since $f\in L^2(\Omega)$, the search spaces are naturally given by 
\begin{equation*}
    u \in H^1_0(\Omega),\qquad \bfq \in H(\Div; \Omega) := \{ \boldsymbol{q} \in (L^2(\Omega))^d:  \Div \bfq \in L^2(\Omega)\}.
\end{equation*}
For any $(u,\bfq)$ in the product space $H_0^1(\Omega) \times H(\Div;\Omega)$, the mathematical norm is often given by (see, e.g., \cite{cai1994first}) 
\begin{equation}\label{eq:standard_norm}
\| (u,\bfq) \|_{H^1_{0} \times H(\Div)} := \left(\|u\|^2_{H_{0}^1} + \| \bfq\|_{(L^2)^d}^2   +  \| \Div \bfq\|_{L^2}^2\right)^{1/2}.
\end{equation}

However, this norm has two limitations. First, because 
it does not incorporate the energy norm introduced in Eq.~\eqref{eq: energy_norm}. Second, for the exact solution $(u^*,\bfq^*)$ of~\eqref{eq: FOS_PDE},
the equivalence constants between $\|u^*\|_{H_0^1}$ and $\|\bfq^*\|_{L^2}$ are $\kappa$-dependent and possibly large. 

To overcome this, we introduce the following weighted norm in $H_0^1(\Omega)\times H(\Div;\Omega)$:
\begin{equation}\label{eq: error_norm}
\vertiii{(u,\bfq)}_\kappa := \left(\|u\|^2_{H_{0,\kappa}^1} + \| \bfq\|_{(L_{\kappa}^2)^d}^2   + C \| \Div \bfq\|_{L^2}^2\right)^{1/2},
\end{equation} where $C>0$ is a weighting coefficient that we will determine, $\|\cdot\|_{H^1_{0,\kappa}}$ is the energy norm defined in~\eqref{eq: energy_norm}, and $\|\cdot\|_{(L_{\kappa}^2)^d}$ is a weighted $L^2$-norm for the flux defined by 
\[\|\bfq\|_{(L_{\kappa}^2)^d} := \|\kappa^{-1/2} \bfq\|_{(L^2)^d}.\]

\subsection{Continuous loss functional}
We propose the following weighted functional associated with the first-order system given by Eq.~\eqref{eq: FOS_PDE}:
\begin{equation}\label{eq: loss}
\mathcal{L}(u,\bfq):= \Vert\kappa^{-1/2}\bfq+\kappa^{1/2}\nabla u\Vert^2_{(L^2)^d} + C_{\mathcal{L}}\Vert\Div\bfq-f\Vert^2_{L^2},
\end{equation}
for some constant $C_{\mathcal{L}}>0$ to be determined. For $C_{\mathcal{L}} = 1$, this reduces to the standard loss proposed in \cite{cai1994first, cai2020deep}, where the authors show that it is equivalent to the standard norm given by Eq. \eqref{eq:standard_norm}, with $\kappa$-dependent equivalence constants. 

Herein, we will show an equivalence of the loss in Eq.~\eqref{eq: loss} with the error measured by the \emph{physically meaningful} norm in Eq.~\eqref{eq: error_norm}, and choose $C$ and $C_\mathcal{L}$ to ensure that the upper- and lower-bound constants are $\kappa$-independent, therefore leading to a robust method.  We state and prove the following theorem.

\begin{theorem}\label{thm: theorem}
Let $C >0$  and $C_{\mathcal{L}}>0$ be the constants defined in~\eqref{eq: error_norm} and~\eqref{eq: loss}, respectively. Then, the loss function in Eq.~\eqref{eq: loss} satisfies:
\begin{equation*}\label{eq: robustness}
c_1\vertiii{ (u,\bfq) - (u^*,\bfq^*)}_\kappa^2 \leq \mathcal{L}(u,\bfq) \leq c_2\vertiii{ (u,\bfq) - (u^*,\bfq^*)}_\kappa^2,
\end{equation*}
with
\[
c_1 = \min \left\{\frac{1}{8}, \frac{C_{\mathcal{L}}}{6(C^P_{\kappa})^2 + {C}}\right\}, \qquad c_2 = 2 \max\left\{1, \frac{C_{\mathcal{L}}}{2C}\right\},
\]
where $C_\kappa^P$ is the energy-norm Poincar\'e constant 
defined as:
\[
C_\kappa^P := \sup_{0\neq v \in H_0^1(\Omega)} \frac{\|v\|_{L^2}}{\|v\|_{H_{0,\kappa}^1}}.
\]
\end{theorem}

\begin{proof}
We first establish the upper bound. Applying Young's inequality together with the definition of the weighted norm, we have: 
\begin{align*}
\mathcal{L}(u,\bfq) &\leq 2\| \kappa^{1/2} \nabla u- \kappa^{1/2} \nabla u^* \|_{(L^2)^d}^2+ 2 \| \kappa^{-1/2} \bfq  + \kappa^{1/2} \nabla u^* \|_{(L^2)^d} ^2 +  C_{\mathcal{L}}\Vert\Div\bfq-f\Vert^2_{L^2}, \\&=2\| u- u^* \|_{H^1_{0, \kappa}}^2+ 2 \| \kappa^{-1/2} \bfq  + \kappa^{1/2} \nabla u^* \|_{(L^2)^d} ^2 +  C_\mathcal{L}\Vert\Div\bfq-f\Vert^2_{L^2}.
\end{align*}
The exact solution satisfies $ \kappa^{1/2} \nabla u^* = -\kappa^{-1/2} \bfq^*$  and $\Div \bfq^* = f$. Thus,
\begin{equation*}\label{eq: c2}
\begin{split}
    \mathcal{L}(u,\bfq) &\leq  2\| u- u^* \|_{H^1_{0, \kappa}}^2+  2\|\bfq - \bfq^*\|_{(L^2_{\kappa})^d}^2 + C_{\mathcal{L}}\Vert\Div\bfq-\Div \bfq^*\Vert^2_{L^2},\\
    & \leq 2 \max\left\{1, \frac{C_{\mathcal{L}}}{2C}\right\}\vertiii{(u,\bfq)-(u^*,\bfq^*)}_\kappa^2.
\end{split}
\end{equation*}
As a result, we consider $c_2 =  2 \max\left\{1, \tfrac{C_{\mathcal{L}}}{2C}\right\}$.

Next, for computing $c_1$ we consider bounds for each term of the energy norm of the error. First, we recall the equivalence between the energy norm of the error $u-u^*$ and the norm of the residual $r_u\in H^{-1}(\Omega):=(H^1_0(\Omega))^*$ defined by
\[
r_u(v) := b(u,v) - \ell(v), \qquad v \in H_0^1(\Omega).
\]
Specifically,
\begin{equation}\label{eq: energy_error_residual}
\|u-u^*\|_{H^1_{0,\kappa}}=\|r_u\|_{({H^1_{0,\kappa}})^*}=\sup\limits_{0\neq v\in H^1_0(\Omega)}{\frac{1}{\|v\|_{H^1_{0,\kappa}}} \int_\Omega \kappa\nabla u\cdot\nabla v-fv\,}.
\end{equation} We emphasize that we are equipping $H^{-1}(\Omega)$ with the dual norm induced via the energy norm in \eqref{eq: energy_norm}. An estimate of this term can be obtained by bounding the last integral in \eqref{eq: energy_error_residual}. To this end, we add and subtract $\bfq \cdot \nabla v$ and then integrate by parts. Then, for all $v\in H^1_0(\Omega)$,
\begin{equation*}
\begin{split}
\left(\int_\Omega \kappa\nabla u\cdot\nabla v-fv\, \right)^2 &=\left(\int_\Omega (\kappa^{-1/2} \bfq  + \kappa^{1/2} \nabla u)\cdot\kappa^{\frac{1}{2}}\nabla v-\bfq\cdot\nabla v-fv\,\right)^2\\
&= \left(\int_\Omega (\kappa^{-1/2} \bfq  + \kappa^{1/2} \nabla u)\cdot\kappa^{\frac{1}{2}}\nabla v\,+ \Div(\bfq) v-fv\, \right)^2.
\end{split}
\end{equation*}
Applying Young's inequality, followed by Cauchy-Schwarz's and the energy-norm Poincar\'{e} ($\|\cdot \|_{L^2}  \leq C^P_{\kappa} \|\cdot \|_{H^1_{0,\kappa}}$) inequalities, we obtain: for all $v\in H^1_0(\Omega)$,
\begin{align}\label{sec: residual_squared}
\left(\int_\Omega \kappa\nabla u\cdot\nabla v-fv\, \right)^2  &\leq 2\left(\int_\Omega (\kappa^{-1/2} \bfq  + \kappa^{1/2} \nabla u)\cdot\kappa^{\frac{1}{2}}\nabla v\,\right)^2+ 2\left(\int_\Omega \Div(\bfq) v-fv\, \right)^2, \nonumber\\
&\leq 2\|\kappa^{-1/2} \bfq  + \kappa^{1/2} \nabla u\|_{(L^2)^d}^2\|v\|_{H^1_{0,\kappa}}^2+2\|\Div(\bfq)-f\|_{L^2}^2\|v\|_{L^2}^2,\nonumber\\
&\leq 2\|\kappa^{-1/2} \bfq  + \kappa^{1/2} \nabla u\|_{(L^2)^d}^2\|v\|_{H^1_{0,\kappa}}^2+2(C^P_{\kappa})^2\|\Div(\bfq)-f\|_{L^2}^2\|v\|_{H^1_{0,\kappa}}^2. 
\end{align} Using the upper bound~\eqref{sec: residual_squared} in \eqref{eq: energy_error_residual}, we arrive at
\begin{equation}\label{eq: u_error_estimation}
\|u-u^*\|_{H^1_{0,\kappa}}^2 \leq  2\|\kappa^{-\frac{1}{2}}\bfq+\kappa^\frac{1}{2}\nabla u\|_{(L^2)^d}^2 + 2(C^P_{\kappa})^2\|\Div(\bfq)-f\|_{L^2}^2.
\end{equation}
Analogously, we estimate the error of the flux in the corresponding weighted norm. Applying Young's inequality together with \eqref{eq: u_error_estimation}, we obtain
\begin{align}\label{eq: q_error_estimation}
\|\bfq-\bfq^*\|_{(L_{\kappa}^2)^d}^2&= \|\bfq+\kappa \nabla u^*\|_{(L_{\kappa}^2)^d}^2, \nonumber\\
&\leq 2\|\bfq+\kappa\nabla u\|_{(L_{\kappa}^2)^d}^2+2\|\kappa\nabla u - \kappa\nabla u^*\|_{(L_{\kappa}^2)^d}^2, \nonumber\\
&= 2\|\kappa^{-\frac{1}{2}}\bfq+\kappa^\frac{1}{2}\nabla u\|_{(L^2)^d}^2+2\|u - u^*\|_{H^1_{0,\kappa}}^2, \nonumber\\
&\leq 2\|\kappa^{-\frac{1}{2}}\bfq+\kappa^\frac{1}{2}\nabla u\|_{(L^2)^d}^2+4\left(\|\kappa^{-\frac{1}{2}}\bfq+\kappa^\frac{1}{2}\nabla u\|_{(L^2)^d}^2+(C^P_{\kappa})^2\|\Div(\bfq)-f\|_{L^2}^2\right), \nonumber\\
&= 6\|\kappa^{-\frac{1}{2}}\bfq+\kappa^\frac{1}{2}\nabla u\|_{(L^2)^d}^2+4(C^P_{\kappa})^2\|\Div(\bfq)-f\|_{L^2}^2.
\end{align}
Combining  \eqref{eq: u_error_estimation} and \eqref{eq: q_error_estimation}, the  bound for the energy norm of the error reads as follows:
\begin{equation*}\label{eq: c1}
\begin{split}
\vertiii{(u,\bfq) - (u^*,\bfq^*)}_\kappa^2  &\leq  8\|\kappa^{-\frac{1}{2}}\bfq+\kappa^\frac{1}{2}\nabla u\|_{L^2}^2 + (6(C^P_{\kappa})^2 + { C})\|\Div(\bfq)-f\|_{L^2}^2,\\
&\leq \max \left\{8, \frac{6(C^P_{\kappa})^2 + C}{C_{\mathcal{L}}} \right\} \mathcal{L}(u,\bfq).
\end{split}
\end{equation*}
Therefore, we consider:
\[
c_1 = \min \left\{\frac{1}{8}, \frac{C_{\mathcal{L}}}{6(C^P_{\kappa})^2 + C}\right\}.
\]
\end{proof}
\begin{remark}
This approach can also be readily extended to Neumann conditions by restricting the trial space of the flux to the subspace of $H(\Div)$ functions such that $\bfq \cdot \mathbf{n} =0$ where the Neumann condition is imposed.
\end{remark}
\begin{corollary}\label{cor:1}
 Let $C$ and $C_{\mathcal{L}}$ in Eqs.~\eqref{eq: error_norm} and~\eqref{eq: loss} be defined as:
\begin{equation*}\label{no} 
C := (C_\kappa^P)^2,\qquad C_\mathcal{L} := 2(C_\kappa^P)^2, 
\end{equation*}
Then, the constants $c_1$ and $c_2$ in Theorem~\ref{thm: theorem} are independent of $\Omega$ and $\kappa$. Moreover,  $c_1=1/8$ and $c_2 = 2$, so that $c_2/c_1=16$. 

The corresponding norm and loss take the explicit form
\begin{align}\label{eq: final_norm}
\vertiii{(u,\bfq)}_\kappa^2
&= \Big( \|u\|^2_{H_{0,\kappa}^1} 
   + \|\bfq\|_{(L_{\kappa}^2)^d}^2   
   + (C^P_{\kappa})^2 \|\Div \bfq\|_{L^2}^2 \Big), \\[1ex]
\mathcal{L}(u,\bfq)
&= \|\kappa^{-1/2}\bfq+\kappa^{1/2}\nabla u\|^2_{(L^2)^d} 
   + 2 (C^P_{\kappa})^2 \|\Div \bfq-f\|^2_{L^2}.\label{eq: loss final}
\end{align}
\end{corollary}
 \begin{proof}
It follows from Theorem~\ref{thm: theorem}.
 \end{proof}
From now on, we define the following norm in $H(\Div)$ as
$$\|\bfq\|^2_{H(\Div, \kappa)}:=\|\bfq\|_{(L_{\kappa}^2)^d}^2 + (C_\kappa^P)^2\|\Div \bfq\|^2_{L^2}.$$
Thus, the energy norm can be expressed as:
\begin{align*}
\vertiii{(u,\bfq)}_\kappa^2
&=  \|u\|^2_{H_{0,\kappa}^1} 
   + \|\bfq\|_{H(\Div,\kappa)}^2. 
\end{align*}

\section{Discrete Subspace Approximation}
\label{sec:discrete}

We seek an approximation of the solution of the continuous problem in
finite-dimensional subspaces
\[
V_{\boldsymbol{\theta}}^u \times V_{\boldsymbol{\theta}}^{\bfq}
\;\subset\;
H_{0}^1 \times H(\mathrm{div}).
\]
Here, 
$V_{\btheta}^u$ and $V_{\btheta}^q$ are the discrete subspaces of $H^1_0$ and $H(\operatorname{div})$ where the approximations $u$ and $\bfq$ are correspondingly sought. The mapping $ \btheta \mapsto V_{\btheta}^u\times V_{\btheta}^q$ denotes a parametrization that defines this space.
The spaces $V_{\btheta}^u$, $V_{\btheta}^{\bfq}$ are defined via a given number of parametrised spanning functions, 
\[
V_{\boldsymbol{\theta}}^u := \mathrm{span}\{\varphi_i\}_{i=1}^{n_u},
\qquad
V_{\boldsymbol{\theta}}^{\bfq} := \mathrm{span}\{\boldsymbol{\tau}_i\}_{i=1}^{n_q}.
\]
Notice that for notational convenience we are omitting the dependence of these functions on $\btheta$. Elements $(u,\bfq)\in V_{\boldsymbol{\theta}}^u \times V_{\boldsymbol{\theta}}^{\bfq}$  are therefore represented as
\begin{align}\label{eq: basis_expansion}
u &= \sum_{i=1}^{n_u} c_i^u \,{\varphi}_i,
\qquad
\bfq = \sum_{i=1}^{n_q} c_i^{\bfq} \,{\boldsymbol{\tau}}_i,
\end{align}
where we collect the coefficients into the vector
$
\boldsymbol{c}:=(\boldsymbol{c}^u,\boldsymbol{c}^{\bfq})\in\mathbb{R}^n,$  $n:=n_u+n_q.
$

\subsection{Continuous loss over a discrete subspace}\label{sec: optimal_coefficients}
 
In terms of the coefficients ${\boldsymbol{c}}$, the quadratic loss of Eq.~\eqref{eq: loss final} can be written as
\begin{equation}
\mathcal{L}(u,\bfq)
=
{\boldsymbol{c}}^\top
\mathbf{H}_{\boldsymbol{\theta}}
{\boldsymbol{c}}
-
2\,{\boldsymbol{c}}^\top  {\boldsymbol{f}}_{\boldsymbol{\theta}}
+
\ell_{\boldsymbol{\theta}},
\label{eq:loss_algebraic}
\end{equation}
where $\mathbf{H}_{\boldsymbol{\theta}}\in\mathbb{R}^{n\times n}$ is symmetric and
positive semidefinite and with the following block decomposition
\[
\mathbf{H}_{\boldsymbol{\theta}}
=
\begin{pmatrix}
\mathbf{H}^{uu} & \mathbf{H}^{u\bfq} \\
(\mathbf{H}^{u\bfq})^\top & \mathbf{H}^{\bfq\bfq}
\end{pmatrix},
\]
with blocks defined by
\begin{align}
(\mathbf{H}^{uu})_{ij}
&=
\int_\Omega
\kappa \,\nabla {\varphi}_i \cdot \nabla {\varphi}_j
 ,
\qquad
(\mathbf{H}^{u\bfq})_{ij}
=
\int_\Omega
\nabla {\varphi}_i \cdot {\boldsymbol{\tau}}_j
 ,  \label{eq:Huu}
\\
(\mathbf{H}^{\bfq\bfq})_{ij} 
&=
\int_\Omega
\kappa^{-1}
\,{\boldsymbol{\tau}}_i \cdot {\boldsymbol{\tau}}_j
+
2\bigl(C_\kappa^P\bigr)^2
\int_\Omega
(\Div {\boldsymbol{\tau}}_i)
(\Div {\boldsymbol{\tau}}_j).\nonumber
\end{align}
The linear and constant terms are given by
\[
(\boldsymbol{f}_{\boldsymbol{\theta}})_j
= 
\begin{cases}
    0, \qquad  1\leq j \leq n_u,\\
    2\bigl(C_\kappa^P\bigr)^2 \int_\Omega f\, \Div {\boldsymbol{\tau}}_{j-n_u}, \qquad n_u + 1 \leq j \leq n
\end{cases},\qquad
\ell_{\boldsymbol{\theta}}
=
2\bigl(C_\kappa^P\bigr)^2
\int_\Omega
f^2.
\]

\paragraph{Least-Squares formulation} 
We are interested in minimizing the loss functional over the discrete spaces, i.e,
\begin{equation*}
(u_{\btheta}, \bfq_{\btheta}) :=\underset{(u,\bfq)\in V_{\boldsymbol{\theta}}^u \times V_{\boldsymbol{\theta}}^{\bfq}}{\mathrm{arg min}}
\mathcal{L}(u,\bfq),
\end{equation*}
where $\mathcal{L}$ denotes the  least--squares functional in~\eqref{eq:loss_algebraic}, and $(u_{\btheta}, \bfq_{\btheta})$ represents a function  that minimizes it over $V_{\boldsymbol{\theta}}^u \times V_{\boldsymbol{\theta}}^{\bfq}$.

\paragraph{Least-Squares solver}

To gain stability for the least-squares solver, we apply a change of variables to the discrete system and introduce a simple diagonal scaling 
based on the block entries of $\mathbf{H}_{\boldsymbol{\theta}}$, which provides an effective 
spectral normalization at negligible additional cost. For that, we define:

\[
 d_i := \sqrt{(\mathbf{H}^{uu})_{ii}+\varepsilon} \quad \text{for } i=1,\dots,n_u,
 \qquad
 d_{n_u+j} := \sqrt{(\mathbf{H}^{\bfq\bfq})_{jj}+\varepsilon}
 \quad \text{for } j=1,\dots,n_q,
 \]
and collect these values into a diagonal matrix
\begin{align}
\mathbf{D}_{\boldsymbol{\theta}} = \mathrm{diag}(d_1,\dots,d_n). \label{eq:diag}
\end{align}
In the above, we add a small $\varepsilon>0$
to stabilize the discrete system when $\mathbf{D}_{\btheta}$ becomes singular or near singular, for instance, when some spanning functions have negligible support inside $\Omega$. 

Then, we perform the change of variables
\[
\boldsymbol{c} := \mathbf{D}_{\boldsymbol{\theta}}^{-1}\,\tilde{\boldsymbol{c}}.
\]
Then, \eqref{eq:loss_algebraic} becomes:
\begin{equation}
\mathcal{L}(u,\bfq)
=
{\tilde{\boldsymbol{c}}}^\top
\tilde{\mathbf{H}}_{\boldsymbol{\theta}}
{\tilde{\boldsymbol{c}}}
-
2\,{\tilde{\boldsymbol{c}}}^\top  {\tilde{\boldsymbol{f}}}_{\boldsymbol{\theta}}
+
\ell_{\boldsymbol{\theta}},
\label{eq:loss_algebraic_change}
\end{equation}
where $\tilde{\mathbf{H}}_{\boldsymbol{\theta}}:=\mathbf{D}_{\boldsymbol{\theta}}^{-1}   {\mathbf{H}}_{\boldsymbol{\theta}} \mathbf{D}_{\boldsymbol{\theta}}^{-1}$, and ${\tilde{\boldsymbol{f}}}_{\boldsymbol{\theta}}:=\mathbf{D}_{\boldsymbol{\theta}}^{-1}{{\boldsymbol{f}}}_{\boldsymbol{\theta}}$.

\begin{remark}\label{rk:mu} The scaled matrix $\tilde{\mathbf H}_{\boldsymbol{\theta}}$ is symmetric and
positive semidefinite, with nonnegative diagonal entries bounded by one. 
Hence,
\[
\lambda_{\max}(\tilde{\mathbf H}_{\boldsymbol{\theta}}) \le n.
\]
Since the dimension of the matrix is moderate, we can safely select a Tikhonov parameter $\mu>0$ independent of the particular spanning set. 
\end{remark}
Thus:
\begin{equation*}
\tilde{\boldsymbol{c}}^*
=
\arg\min_{\tilde{\boldsymbol{c}}\in\mathbb{R}^n}
\left(
\tilde{\boldsymbol{c}}^\top
(\tilde{\mathbf{H}}_{\boldsymbol{\theta}}+\mu \mathbf{I})
\tilde{\boldsymbol{c}}
-
2\,\tilde{\boldsymbol{c}}^\top
\tilde{\boldsymbol{f}}_{\boldsymbol{\theta}}
\right)=
\left(
\tilde{\mathbf{H}}_{\boldsymbol{\theta}}+\mu \mathbf{I}
\right)^{-1}
\tilde{\boldsymbol{f}}_{\boldsymbol{\theta}} .
\label{eq:LS_solution}
\end{equation*}

The physical coefficients associated with the original (non-normalized) spanning functions
are then recovered by
\[
\boldsymbol{c}^*
=
\mathbf{D}_{\boldsymbol{\theta}}^{-1}\tilde{\boldsymbol{c}}^* ,
\]
and the corresponding solution $(u_{\boldsymbol{\theta}},\bfq_{\boldsymbol{\theta}})$
is reconstructed using the spanning function expansion in \eqref{eq: basis_expansion}.

\subsection{Approximation of the energy-norm Poincar\'e constant 
\label{sec: Poincare}}
Our framework involves the Poincar\'{e} constant $C_\kappa^P$, whose exact value is generally unknown {\em a priori}. To estimate it, we consider its approximation over the finite-dimensional subspace $V_{\boldsymbol\theta}^u$.  Let
\[
C_{\kappa,{\boldsymbol\theta}}^P := \max_{0\neq v \in V_{\boldsymbol\theta}^u} \frac{\|v\|_{L^2}}{\|v\|_{H_{0,\kappa}^1}} \leq \sup_{0\neq v \in H_0^1(\Omega)} \frac{\|v\|_{L^2}}{\|v\|_{H_{0,\kappa}^1}} = C_{\kappa}^P.
\]  This constant is computed from the underlying eigenvalue problem associated to the second-order PDE: find $(u_\lambda,\lambda)\in V_{\boldsymbol\theta}^u\times \mathbb{R}$ such that
\begin{align*}
\int_\Omega \kappa \nabla u_\lambda \cdot \nabla v 
= \lambda \int_\Omega u_\lambda v,
\qquad \forall v \in V_{\btheta}^u.
\end{align*}
Letting $\lambda_{\min,\boldsymbol\theta}$ be the smallest positive eigenvalue, we then set $C_{\kappa,\boldsymbol\theta}^{P}= (\lambda_{\min,\boldsymbol\theta})^{-1/2}$.

The quality of the approximation is dependent on how well the subspace $V_{\boldsymbol{\theta}}^u$ approximates the first eigenfunction. 
Expanding $u_\lambda$ in terms of a chosen  finite-dimensional set $\{{\varphi}_i\}_{i=1}^{n_u}$ of $V_{\boldsymbol{\theta}}^u$, the computation of  $\lambda_{\min,\boldsymbol{\theta}}$ reduces to solving the generalized eigenvalue problem
\begin{align}
\mathbf{H}^{uu}\boldsymbol{a}=\lambda \mathbf{M}\boldsymbol{a}, \label{eq:eigs_O}
\end{align}
where $\boldsymbol{a}$ is the coordinate vector of the eigenfunction
$u_\lambda \in V_{\boldsymbol{\theta}}^u$ with respect to the spanning functions
$\{{\varphi}_i\}_{i=1}^{n_u}$, $\mathbf{H}^{uu}$ is defined as in Eq.~\eqref{eq:Huu} and
$\mathbf{M}$ denotes the mass matrix, i.e.,
\[
\mathbf{M}_{ij} := \int_\Omega {\varphi}_i {\varphi}_j, \qquad 1\leq i,j\leq n_u. 
\]

Both matrices are symmetric and positive semidefinite. In particular, $\mathbf{M}$ may fail to be positive definite if the spanning functions are linearly dependent in $L^2(\Omega)$ or have negligible or vanishing support within $\Omega$.

As in the least-squares solver, we perform a change of variables and regularization for numerical stability: we first perform a change of variables and then regularize the resulting generalized eigenvalue problem. Instead of solving Eq.~\eqref{eq:eigs_O}, we  consider
\begin{align}
(\tilde{\mathbf{H}}^{uu} + \alpha_1\mathbf{I})\, \tilde{\boldsymbol{a}}
= \lambda\, (\tilde{\mathbf{M}} + \alpha_2 \mathbf{I})\, \tilde{\boldsymbol{a}}, \label{eq:gen eigen}
\end{align}
where $\alpha_1, \alpha_2 >0$ are regularization parameters, 
${\tilde{\mathbf{H}}}^{uu} :=  \mathbf{D}_{\btheta,u} ^{-1}{\mathbf{H}}^{uu} \mathbf{D}_{\btheta,u}^{-1}$, ${\tilde{\mathbf{M}}} :=  \mathbf{D}_{\btheta,u} ^{-1}{\mathbf{M}} \mathbf{D}_{\btheta,u}^{-1}$, $\tilde{\boldsymbol{a}} :=  \mathbf{D}_{\btheta,u} ^{-1} \boldsymbol{a} $, and $\mathbf{D}_{\btheta,u}$ correspond to the upper block of size $n_u$ of the  matrix $\mathbf{D}_{\btheta}$ defined in Eq.~\eqref{eq:diag}.

\begin{remark}
The parameter $\alpha_2>0$ ensures that 
$\tilde{\mathbf{M}}+\alpha_2\mathbf I$ is invertible and $\alpha_1>0$ regularizes the stiffness matrix 
$\tilde{\mathbf{H}}^{uu}$, which is generally only positive semidefinite and may have a null eigenvalue. 
In practice, both parameters are chosen to be small and fixed, just large enough to obtain an accurate approximation of the smallest eigenvalue, which is the primary quantity of interest. Although the regularization perturbs the exact discrete eigenvalues,
standard eigenvalue perturbation theory implies that, for sufficiently
small $\alpha_1$ and $\alpha_2$, the smallest eigenvalue of
\eqref{eq:gen eigen} remains a consistent approximation of
$\lambda_{\min,\boldsymbol{\theta}}$.
\end{remark}

\begin{remark}
In practice,  $\mathbf{H}^{uu}$ is assembled as part of the least--squares loss evaluation, since it depends only on the discrete space $V_{\boldsymbol{\theta}}^u$. Because the loss itself involves the Poincar\'e constant, this computation is performed once per update of the spaces. The extra cost for solving the energy-norm Poincar\'e constant is limited to the solution of the small generalized eigenvalue problem, which remains inexpensive as long as the dimension of the parameter space is moderate.
\end{remark}

\subsection{Update of the parametric spaces}
We define the loss functional
\begin{equation}
\mathcal{J}:
\boldsymbol{\theta} \longmapsto
\mathcal{L}(u_{\boldsymbol{\theta}},\bfq_{\boldsymbol{\theta}})
:=
\min_{(u,\bfq)\in V_{\boldsymbol{\theta}}^u \times V_{\boldsymbol{\theta}}^{\bfq}}
\mathcal{L}(u,\bfq),
\label{eq:loss_theta}
\end{equation}
which measures the optimal value of the functional attainable within the parametrized approximation spaces. The approximate solution
$(u_{\btheta},\bfq_{\btheta})$ is sought within these spaces by minimizing $\mathcal{L}$ as described in Section~\ref{sec: optimal_coefficients}.
The update of the space $V_{\boldsymbol{\theta}}^u \times V_{\boldsymbol{\theta}}^{\bfq}$ is then formulated as the following optimization problem:
\begin{equation*}
\min_{\boldsymbol{\theta}} \mathcal{J}(\boldsymbol{\theta}),
\label{eq:nested_minimization}
\end{equation*}
which aims to obtain optimal discrete spaces to approximate the solution. 
\begin{remark}
The proposed formulation can be interpreted as a reduced--order modeling
strategy in which the discrete spaces are optimized through the minimization of a
variational residual, with linear dependence on the reduced coefficients and
nonlinear dependence on the space parametrization.
\end{remark}

\begin{remark}
The framework involves two conceptually distinct minimization processes based on the same continuum loss functional. One may generalize these ideas by considering
\[
(u_{\boldsymbol{\theta}}, \bfq_{\boldsymbol{\theta}})
=
\operatorname*{arg\,min}_{(u,\bfq) \in V_{\boldsymbol{\theta}}}
\mathcal{L}_1(u,\bfq),
\]
and the minimization of
\[
\mathcal{J}(\boldsymbol{\theta})
=
\mathcal{L}_2\bigl(u_{\boldsymbol{\theta}}, \bfq_{\boldsymbol{\theta}}\bigr).
\]
Notice that our framework corresponds to the particular choice $\mathcal{L}_1 = \mathcal{L}_2 = \mathcal{L}$.

Alternative definitions of $\mathcal{J}$ could be adopted, for example based on 
a mixture of PDE and data-driven objectives. Similarly,
$\mathcal{L}$ could be approximated using different discretization strategies, such as a finite element method over a mesh dependent on the parameters $\btheta$ (see, e.g., \cite{aballay2025r}). This perspective suggests the possibility of introducing hybrid schemes employing different losses.
\end{remark}

\section{Neural Network-based Parametrization} \label{sec:section3}
We construct a finite-dimensional space
$
V_{\boldsymbol{\theta}}^u \times V_{\boldsymbol{\theta}}^{\bfq} 
$
using a feedforward neural network with $L$ hidden layers, defined recursively as
\[
\boldsymbol{\Phi}_0(\bfx) := \bfx \in \mathbb{R}^d, \qquad
\boldsymbol{\Phi}_l(\bfx) := \sigma_l\!\left(\mathbf{W}_l \boldsymbol{\Phi}_{l-1}(\bfx) + \mathbf{b}_l\right),
\quad 1 \le l \le L,
\]
where  
$\mathbf{W}_l$, $\mathbf{b}_l$ are trainable parameters, $\boldsymbol{\Phi}_l = (\Phi_l^{(1)},\dots,\Phi_l^{(n_l)}) \in \mathbb{R}^{n_l}$ and $\sigma_l\in C^{1,1}_{\mathrm{loc}}(\mathbb{R})$ is an activation function that acts component-wise. This regularity is essential for a consistent numerical evaluation of the gradient of the variational loss functional. In particular, activations with lower regularity, such as ReLU, lead to distributional derivatives with respect to the trainable parameters, introducing Dirac delta-type singularities that hinder a stable and consistent computation of gradients (see \cite[Section 2.2]{taylor2025regularity}).

The output of the last hidden layer,
\[
\boldsymbol{\Phi}_L : \Omega \to \mathbb{R}^{n_L}, \qquad
\boldsymbol{\Phi}_L = (\Phi_L^{(1)}, \dots, \Phi_L^{(n_L)}),
\]
is a vector-valued function, with each component $\Phi_L^{(i)} \in  W^{2,\infty}(\Omega)$. 
These functions provide a $\boldsymbol{\theta}$-dependent parametric spanning set of the form
\[
\boldsymbol{\theta} := \{\mathbf{W}_1, \mathbf{b}_1, \dots, \mathbf{W}_L, \mathbf{b}_L\}.
\]

To enforce homogeneous Dirichlet boundary conditions on $V_{\boldsymbol{\theta}}^u$, we introduce a smooth {cutoff} function 
$g_D \in C^{1,1}(\Omega)$, that vanishes on $\partial \Omega$ and is strictly positive in the interior.  
The spanning set is then defined via a multiplicative ansatz:
\begin{align*}
\varphi_j(x) := g_D(x)\,\Phi_L^{(j)}(x),\qquad 
 \boldsymbol{\tau}_{j +(k-1)n_L}(x) := \Phi_L^{(j)}(x)\,\mathbf{e}_{k},
\quad j=1,\dots,n_L,\; k=1,\dots,d,
\end{align*}
where $\{\mathbf{e}_k\}_{k=1}^d$ is the canonical basis of $\mathbb{R}^d$. This multiplicative ansatz is standard in variational formulations and ensures conformity of the discrete spaces within
$H_0^1(\Omega) \times H(\mathrm{div};\Omega)$ independently of $\boldsymbol{\theta}$
(see e.g., \cite{Berrone2023Dirichlet}).
Figure~\ref{fig: nn} illustrates our particular parametrization of the discrete space and the whole scheme for finding the solution presented  in Section~\ref{sec:discrete}.

\begin{figure}[!ht]
\begin{tikzpicture}[x=1.8cm, y=1.3cm, >=stealth]

\def\ninput{3}
\def\nhiddenA{4}
\def\nhiddenB{4}
\def\nhiddenC{4}
\def\noutput{3}

\foreach \i in {1,...,\ninput}
    \node[circle,draw,fill=gray!20, minimum size=0.6cm] (I\i) at (0,-\i-0.5) {};

 \node at (-0.4,-1-0.5) {$x_1\to$}; 
 \node at (-0.4,-2-0.5) {$x_2\to$}; 
 \node at (0,-2.5-0.5) {$\vdots$};  
 \node at (-0.4,-3-0.5) {$x_d\to$}; 

\foreach \i in {1,...,\nhiddenA}
    \node[circle,draw,fill=gray!20, minimum size=0.6cm] (H1\i) at (1.5,-\i) {};

 \node at (1.5,-3.5) {$\vdots$};   

\foreach \i in {1,...,\nhiddenB}
    \node[circle,draw,fill=gray!20, minimum size=0.6cm] (H2\i) at (2.5,-\i) {};
     \node at (2.5,-3.5) {$\vdots$}; 

\foreach \i in {1,...,\nhiddenC}
    \node[circle,draw,fill=gray!20,minimum size=0.6cm] (H3\i) at (3.5,-\i) {};
     \node at (3.5,-3.5) {$\vdots$}; 

\foreach \i in {1,...,\noutput}
    \node[circle,draw,fill=gray!20,minimum size=0.6cm] (O\i) at (5.,-1.5*\i+0.5) {};
 \node at (5.6,-1*1.5+0.5) {$\leftarrow u\in V_{\btheta}^u$};
 \node at (5.45,-2*1.5+0.5) {\small{$\leftarrow \bfq^{x_1}$}}; 
 \node at (5.45,-2.5*1.5+0.5) {$ \vdots$};
 \node at (5.45,-3*1.5+0.5) {\small$ \leftarrow\bfq^{x_d}$}; 
\draw [thick, decorate,decoration={brace}]
(5.75,-2.3) -- (5.75,-4.2);
\node at (6.1,-3.2) {$\in V_{\btheta}^{\bfq}$};

\draw[dashed, thick, rounded corners, fill=blue!20] 
    (4.4,-1.8) rectangle (4.75,-0.5);
\node at (4.6,-0.7) {${g_D}$};

\foreach \i in {1,...,\ninput}
    \foreach \j in {1,...,\nhiddenA}
        \draw[red,thick,->] (I\i) -- (H1\j);

\foreach \i in {1,...,\nhiddenA}
    \foreach \j in {1,...,\nhiddenB}
        \draw[red,thick,->] (H1\i) -- (H2\j);

\foreach \i in {1,...,\nhiddenB}
    \foreach \j in {1,...,\nhiddenC}
        \draw[red,thick,->] (H2\i) -- (H3\j);

\foreach \i in {1,...,\nhiddenC}
    \foreach \j in {1,...,\noutput}
        \draw[blue,thick,->] (H3\i) -- (O\j);

\node at (1.5,-0.5) {$\boldsymbol{\Phi}_1$};
\node at (2.5,-0.5) {$\boldsymbol{\Phi}_2$};
\node at (3,-0.5) {$...$};
\node at (3.5,-0.5) {$\boldsymbol{\Phi}_{L}$};

\draw[ thick, rounded corners] 
    (-1.,0) rectangle (3.8,-4.5);
\node at (0,-0.5) {Neural network};%
\node at (0,-4.2) {$\btheta$ parameters};%

\draw[ thick, rounded corners] 
    (6,-2.4) rectangle (8.4,-1.3);
\node at (7.2,-2) {$\mathcal{J}(\btheta) :=\displaystyle{\min_{(u,\bfq)\in V_{\btheta}^u\times V_{\btheta}^{\bfq}}} \mathcal{L}(u,\bfq)$};%

\draw[thick, blue] (4,-4.4) --(4,-4.6) --(4.6,-4.6);
\node[blue] at (4.8,-4.6) {LS};
\draw[->,thick, blue] (5,-4.6) -- (7.3,-4.6)--(7.3,-2.4);

\draw[<-,thick, red] (0,-4.5) --(0,-5) --(3.4,-5);
\node[red] at (4,-5) {Update $\btheta$};
\draw[thick, red] (4.5 ,-5) -- (6.4,-5)--(6.4,-2.4);

\draw[thick] (5.8,-0.7) --(5.8,-0.2)--(6.1,-0.2);
\node at (6.3,-0.2) {$C_{\kappa,\boldsymbol{\theta}}^P$};
\draw[thick,->](6.6,-0.2)--(7.,-0.2) -- (7.,-1.3);

\end{tikzpicture}
\caption{Fully connected neural network architecture and update.\label{fig: nn}}
\end{figure}
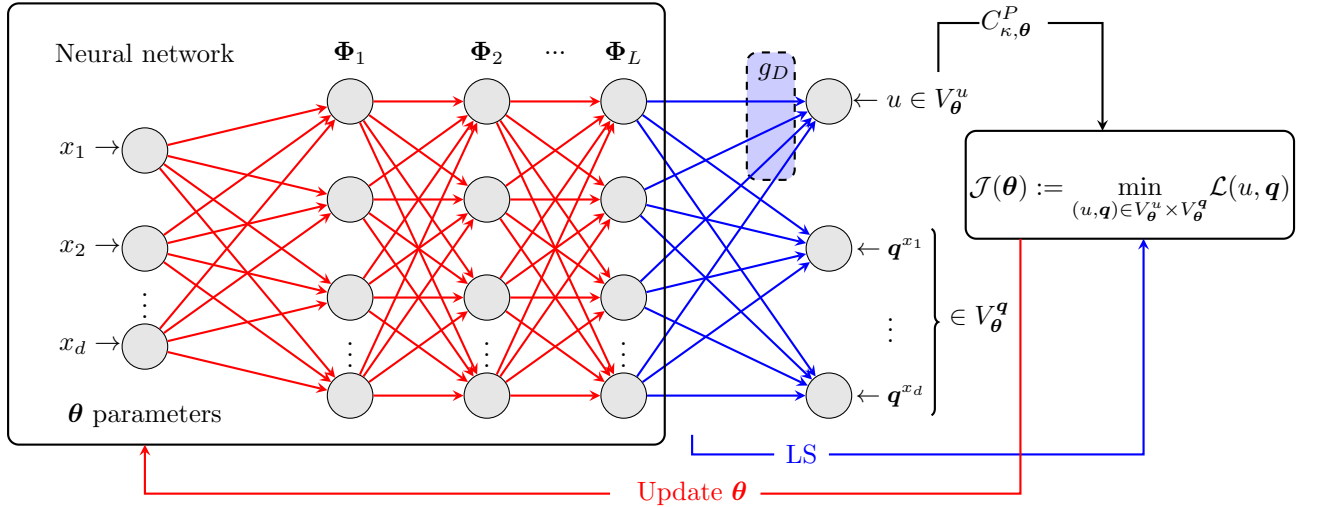

\subsection{Initialization of the discrete space}
\label{sec: bases_initialization}

 This subsection describes the construction of the initial  discrete spaces
$V_{\boldsymbol{\theta}}^u \times V_{\boldsymbol{\theta}}^{\bfq}$.
We focus on the choice of activation function and the resulting functional setting.

\paragraph{Activation function and functional setting}
We select the Rectified Quadratic Unit (ReQU) activation function, defined by
\[
\operatorname{ReQU}(x) := (\max\{0,x\})^2.
\]
Together with the strongly imposed boundary conditions $g_D \in C^{1,1}(\Omega)$, and the selected architecture, this choice yields an ansatz space
\[
V_{\boldsymbol{\theta}}^u \times V_{\boldsymbol{\theta}}^{\bfq}
\subset H_0^1(\Omega) \times H(\Div;\Omega),
\]
which is conforming with the functional setting of the  problem. 
Among admissible $C^{1,1}_{\mathrm{loc}}$ activations, $\mathrm{ReQU}$ is particularly appealing because it has comparatively low regularity while remaining in this class, which helps mitigate the quasi-Gibbs phenomenon.

\paragraph{Geometric structure of the ansatz space}
The initial distribution of the network parameters determines the initial geometric structure of the ansatz space.
Following~\cite{cai2024least}, we characterize this structure in terms of the support of the spanning functions,
defined through \emph{breaking hypersurfaces}.
For the $l$-th hidden layer, these hypersurfaces (breaking points in one dimension and breaking curves in two dimensions)
are defined as
\begin{equation}\label{eq:breaking_surfaces}
\mathcal{S}^{(l)}_j
:=
\left\{
\boldsymbol{x} \in \mathbb{R}^d :
\big(\mathbf{W}_l \boldsymbol{\Phi}_{l-1}(\boldsymbol{x}) + \mathbf{b}_l\big)^{(j)} = 0
\right\},
\qquad j = 1, \ldots, n_l.
\end{equation}
Each hypersurface $\mathcal{S}^{(l)}_j$ separates the input domain into regions where the corresponding
network component is either active or inactive, thereby inducing a piecewise-polynomial structure of the ansatz.

\begin{remark}
When considering more than one hidden layer, the breaking hypersurfaces associated with 
$\mathcal{S}^{(l)}_j$ can be described by piecewise-polynomial equations of degree up to $2l$. 
 In particular, in 2D these hypersurfaces appear as \emph{breaking curves} rather than straight lines. This contrasts with the case of standard ReLU networks, where the activation introduces linear partition boundaries. 
\end{remark}

\paragraph{Initialization in one dimension}
In the one-dimensional case, $\Omega = (0,1)$, the weights and biases of the first hidden layer $\boldsymbol{\Phi}_1$
are chosen as
\[
\mathbf{W}^{(i)}_1 = (-1)^{i+1}, 
\qquad 
\mathbf{b}^{(i)}_1 = (-1)^i \frac{i}{n_1+1},
\qquad i = 1, \ldots, n_1.
\]
This initialization produces $n_1$ equispaced breaking points located strictly inside the domain. Figure~\ref{fig:initialization} illustrates the resulting $\{\Phi_1^{(j)}\}_{j=1}^{10}$ functions when  $n_1=10$,  and their corresponding breaking points. For a single hidden layer ($L=1$), this choice uniquely determines the ansatz spaces $ V_{\boldsymbol{\theta}}^u \times V_{\boldsymbol{\theta}}^{\bfq}$.
\begin{figure}[!ht]
     \centering
 \begin{subfigure}[t]{0.49\textwidth}
 \centering
 \includegraphics[width=\linewidth]{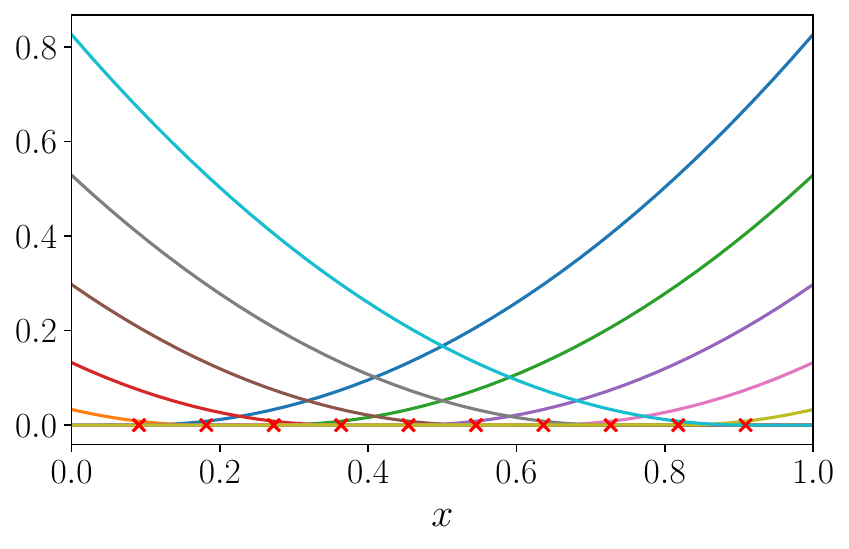}
 \caption{$\{\bfPhi_1^{(j)}\}_{j=1}^{10}$ }
\end{subfigure} 
\begin{subfigure}[t]{0.49\textwidth}
\centering
         \includegraphics[width=\linewidth]{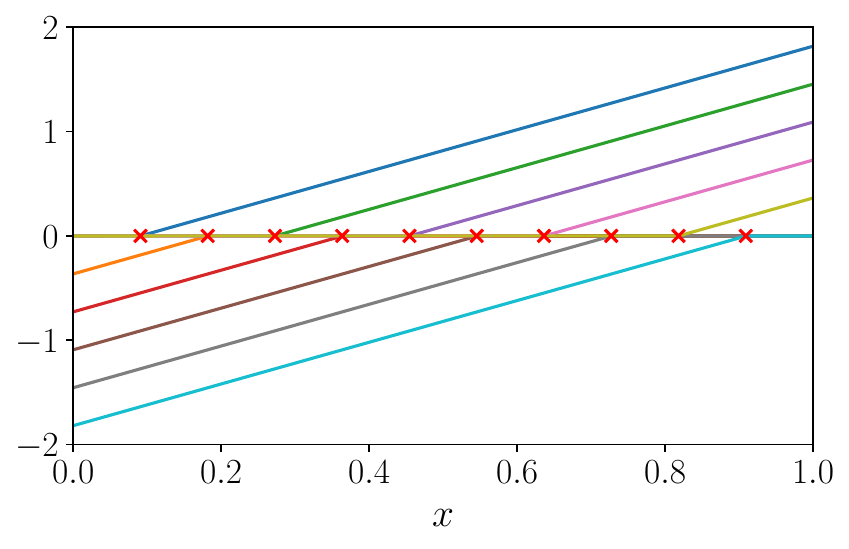}
         \caption{\centering Derivatives of $\{\bfPhi_1^{(j)}\}_{j=1}^{10}$\label{fig: init}}
     \end{subfigure}
     \caption{Initialization of $\{\bfPhi_1^{(j)}\}_{j=1}^{10}$ (left) and the corresponding derivatives (right) in $1$D.  The red markers represent the breaking point of each function. \label{fig:initialization}}
  
\end{figure}

\paragraph{Higher-dimensional initialization}
In higher dimensions, the initialization of the coefficients defining $\boldsymbol{\Phi}_1$
is extended via tensor products.
This construction yields axis-aligned breaking hypersurfaces that are equispaced along each coordinate direction.
Figure~\ref{fig: basis2D} shows an example of the functions generated by the components of $\boldsymbol{\Phi}_1$
when $\Omega = (0,1)^2$, $n_1 = 4$.

\paragraph{Deep networks}
For deep architectures ($L > 1$), the weights of the subsequent hidden layers are initialized
as identity mappings, i.e., with identity weight matrices and zero biases.
This choice preserves the geometric structure induced by the first layer while allowing
the optimization process to adapt higher-order features during training.

\begin{figure}[!ht]
\centering
\begin{subfigure}{0.32\textwidth}
\centering
    \includegraphics[width = \linewidth]{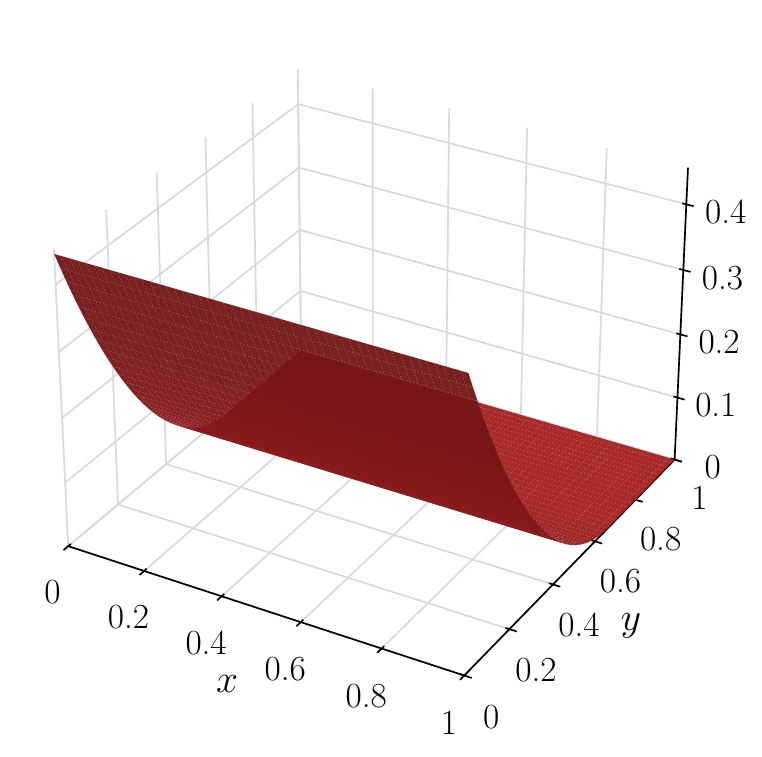}
    \caption{$\bfPhi_1^{(1)}$}
\end{subfigure}
\begin{subfigure}{0.32\textwidth}
\centering
    \includegraphics[width = \linewidth]{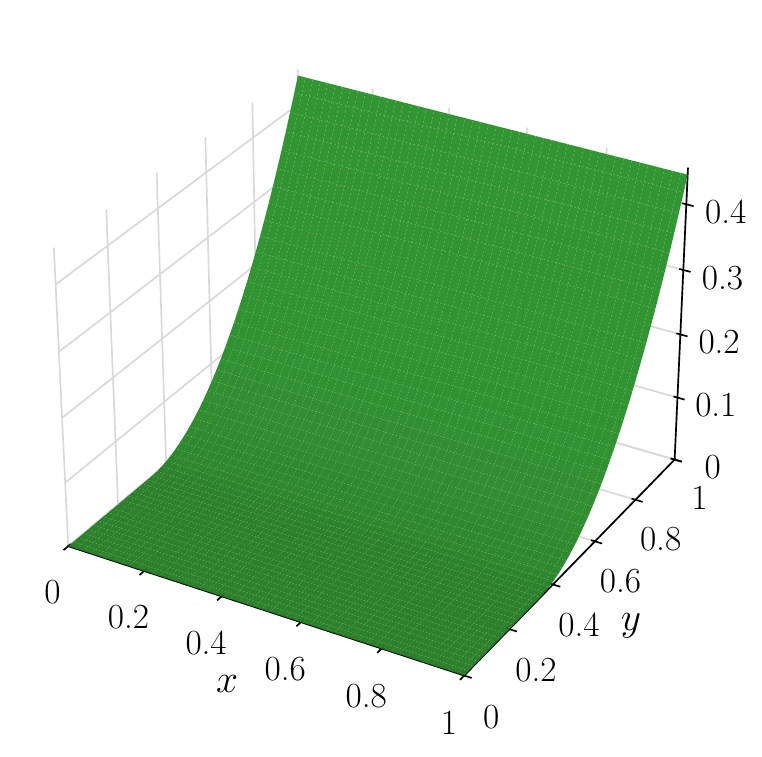}
    \caption{$\bfPhi_1^{(2)}$}
\end{subfigure}
\begin{subfigure}{0.32\textwidth}
\centering
    \includegraphics[width = \linewidth]{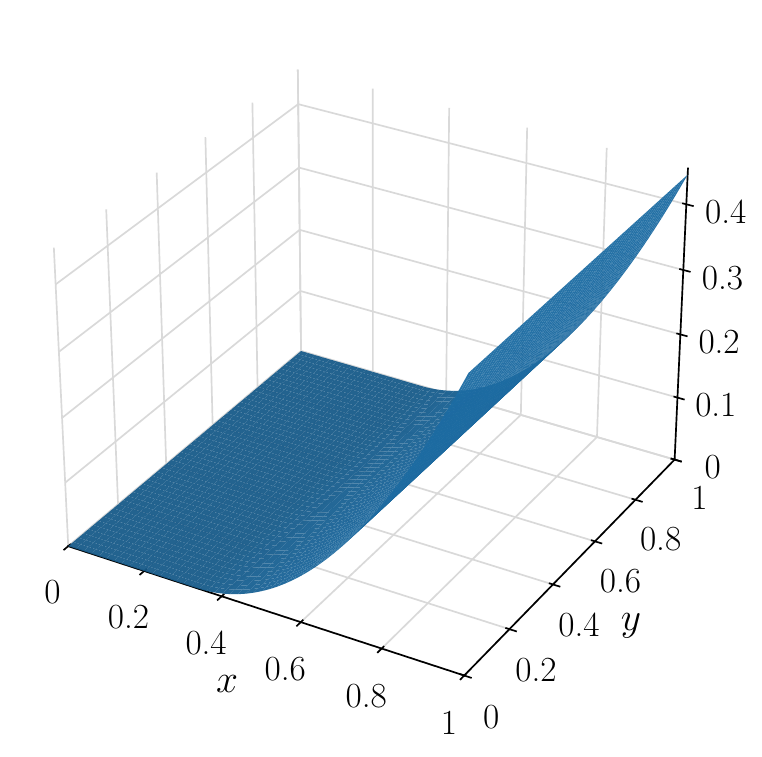}
    \caption{$\bfPhi_1^{(3)}$}
\end{subfigure}\\
\centering
\begin{subfigure}{0.33\textwidth}
\centering
    \includegraphics[width = \linewidth]{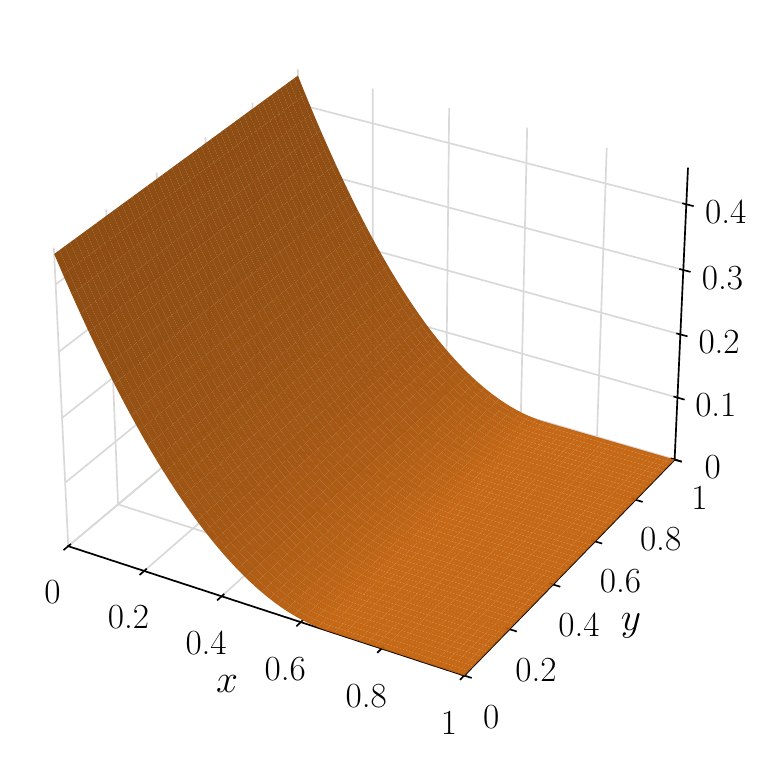}
    \caption{$\bfPhi_1^{(4)}$}
\end{subfigure}
\begin{subfigure}{0.33\textwidth}
\centering
    \includegraphics[width = \linewidth]{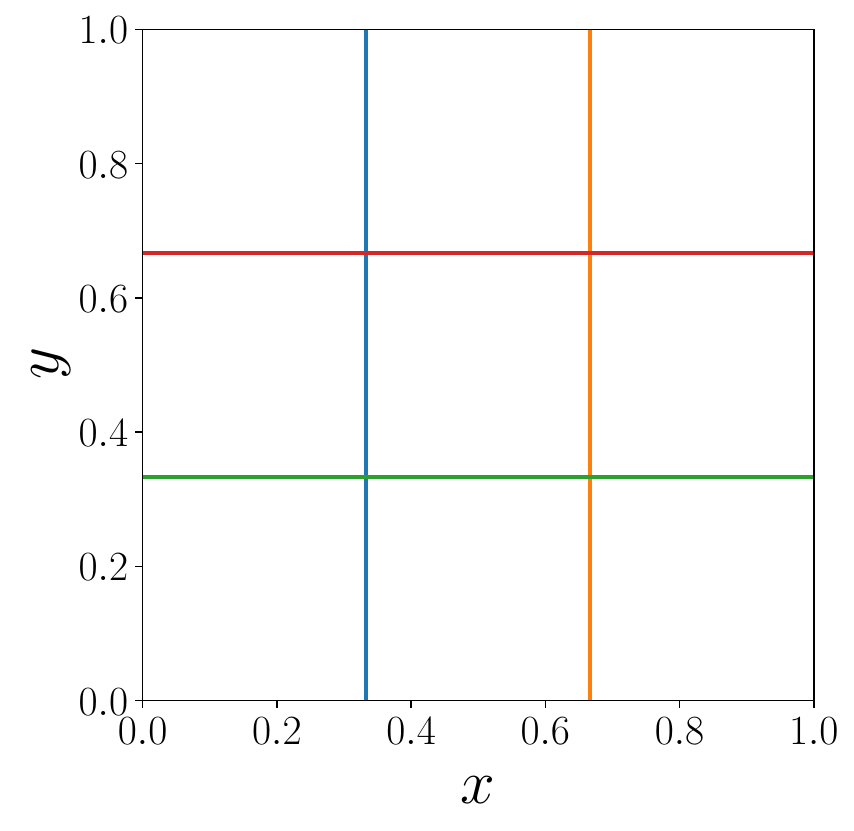}
    \caption{Breaking lines \label{fig:2D_BL}}
\end{subfigure}
\caption{Initialization of $\{\bfPhi_1^{(j)}\}_{j=1}^{8}$ in $2$D. In panel (e), the breaking line color corresponds to the associated component. \label{fig: basis2D}}
\end{figure}

\section{Stochastic Quadrature}\label{sec:integration}
Since for arbitrary parametric spaces $V_{\btheta}^u\times V^{\bfq}_{\btheta}$ the integrals present in the loss \eqref{eq:loss_algebraic_change} and in the energy-norm Poincar\'{e} approximation \eqref{eq:gen eigen} cannot be computed exactly, we resort to a quadrature rule $Q_N$ with $N$ points. The integral of a function 
$\mathcal{I}$ over $\Omega$ is then approximated by
\[
\int_{\Omega} \mathcal{I}
\;\approx\; 
Q_N[\mathcal{I}]
\;:=\;
\sum_{i=1}^N w_i\, \mathcal{I}(x_i),
\]
where $\{(w_i,x_i)\}_{i=1}^N$ are the corresponding quadrature weights and points.

Classical deterministic rules (e.g., Gaussian quadrature) may lead to overfitting in neural network spaces \cite{rivera2022quadrature}. 
To overcome this overfitting, the Monte Carlo (MC) estimator samples $x_i$ uniformly in $\Omega$ with weights $w_i=|\Omega|/N$,
\[
Q_N^{\mathrm{MC}}[\mathcal{I}] 
:= \frac{|\Omega|}{N}\sum_{i=1}^N \mathcal{I}(x_i),
\]
which is unbiased, i.e.,  $\mathbb{E}[Q_N^{\mathrm{MC}}[\mathcal{I}]] = \int_{\Omega} \mathcal{I}$, and satisfies 
\[
\operatorname{Var}\!\left(Q_N^{\mathrm{MC}}[\mathcal{I}]\right)
=
\frac{|\Omega|^2}{N}
\left(
\frac{1}{|\Omega|}\int_\Omega \mathcal{I}^2
-
\left(\frac{1}{|\Omega|}\int_\Omega \mathcal{I}\right)^2
\right),
\]
yielding a root mean-squared error of order $\mathcal{O}(N^{-1/2})$.

As the variance decays slowly with $N$, a large number of integration points may be required to reduce integration errors. We therefore adopt
a first-order stratified strategy. Let $\mathcal{T}_h=\{K\}$ be a partition of $\Omega$ into disjoint hyperrectangles and denote by $F_K:\widehat{K}=(-1,1)^d\to K$ the associated affine maps with Jacobian $|J_{F_K}|$. Using the unbiased first-order rule of \cite{haber1967modified}, the quadrature formula reads: 

For each $K$, let $\widehat{\boldsymbol{x}}_1^{K}$ be uniformly distributed in $\widehat{K}$ and define $\widehat{\boldsymbol{x}}_2^{K}=-\widehat{\boldsymbol{x}}_1^{K}$, with weights $\widehat{w}_1^{K}=\widehat{w}_2^{K}=2^{d-1}$. The global approximation is given by
\begin{equation*}\label{eq:quadrature_loss}
Q_N^{P1}[\mathcal{I}]
:=
\sum_{K\in\mathcal{T}_h}
|J_{F_K}|
\sum_{j=1}^{2}
\widehat{w}_j^K\,
\mathcal{I}\circ F_K(\widehat{\boldsymbol{x}}_j^K). 
\end{equation*}
When employing a uniform partition, the root mean-squared error of this rule behaves as $\mathcal{O}(N^{-1/2-2/d})$ for sufficiently smooth integrands, see \cite{taylor2025stochastic}.

 \subsection{Passive variance reduction}

The functional
\eqref{eq:loss_theta} is discretized as 
 \begin{equation}\label{eq: quadrature_loss}
 \mathcal{J}(\btheta) \approx  Q^{P1}_N[\mathcal{I}_{\btheta}] := \sum_{K \in \mathcal{T}_h}  |J_{F_K}| \sum_{j=1}^{2} \widehat{w}_j^K \left( \|\kappa^{1/2}\bfq_{\btheta}+ \kappa^{-1/2}\nabla u_{\btheta}\|^2_{\mathbb{R}^d} +2(C^P_{\kappa})^2|\Div\bfq_{\btheta} - f|^2\right)\circ F_K(\widehat{\boldsymbol{x}}_j^K),
 \end{equation}
 where $\mathcal{I}_{\btheta}$ is the integrand of the loss function.

The parameters $\btheta$ are updated by means of a gradient-based optimization algorithm.  
The update requires the explicit computation of the gradient of the functional \eqref{eq: loss final}. 
Replacing the exact loss by its quadrature approximation leads to 
\[
\nabla_{\btheta}\mathcal{J} 
\approx 
\nabla_{\btheta} Q_N^{P_1}[\mathcal{I}_{\btheta}].
\]
 The resulting discrete parameter update can be interpreted as a Stochastic Gradient Descent (SGD) scheme.
The following theorem demonstrates that as $(u_{\btheta},\bfq_{\btheta})$ approaches the exact solution, the variance of the gradient approximation naturally decreases, which accelerates the convergence of the optimization strategy\footnote{This is not the case for other deep variational approaches as VPINNs \cite{kharazmi2019variational, URIARTE202576} and its variants (see, e.g. \cite{rojas2024robust}), or Deep Ritz \cite{yu2018deep}.}.

\begin{theorem}[Passive variance reduction of the loss gradient]
\label{thm: passive_variance}
Assume that the neural network parametrizations
$(u_{\btheta},\bfq_{\btheta})$ is sufficiently smooth with respect to the trainable parameters $\btheta$. 
 Furthermore, for each component $\theta$ of the parameter vector $\boldsymbol{\theta}$, assume that there exists a constant $C_{\rm grad}>0$ such that, for almost every
$\boldsymbol{x}\in\Omega$,
\begin{equation}
\label{eq:bounded_grad}
\left\|
\begin{pmatrix}
\kappa^{1/2}(\boldsymbol{x})\dfrac{\partial \bfq_{\btheta}}{\partial \theta}(\boldsymbol{x})
+\kappa^{-1/2}(\boldsymbol{x})\dfrac{\partial \nabla u_{\btheta}}{\partial \theta}(\boldsymbol{x})
\\[1ex]
\sqrt{2}\,C^P_{\kappa}\,\dfrac{\partial \Div \bfq_{\btheta}}{\partial \theta}(\boldsymbol{x})
\end{pmatrix}
\right\|_{\mathbb{R}^{d+1}}
\le C_{\rm grad}.
\end{equation}

Then, for the unbiased quadrature approximation
$Q^{P1}_N[\mathcal{I}_{\btheta}]$ of the loss functional in Eq.~\eqref{eq: quadrature_loss},
the variance of its gradient satisfies
\begin{equation}
\label{eq: passive_variance_bound_improved}
\operatorname{Var}\left[ \frac{\partial}{\partial \theta} Q_N^{P1}[\mathcal{I}_{\btheta}]\right]
\leq
4C_{\rm grad}^2 \, |\Omega|\,
\mathcal{L}(u_{\btheta}, \bfq_{\btheta}),
\end{equation}
where the constant $C_{\rm grad}>0$ depends only on the bound
\eqref{eq:bounded_grad} and is independent of the quadrature rule.
\end{theorem}

\begin{proof}
Introduce the $\mathbb{R}^{d+1}$-valued residual
\[
\mathbf{R}_{\btheta}
:=
\begin{pmatrix}
\kappa^{1/2}\bfq_{\btheta}+\kappa^{-1/2}\nabla u_{\btheta}
\\[0.5ex]
\sqrt{2}\,C^P_{\kappa}\,(\Div\bfq_{\btheta}-f)
\end{pmatrix}.
\]
Since 
\[
\|\mathbf{R}_{\btheta}\|_{\mathbb{R}^{d+1}}^2
=
\|\kappa^{1/2}\bfq_{\btheta}+\kappa^{-1/2}\nabla u_{\btheta}\|_{\mathbb{R}^d}^2
+
2(C^P_{\kappa})^2 |\Div\bfq_{\btheta}-f|^2,
\]
we have
\[
\mathcal{L}(u_{\btheta},\bfq_{\btheta})
=
\int_{\Omega} \|\mathbf{R}_{\btheta}\|_{\mathbb{R}^{d+1}}^2= \int_{\Omega} \mathcal{I}_{\btheta}. 
\]

Differentiating the quadrature-based loss \eqref{eq: quadrature_loss} with respect to a single parameter $\theta$ of $\btheta$, and using the chain rule, we obtain
\begin{equation*}
\begin{aligned}
\frac{\partial}{\partial\theta}Q^{P1}_N[\mathcal{I}_{\btheta}]
&=
2\sum_{K \in \mathcal{T}_h} |J_{F_K}| \sum_{j=1}^{2} \widehat{w}_j^K
\Big\langle
\mathbf{R}_{\btheta},
\frac{\partial \mathbf{R}_{\btheta}}{\partial \theta}
\Big\rangle_{\mathbb{R}^{d+1}}
\circ F_K(\widehat{\boldsymbol{x}}_j^K),
\end{aligned}
\end{equation*}
where
\[
\frac{\partial \mathbf{R}_{\btheta}}{\partial \theta}
=
\begin{pmatrix}
\kappa^{1/2}\dfrac{\partial \bfq_{\btheta}}{\partial\theta}
+\kappa^{-1/2}\dfrac{\partial \nabla u_{\btheta}}{\partial \theta}
\\[1ex]
\sqrt{2}\,C^P_{\kappa}\,\dfrac{\partial\Div\bfq_{\btheta}}{\partial\theta}
\end{pmatrix}.
\]

Using the bound \eqref{eq:bounded_grad}, the positivity of the quadrature weights, and the Cauchy--Schwarz inequality in $\mathbb{R}^{d+1}$, we obtain
\begin{equation}
\label{eq:boundedQM_clarified}
\begin{aligned}
\left|
\frac{\partial}{\partial \theta} Q^{P1}_N[\mathcal{I}_{\btheta}]
\right|
&\le
2C_{\rm grad}
\sum_{K\in\mathcal{T}_h}|J_{F_K}| \sum_{j=1}^{2} \widehat{w}_j^K
\|\mathbf{R}_{\btheta}\circ F_K(\widehat{\boldsymbol{x}}_j^K)\|_{\mathbb{R}^{d+1}}.
\end{aligned}
\end{equation}
Applying now the Cauchy--Schwarz inequality to the quadrature sum gives
\begin{equation*}
\begin{aligned}
\left(
\sum_{K\in\mathcal{T}_h}|J_{F_K}| \sum_{j=1}^{2} \widehat{w}_j^K
\|\mathbf{R}_{\btheta}\circ F_K(\widehat{\boldsymbol{x}}_j^K)\|_{\mathbb{R}^{d+1}}
\right)^2
&\le \left(
\sum_{K\in\mathcal{T}_h}|J_{F_K}| \sum_{j=1}^{2} \widehat{w}_j^K
\right)\\
& \times\left(
\sum_{K\in\mathcal{T}_h}|J_{F_K}| \sum_{j=1}^{2} \widehat{w}_j^K
\|\mathbf{R}_{\btheta}\circ F_K(\widehat{\boldsymbol{x}}_j^K)\|_{\mathbb{R}^{d+1}}^2
\right).
\end{aligned}
\end{equation*}
Since the quadrature is unbiased and exactly integrates constants, the first factor equals $|\Omega|$. Therefore, from \eqref{eq:boundedQM_clarified},
\begin{equation*}
\begin{aligned}
\left|
\frac{\partial}{\partial \theta} Q^{P1}_N[\mathcal{I}_{\btheta}]
\right|^2
&\le
4C_{\rm grad}^2 |\Omega|
\sum_{K\in\mathcal{T}_h}|J_{F_K}| \sum_{j=1}^{2} \widehat{w}_j^K
\|\mathbf{R}_{\btheta}\circ F_K(\widehat{\boldsymbol{x}}_j^K)\|_{\mathbb{R}^{d+1}}^2.
\end{aligned}
\end{equation*}

Taking expectations and using again the unbiasedness of the quadrature rule, we obtain
\begin{equation*}
\begin{aligned}
\mathbb{E}\!\left[
\left|
\frac{\partial}{\partial \theta} Q^{P1}_N[\mathcal{I}_{\btheta}]
\right|^2
\right]
&\le
4C_{\rm grad}^2 |\Omega|
\,
\mathbb{E}\!\left[
\sum_{K\in\mathcal{T}_h}|J_{F_K}| \sum_{j=1}^{2} \widehat{w}_j^K
\|\mathbf{R}_{\btheta}\circ F_K(\widehat{\boldsymbol{x}}_j^K)\|_{\mathbb{R}^{d+1}}^2
\right]
\\
&=
4C_{\rm grad}^2 |\Omega|
\int_{\Omega}\|\mathbf{R}_{\btheta}(\boldsymbol{x})\|_{\mathbb{R}^{d+1}}^2\,d\boldsymbol{x}
\\
&=
4C_{\rm grad}^2 |\Omega|\,\mathcal{L}(u_{\btheta},\bfq_{\btheta}).
\end{aligned}
\end{equation*}
Since
\[
\operatorname{Var}\left[ \frac{\partial}{\partial \theta} Q_N^{P1}[\mathcal{I}_{\btheta}]\right]
\le
\mathbb{E}\!\left[
\left|
\frac{\partial}{\partial \theta} Q_N^{P1}[\mathcal{I}_{\btheta}]
\right|^2
\right],
\]
the result follows.
\end{proof}

\section{Numerical Experiments}\label{sec:section6}

\subsection{Numerical procedure and setting}
\label{sec:implementation}

We initialize the space $V_{\btheta}^u\times V_{\btheta}^{\bfq}$ as detailed in Section~\ref{sec: bases_initialization}. 
We then solve the resulting finite-dimensional least-squares problem over
$V_{\boldsymbol{\theta}}^u \times V_{\boldsymbol{\theta}}^{\bfq}$, stated in Section \ref{sec: optimal_coefficients}. Subsequently, we update the network parameters using the Adam optimizer \cite{kingma2015adam}.
If the energy-norm Poincar\'{e} constant $C_{\kappa}^P$ is unknown, we estimate it via the eigenvalue problem in Section~\ref{sec: Poincare} at predetermined iterations, approximating the integrals with the same quadrature rule. Since this constant serves as an upper-bound constant, we keep the largest value between the previous and the new estimate. Algorithm~\ref{alg:training} summarizes the overall training procedure. 

\begin{algorithm}[H]
\caption{Training and update of the parametric trial spaces}
\label{alg:training}
\begin{algorithmic}[1]
\Require Number of iterations $K$, domain $\Omega$.
\State Initialize nonlinear parameters $\boldsymbol{\theta}^{(0)}$
\For{$k = 0,1,\dots,K-1$}
    \State Construct discrete spaces
    $V_{\boldsymbol{\theta}^{(k)}}^u \times
    V_{\boldsymbol{\theta}^{(k)}}^{\bfq}$
    \If{Poincaré update step}
        \State Update the energy-norm Poincaré constant $C_{\kappa,\btheta}^P$
    \EndIf
     \State Solve the reduced least-squares problem in
    $V_{\boldsymbol{\theta}^{(k)}}^u \times
    V_{\boldsymbol{\theta}^{(k)}}^{\bfq}$
    \State Update parameters $\boldsymbol{\theta}^{(k+1)}$ with Adam
\EndFor
\State Solve the reduced least-squares problem in
    $V_{\boldsymbol{\theta}^{(K)}}^u \times
    V_{\boldsymbol{\theta}^{(K)}}^{\bfq}$
\State \Return $(u_{\text{NN}}, \bfq_{\text{NN}}): = (u_{\boldsymbol{\theta}^{(K)}},\bfq_{\boldsymbol{\theta}^{(K)}})$
\end{algorithmic}
\end{algorithm}

Following the computation of the optimal solution at Step 7, we assess accuracy at each iteration by evaluating the error in the $H_{0,\kappa}^1 \times H(\Div, \kappa)$ norm and its individual components.  We compute all error norms using a fine trapezoidal rule on an equispaced grid.
\subsection{One-dimensional experiments}\label{sec:1D_ex}
We consider the domain $\Omega = (0,1)$. We use a shallow neural network ($L=1$) with ReQU activation and $n_1=16$ spanning functions, whose hidden layer output $\bfPhi_1$ induces a piecewise-quadratic representation in 1D.  The homogeneous Dirichlet boundary condition is enforced via the lifting $g_D = x(1-x)$.
The least-squares solver utilizes a Tikhonov parameter $\mu = 10^{-12}$. When needed, we estimate the energy-norm Poincar\'{e} constant every 100 iterations using regularization parameters $\alpha_1 = 10^{-8}$ and $\alpha_2 = 10^{-10}$.  We also take $\varepsilon = 10^{-15}$ in Eq. \eqref{eq:diag}.

\subsubsection{Effect of Stochastic Integration}\label{sec: num_integration}
We investigate how the number of quadrature points $N$ affects training convergence and stability by comparing the proposed FOSLS formulation against the Deep Ritz (DR) method \cite{yu2018deep}:
\[
\mathcal{L}_{\text{Ritz}}(u) = \int_{\Omega} \left(\frac{1}{2}\kappa |\nabla u|^2 -fu\right).
\]
We consider a smooth diffusion problem with $\kappa=1$ and exact solution $u^{*}=\sin(2\pi x)$. Neural networks are trained using  a learning rate of $10^{-4}$. To visualize the evolution of the loss and the errors, we employ a log-binning approach\footnote{We partition the iterations into 40 logarithmically equispaced bins, compute the mean and standard deviation within each bin, and display the results with error bars.}.
\begin{figure}[!ht]
     \begin{subfigure}[t]{0.49\linewidth}
        \includegraphics[width = \linewidth]{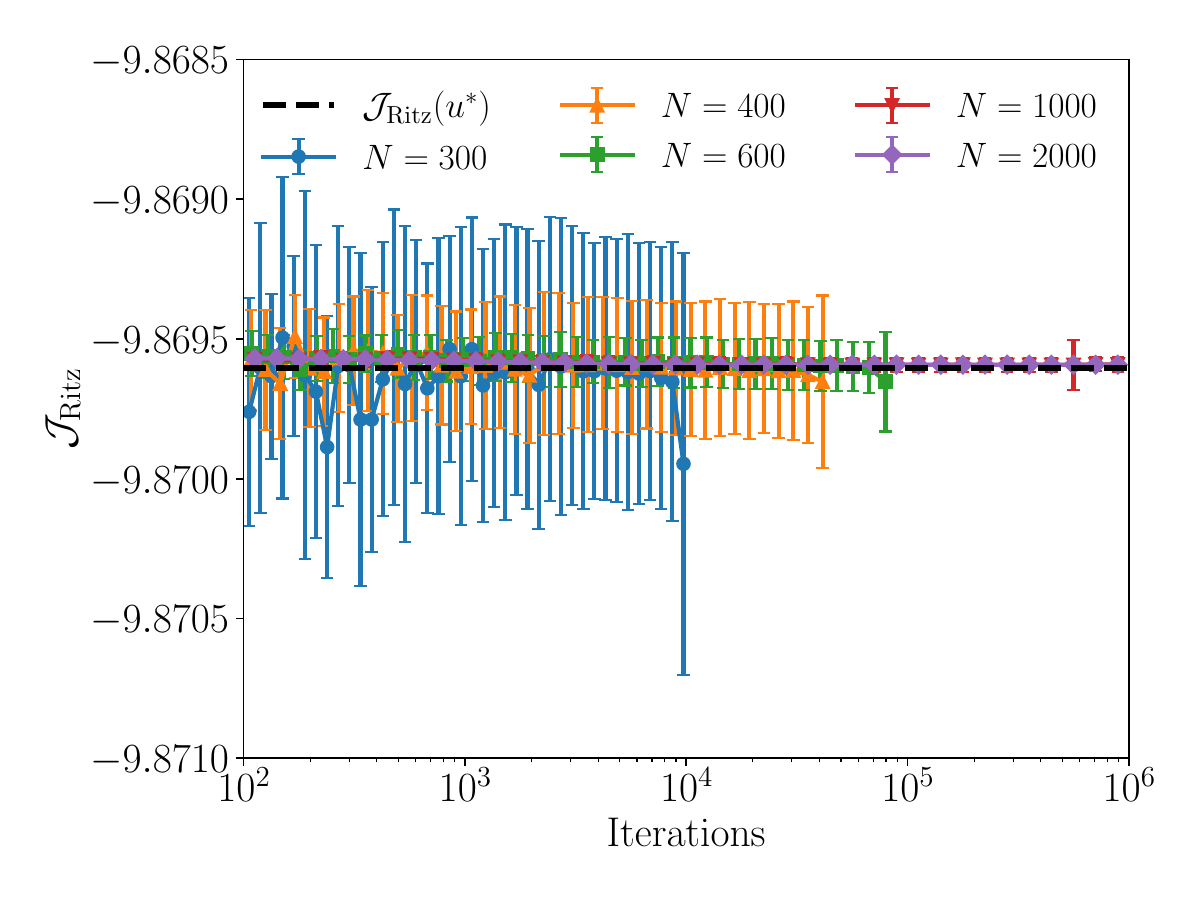}
        \caption{Deep Ritz functional\label{fig:DR_loss}}
    \end{subfigure}
     \begin{subfigure}[t]{0.49\linewidth}
        \includegraphics[width = \linewidth]{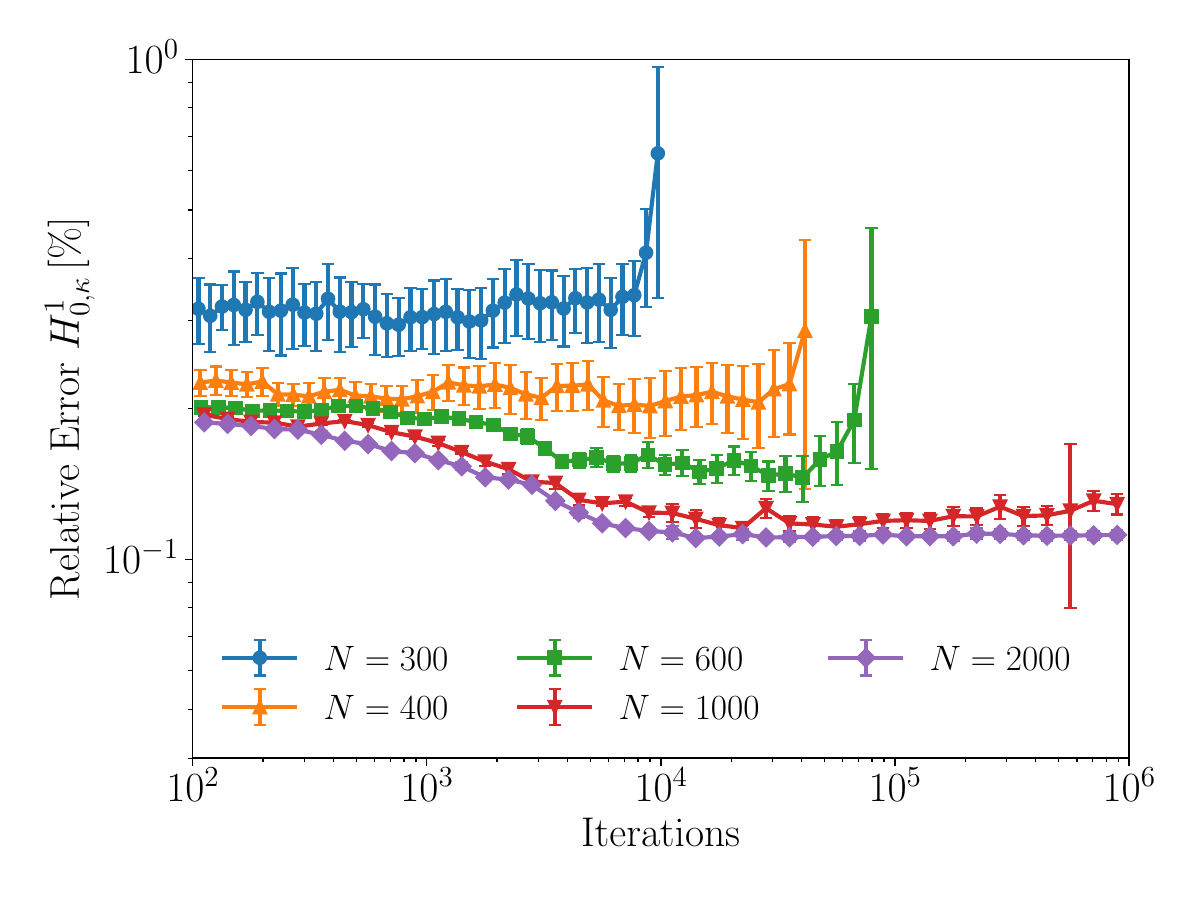}
        \caption{Relative error in $H_{0,\kappa}^1$ \label{fig:DR_error}}
    \end{subfigure}
    \caption{Loss and $H_{0,\kappa}^1$ relative errors  obtained from the DR functional on a smooth problem for varying number of integration points. Notice the first 100 iterations are not shown. In panel \ref{fig:DR_loss}, $\mathcal{J}_{\rm Ritz}$  denotes the functional \eqref{eq:loss_theta} with respect to $\mathcal{L}_{\rm Ritz}$.\label{fig: integration_Ritz}}
\end{figure}

 For the DR method, we consider $N \in \{300, 400, 600, 1000, 2000\}$ quadrature points. Figure~\ref{fig: integration_Ritz} shows the results. Since the DR formulation fails to satisfy the Passive Variance Reduction property, the loss exhibits overflow errors (see Figure \ref{fig:DR_loss}). This behavior is associated with a high variance in the approximation of the gradient. As a consequence, the training becomes unstable unless a sufficiently large number of integration points are used (see Figure~\ref{fig:DR_error}).

\begin{figure}[!t]
    \begin{subfigure}[t]{0.49\linewidth}
        \includegraphics[width = \linewidth]{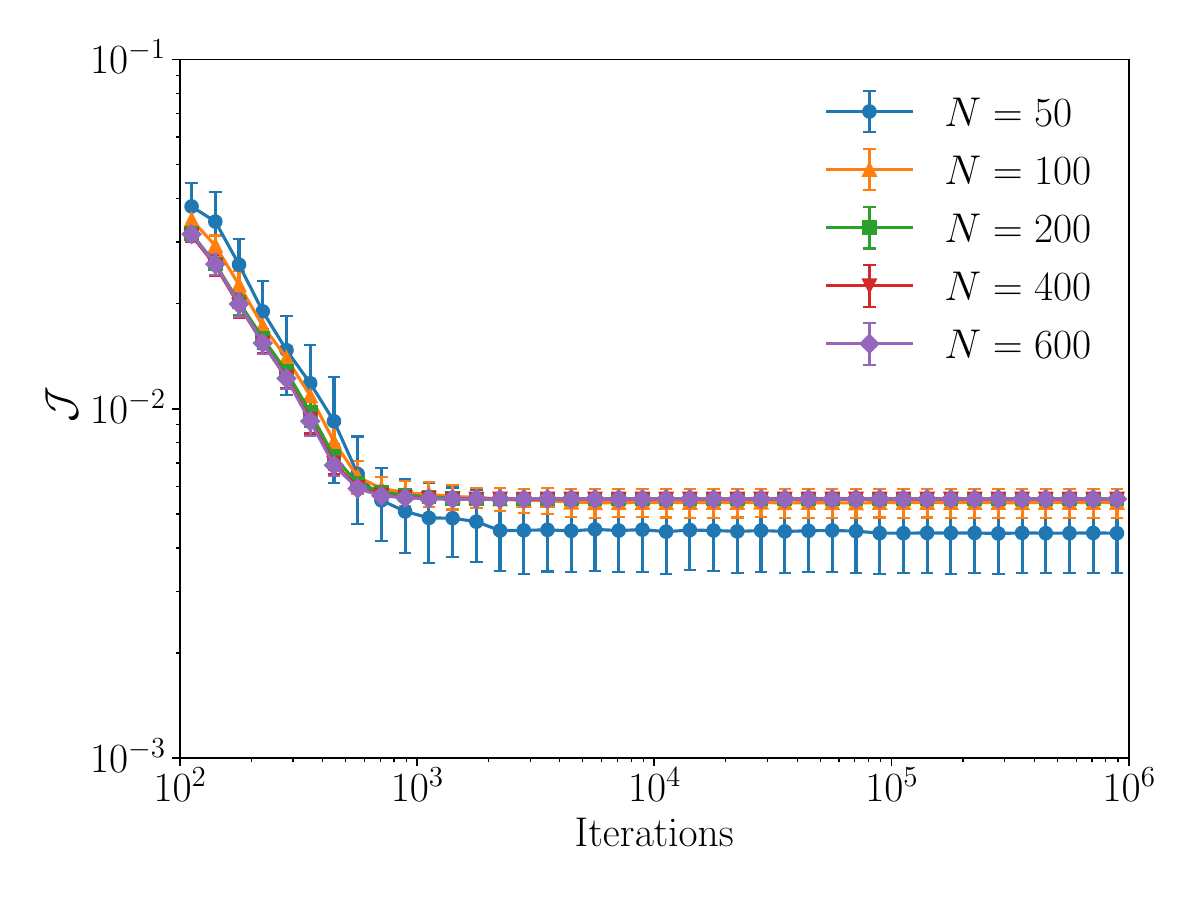}
        \caption{FOSLS functional}
    \end{subfigure}
    \begin{subfigure}[t]{0.49\linewidth}
        \includegraphics[width = \linewidth]{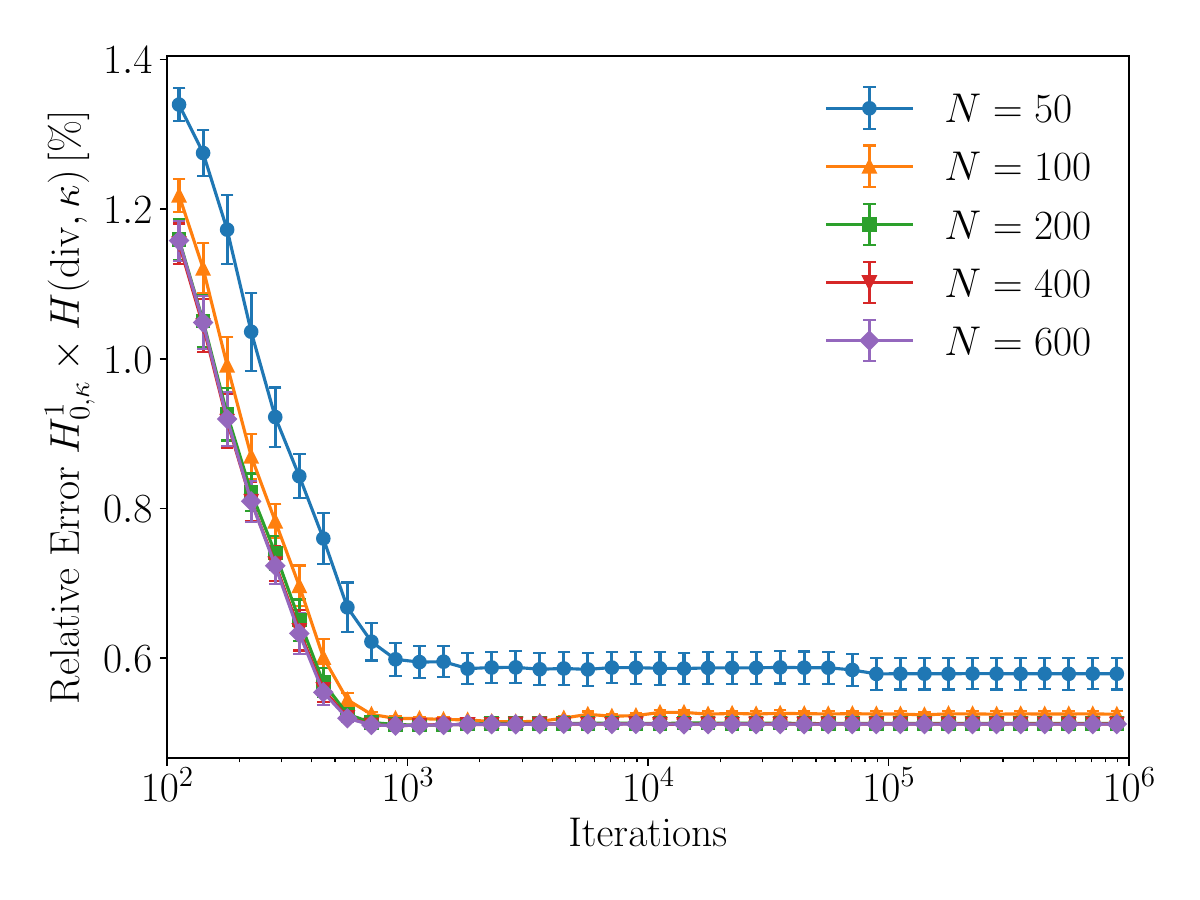}
        \caption{Relative error in energy norm $H_{0,\kappa}^1 \times H(\Div,\kappa)$}
    \end{subfigure}\\
    \caption{Loss and relative errors obtained from the  FOSLS functional on a smooth problem for varying number of integration points. Notice the first 100 iterations are not shown. \label{fig: integration_FOSLS} }
\end{figure}
 For FOSLS, we test the method under coarser rules: $N \in \{50, 100, 200, 400, 600\}$. Figure~\ref{fig: integration_FOSLS} shows the results. Despite employing significantly fewer quadrature points than with the DR method, the training remains stable.  
 This behavior underscores the importance of defining a loss functional that satisfies the Passive Variance Reduction property. In higher-dimensional settings or more challenging problems to integrate, formulations possessing this property can achieve stable training while maintaining a reasonably small number of quadrature points.

\subsubsection{Robust functional vs Standard (non-robust) functional}
In this and the following examples, we consider the following setup: a diffusion problem with a  discontinuous coefficient and a  smooth source term:
\begin{equation}\label{eq: interface_problem}
\kappa=
\begin{cases}
\kappa_0, & x\in \bigl(0,\tfrac12\bigr),\\[4pt]
1, & x\in \bigl(\tfrac12,1\bigr),
\end{cases}
\qquad \text{ and} \qquad 
f(x)=4\pi^2\sin(2\pi x),
\end{equation}
where $\kappa_0>0$. For this family of problems, the exact solution is given by:
\[
u^*=
\begin{cases}
\dfrac{1}{\kappa_0}\,\sin(2\pi x), & x\in \bigl[0,\tfrac12\bigr],\\[8pt]
\sin(2\pi x), & x\in \bigl[\tfrac12,1\bigr],
\end{cases} \qquad \text{ and }\qquad 
\bfq^*=-2\pi\cos(2\pi x).
\]
The exact solution satisfies $u^* \in H^{1}(\Omega)\backslash H^2(\Omega)$ and $\bfq^* \in H^2(\Omega)$.  We use $2\, 000$ points for the stochastic integration.
 
We compare the proposed robust norm--loss pair \eqref{eq: final_norm}--\eqref{eq: loss final} with the standard FOSLS formulation used in the literature \cite{cai2020deep}, which employs the unweighted $H_0^1 \times H(\Div)$ norm \eqref{eq:standard_norm} and the loss \eqref{eq: loss} with $C_{\mathcal{L}} = 1$. We evaluate problem \eqref{eq: interface_problem} across extreme contrast regimes, $\kappa_0 \in \{10^{-6},\,10^{-3},\,10^{3},\,10^{6}\}$. For each case, we train the neural network using the robust loss for $2\,500$ iterations with a learning rate of $10^{-4}$. At each iteration, we evaluate both loss formulations alongside the corresponding energy norms. 

\begin{figure}[!ht]
    \centering
    \begin{subfigure}{0.49\linewidth}
        \includegraphics[width = \textwidth]{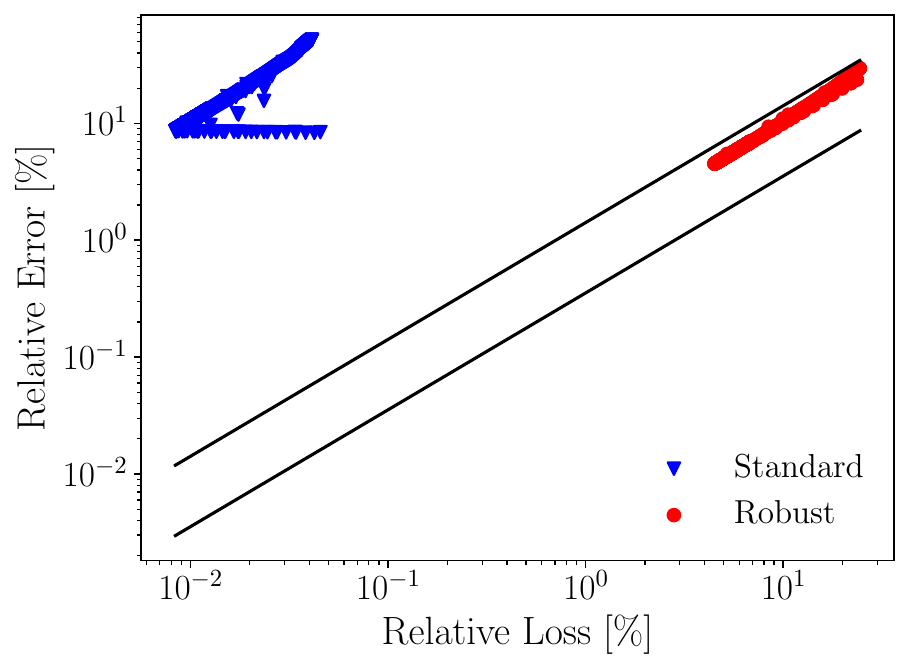}
        \caption{$\kappa_0 = 10^{-6}$}
    \end{subfigure}\hfill
    \begin{subfigure}{0.49\linewidth}
        \includegraphics[width = \textwidth]{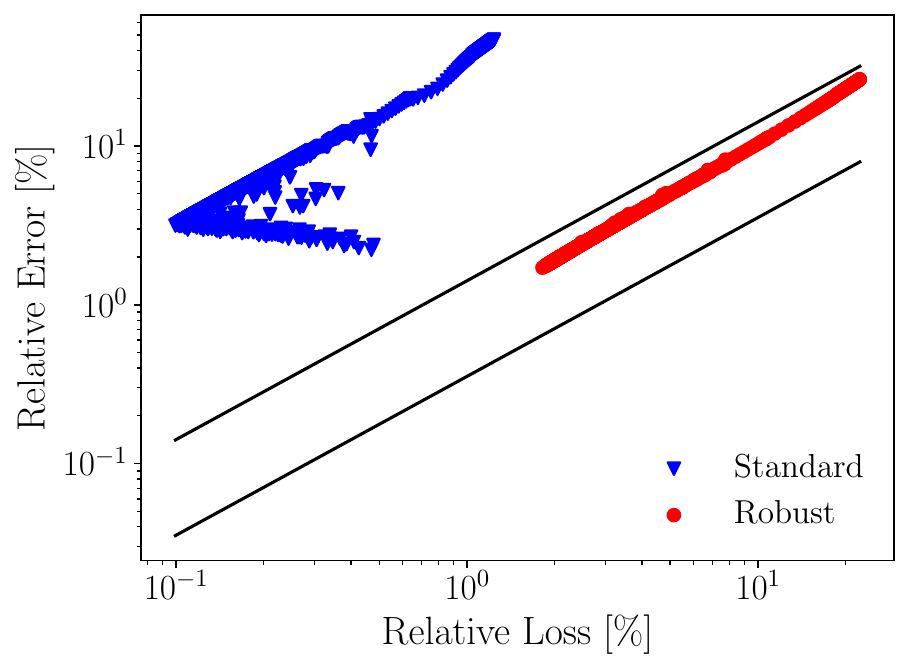}
        \caption{$\kappa_0 = 10^{-3}$}
    \end{subfigure}\\
    \begin{subfigure}{0.49\linewidth}
        \includegraphics[width = \textwidth]{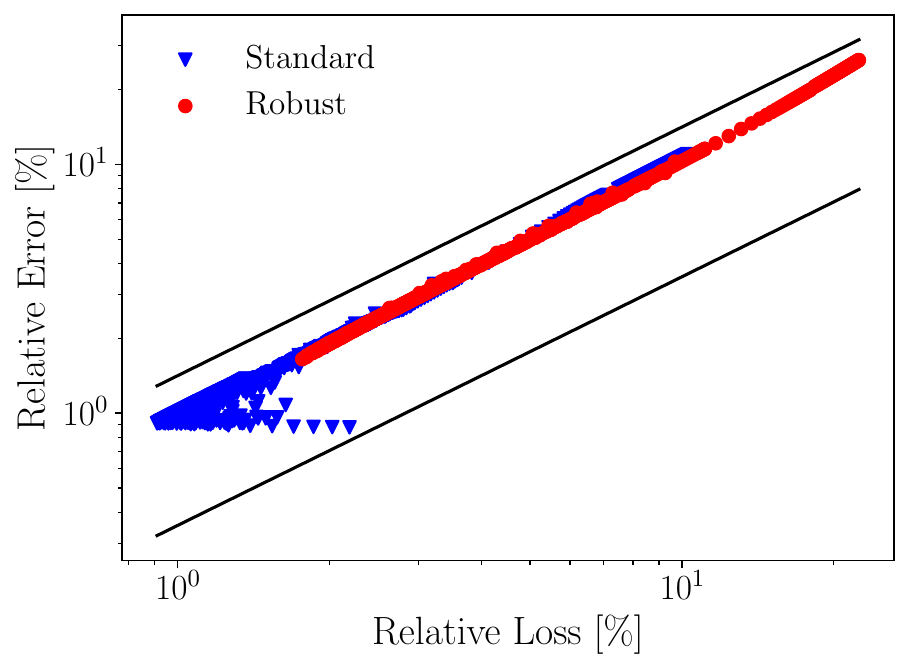}
        \caption{$\kappa_0 = 10^{3}$}
    \end{subfigure}\hfill
    \begin{subfigure}{0.49\linewidth}
       \includegraphics[width = \textwidth]{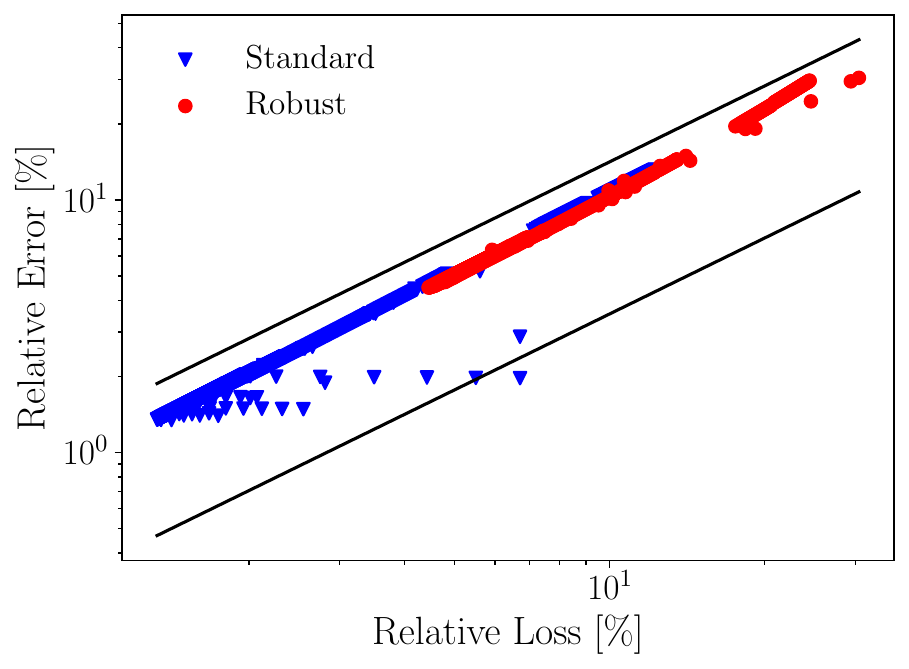}
        \caption{$\kappa_0 = 10^{6}$}
    \end{subfigure}
    \caption{Comparison between robust and non-robust FOSLS formulation. The relative loss is defined as the square root of the loss over the error norm. The loss is evaluated with the same fine integration rule as the error. The black lines correspond to the robustness constant bounds derived in Section \ref{sec: FOSLS}. \label{fig: robustness}}
\end{figure}

 Figure~\ref{fig: robustness} illustrates the correlation between the loss and the error norm for all tested values of $\kappa_0$. For  the standard formulation, we observe a strong $\kappa_0$-dependent relationship between loss and error, which severely deteriorates the correlation. Conversely, with our proposed approach, not only do we obtain a robust error estimator as dictated by the theory, but we indeed observe an almost linear dependence between the loss and the norm of the error. This confirms the theoretical analysis in Section~\ref{sec: FOSLS}: the proposed loss acts as a robust error estimator.  
The approximation of the energy-norm Poincar\'{e} constant is displayed in Figure \ref{fig: Poincare}, showing good agreement with the reference values across all cases.

\begin{figure}[!ht]
    \centering \includegraphics[width = 0.5\linewidth]{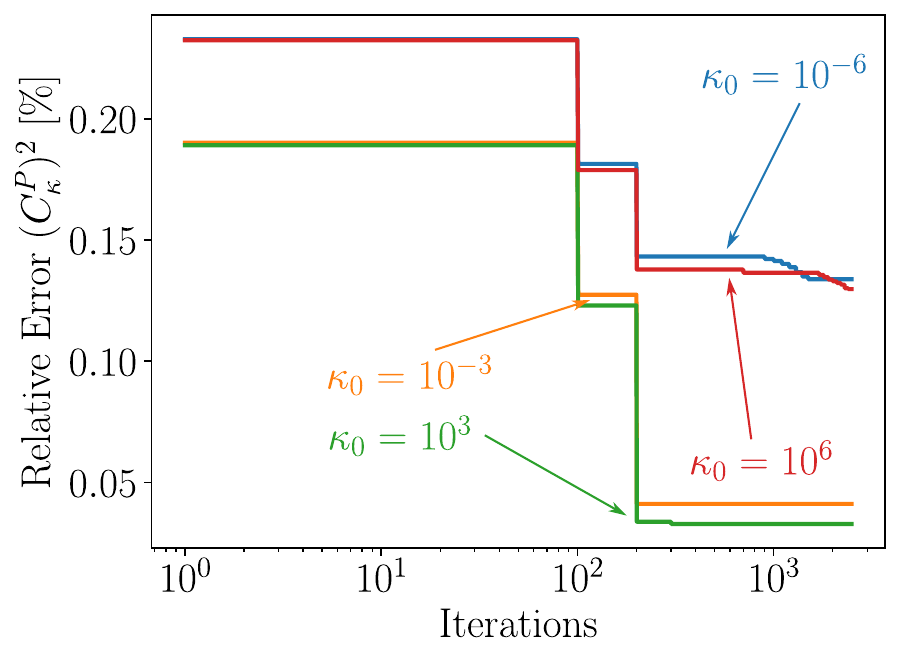}
    \caption{Relative errors in the energy-norm Poincar\'{e} constant. The reference value is computed by solving numerically the transcendental equation in \ref{Appendix-poincare}. \label{fig: Poincare}}
\end{figure}

\subsubsection{Quasi-Gibbs phenomenon}
In \cite{cai2024least}, the authors provide a constructive proof showing that ReLU networks ($\sigma(x)=\mathrm{ReLU}(x)$) can approximate characteristic functions in $\mathbb{R}^d$ over sufficiently regular domains without exhibiting Gibbs-type oscillations. This construction  assumes that the discontinuity set (location and geometry) is known \emph{a priori}. However, herein we allow the neural network to \emph{discover} the discontinuity location.  While the network can approximate this location, we often observe a quasi-Gibbs phenomenon due to the misalignment between the derivative discontinuity and its neural network approximation, an effect that is magnified when using a smooth activation function $\sigma$. This is particularly relevant because such oscillations enter directly into the loss functional, thereby requiring a more careful integration to control the variance.

To see this quasi-Gibbs phenomenon, we compare  our $\operatorname{ReQU}$-based discrete space versus a smooth neural network using $\sigma(x)=\tanh(mx)\in C^\infty, \, m>0,$ in the problem \eqref{eq: interface_problem} with $\kappa_0 = 3$. For the former space, we consider two cases: (a) the network defined in Section \ref{sec:1D_ex} and (b) a neural network  with $L=2$ hidden layers and $n_1=n_2= 16$. For the latter, we consider $m \in \btheta$, i.e., a parameter to be optimized, initialized with $m=50$.  Figure~\ref{fig: gradient_error} shows the results. The smooth network exhibits a noticeable quasi-Gibbs phenomenon localized near the interface. This behavior was observed throughout all iterations, a pattern also observed with other $\sigma\in C^\infty$ functions. In contrast, the $\operatorname{ReQU}$ network exhibits only mild oscillations beyond those inherent to its piecewise structure. Even with more hidden layers, after some iterations, the quasi-Gibbs phenomenon is dampened considerably. To quantify this phenomenon near the interface, we measure the Total Variation (TV) of the gradient error,  $\|\nabla(\nabla(u^*-u_{\mathrm{NN}}))\|_{L^1}$ restricted to the interval $[0.4, 0.6]$ in the final iteration. The neural network with $\operatorname{ReQU}$ activation and $L=1$ layer has a TV of  $ 4.76$,  whereas the network with $L=2$ layers has a TV of $4.67$. The $\tanh$ network with $m \in \btheta$ has a TV of $9.29$.  A simple  calculation helps us interpret the TV--error. Since we are measuring the total variation of the gradient error, we model the jump in the gradient by considering:
\[
g_\varepsilon(x)=
\begin{cases}
0, & x<-\varepsilon,\\[3pt]
\dfrac{A}{2\varepsilon}(x+\varepsilon), & -\varepsilon \le x \le \varepsilon,\\[6pt]
A, & x>\varepsilon,
\end{cases}
\]
which represents an \emph{ideal} piecewise linear approximation of a step function with jump $A$. The total variation of this approximation is
\[
\int_{\mathbb{R}} | \nabla g_\varepsilon|
=
\int_{-\varepsilon}^{\varepsilon} \frac{A}{2\varepsilon}\,dx
=
A.
\]
Thus, even an ideal non-oscillatory approximation contributes a nonzero amount $A$ to the TV norm.  In our case, the jump in the gradient is $2\pi \left(1-1/3\right)\approx 4.188$,  which we can interpret as an excess error, with the excess error in the tanh case being around nine times larger than the ReQU implementation. Moreover, since the exact gradient is discontinuous whereas the approximation has a continuous gradient, the $L^\infty-$error is expected to be at least of order $A/2$. Hence, the errors observed near the discontinuity should not be interpreted as unexpectedly large.

\begin{figure}[!ht]
\begin{subfigure}{0.49\linewidth}
        \includegraphics[width = \linewidth]{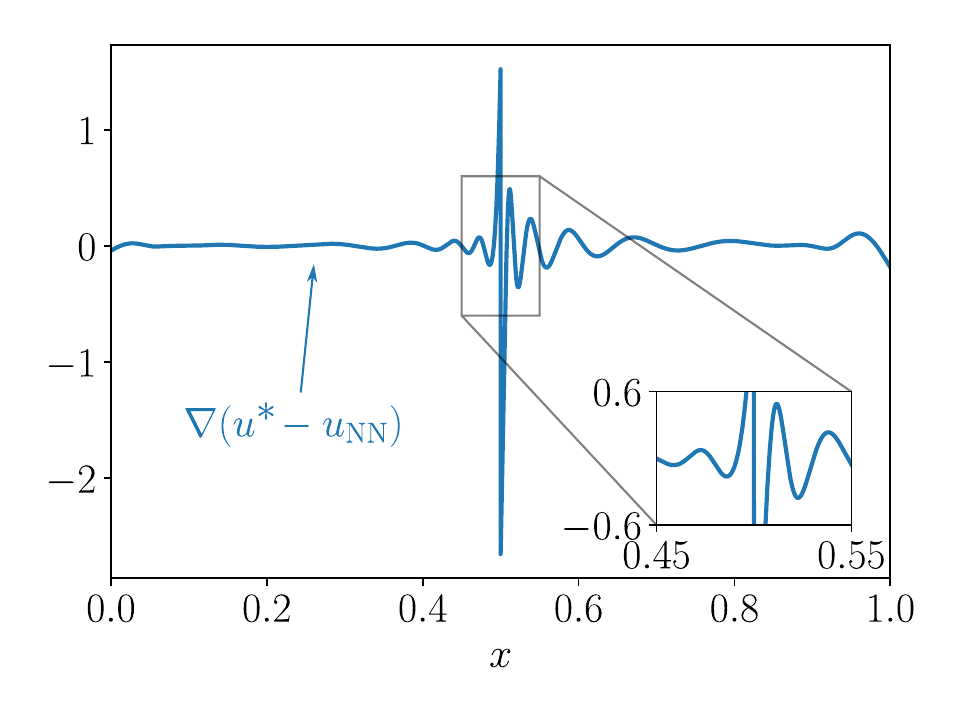}
        \caption{$\sigma =\tanh(mx)$, $L=1$, and $10\,000$ iterations}
    \end{subfigure}
    \begin{subfigure}{0.49\linewidth}
        \includegraphics[width = \linewidth]{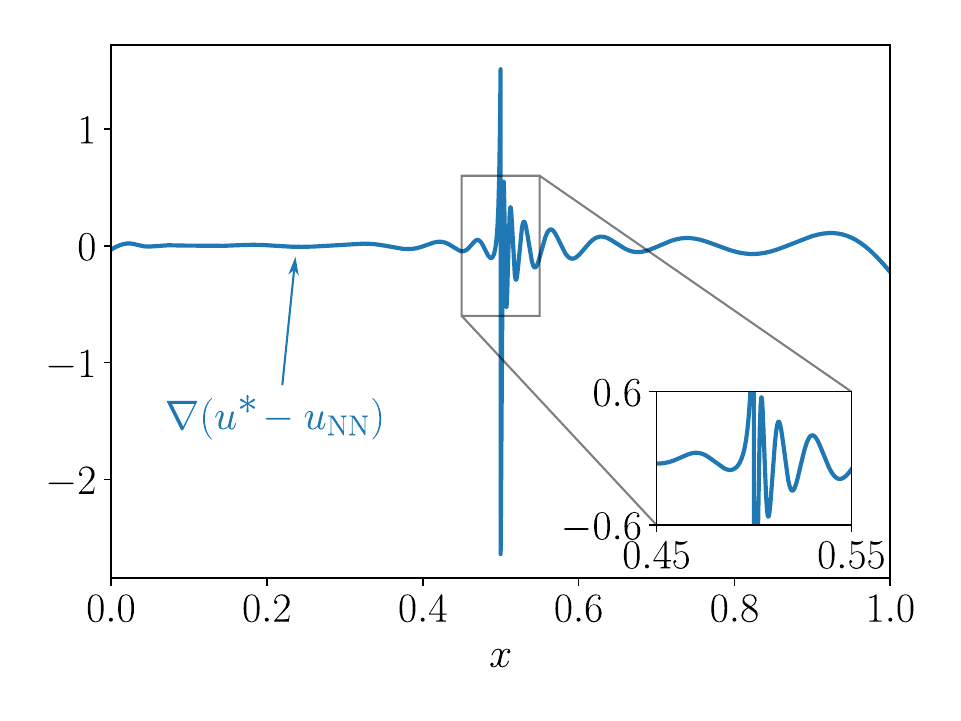}
        \caption{$\sigma =\tanh(mx)$, $L=1$, and $50\,000$ iterations}
    \end{subfigure}\\
\begin{subfigure}{0.49\linewidth}
        \includegraphics[width = \linewidth]{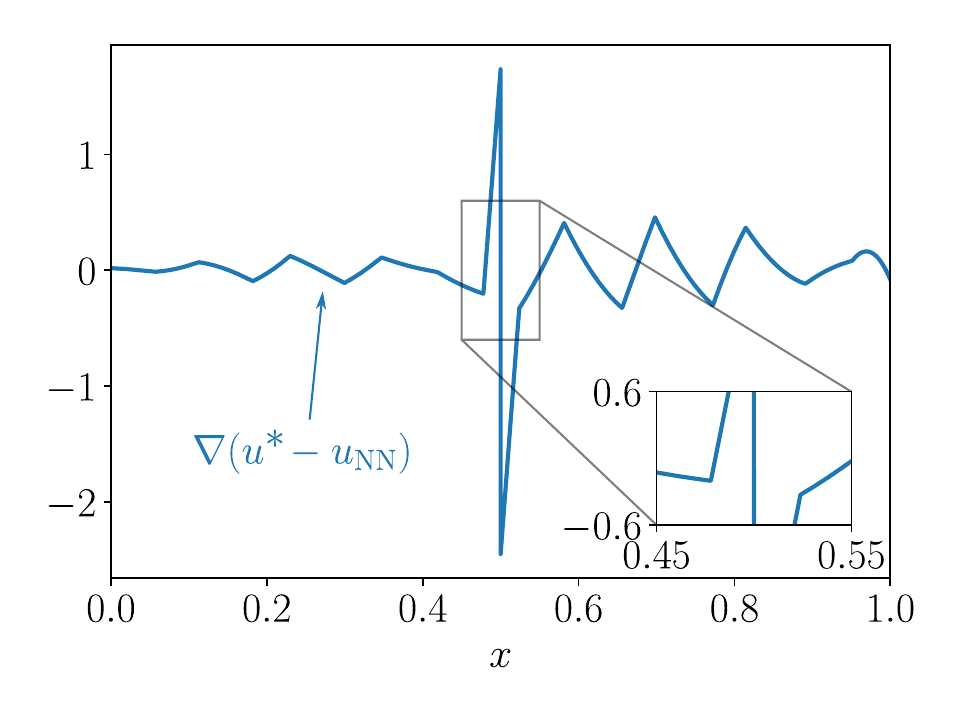}
        \caption{$\sigma =\operatorname{ReQU}$, $L=1$ layers, and 50 iterations}
    \end{subfigure}
    \begin{subfigure}{0.49\linewidth}
        \includegraphics[width = \linewidth]{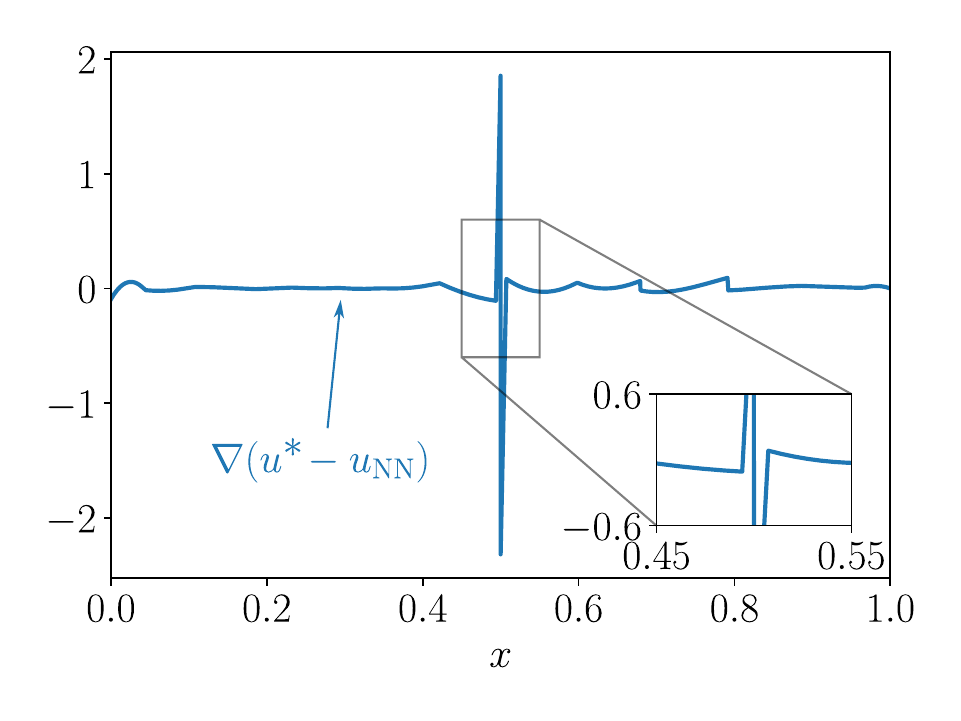}
        \caption{$\sigma =\operatorname{ReQU}$, $L=1$, and 150 iterations}
    \end{subfigure}\\
    \begin{subfigure}{0.49\linewidth}
        \includegraphics[width = \linewidth]{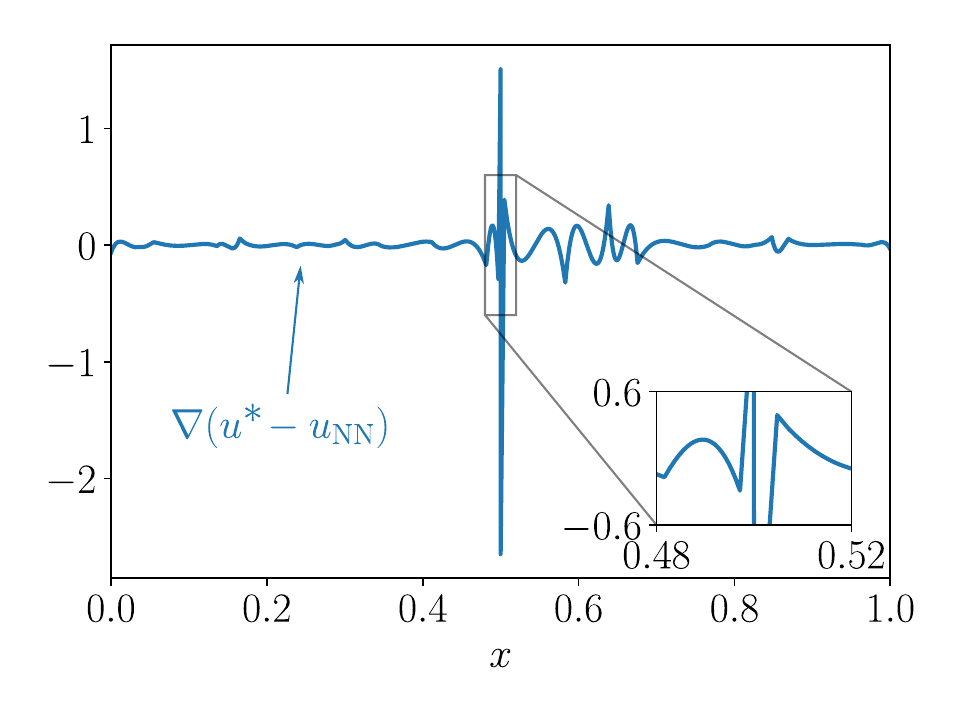}
        \caption{$\sigma =\operatorname{ReQU}(x)$, $L=2$, and 200 iterations}
    \end{subfigure}
    \begin{subfigure}{0.49\linewidth}
        \includegraphics[width = \linewidth]{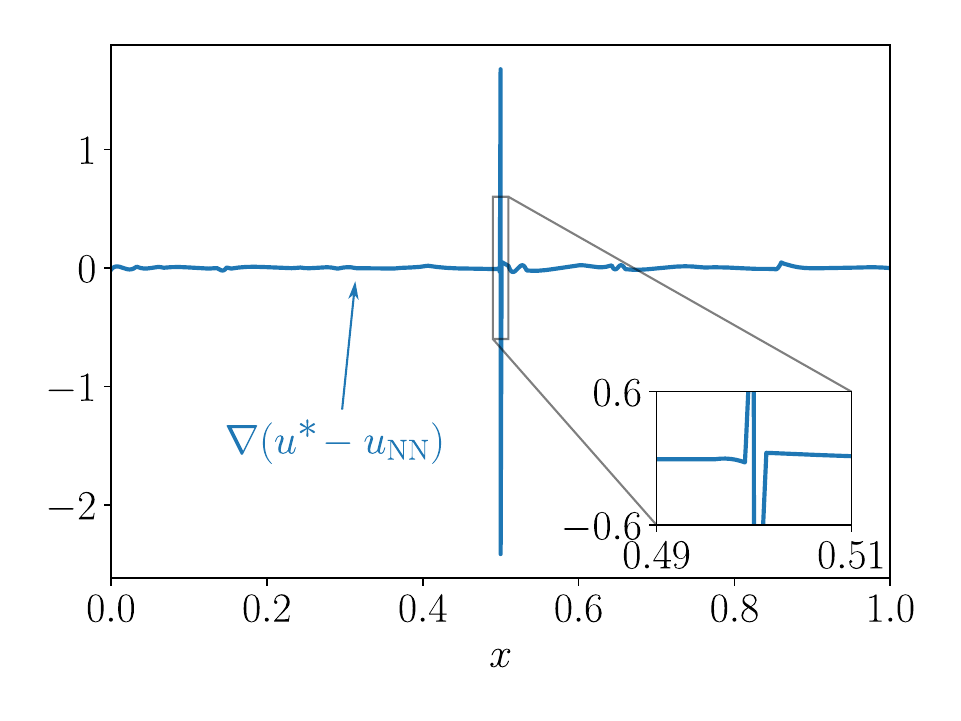}
        \caption{$\sigma = \operatorname{ReQU}(x)$, $L=2$ layers, and  $2\,000$ iterations}
    \end{subfigure}
    \caption{Gradient error $\nabla (u^*-u_{\text{NN}})$ with different neural network architectures.}\label{fig: gradient_error} 
\end{figure}

\subsubsection{Interface problem. Discontinuous diffusion term \label{sec:JamieInterface}}
We now train the neural network for $2\,500$ iterations to solve the problem \eqref{eq: interface_problem}. During the last $1\,000$ iterations we apply an exponential decay to the learning rate with a factor of $0.995$ per iteration. Figure \ref{fig: interface_approximation} shows the loss evolution and final approximations obtained. The pairs $(u_{\text{NN}}, \bfq_\text{NN})$ provide good approximations of $(u^*, \bfq^*)$. 
\begin{figure}[!ht]
     \begin{subfigure}{0.49\linewidth}
     \includegraphics[width = \linewidth]{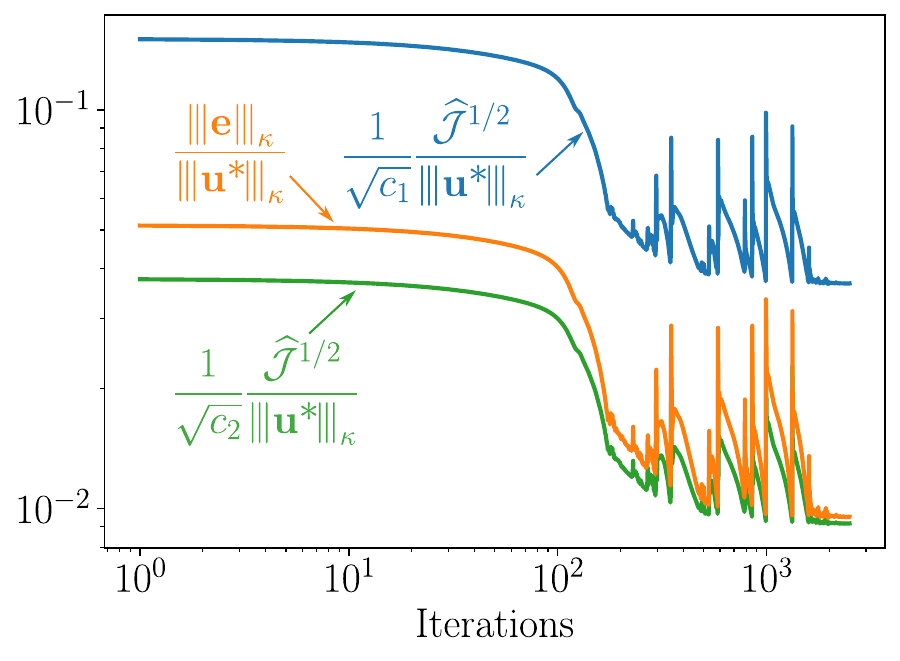}
     \caption{Loss-Error robustness \label{fig:loss_vs_error}}
    \end{subfigure}
    \begin{subfigure}{0.49\linewidth}
     \includegraphics[width = \linewidth]{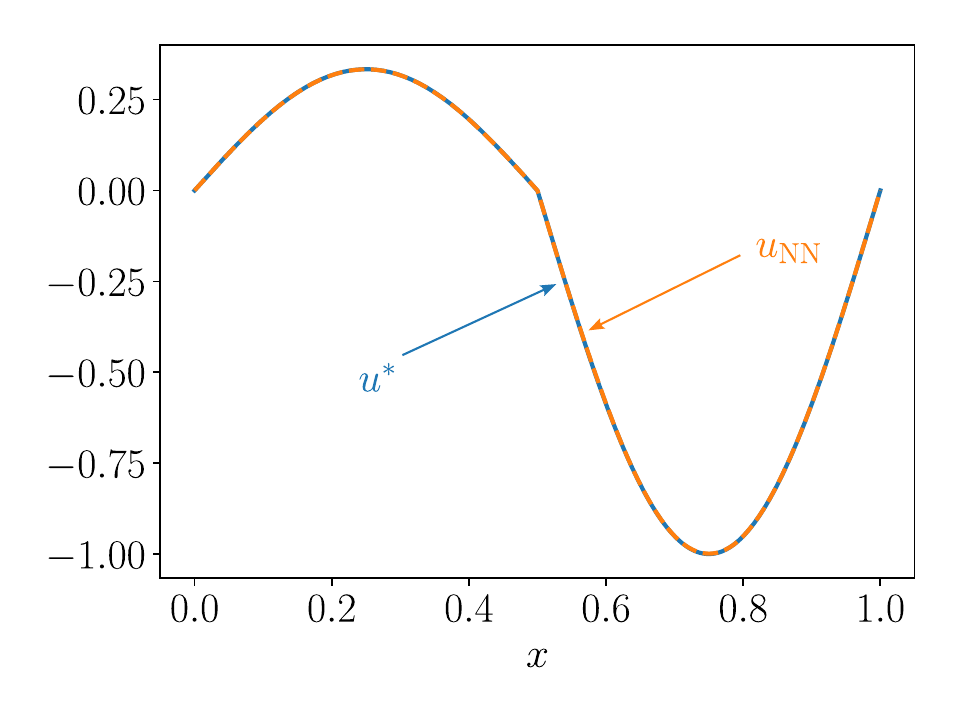}
      \caption{Approximation of $u^*$}
    \end{subfigure}\\
    \begin{subfigure}{0.49\linewidth}
    \centering
    \includegraphics[width = \linewidth]{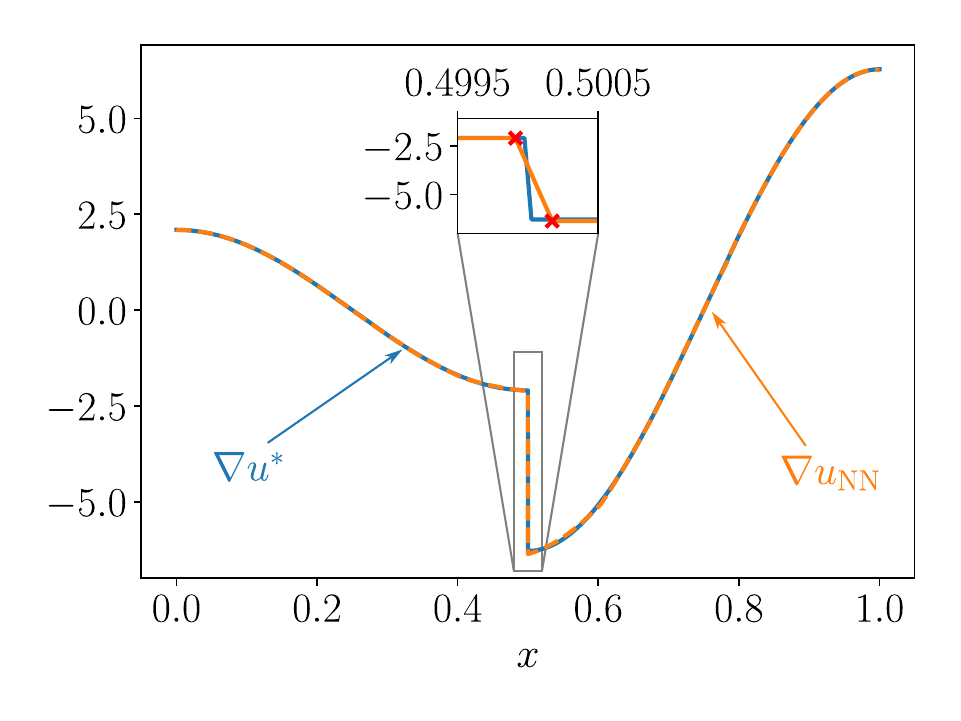}
    \caption{Gradient approximation  of $\nabla u^*$ \label{fig:grad_u}}
    \end{subfigure}
    \begin{subfigure}{0.49\linewidth}
    \centering
        \includegraphics[width =  \linewidth]{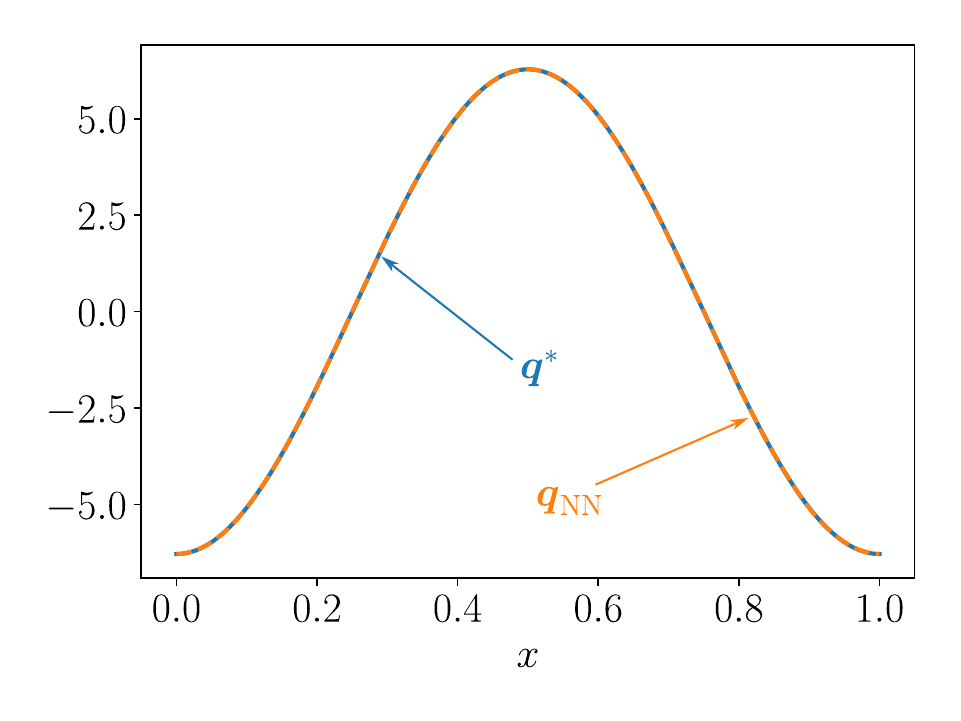}
        \caption{Flux approximation of $\bfq^*$}
    \end{subfigure}
    \caption{Loss and error evolution and final approximations. In panel (a), we denote $\mathbf{e} = (u^*, \bfq^*)-(u_{\text{NN}}, \bfq_{\text{NN}})$ and $\mathbf{u}^* = (u^*, \bfq^*)$.  The validation loss $\widehat{\mathcal{J}}$ was computed with a fine integration rule. The red crosses in panel (c) correspond to the evaluation of the gradient at the breaking points closest to the interface.  The relative error in $H_{0, \kappa }^1$ is $0.82\%$ and the relative error in $H(\Div,\kappa)$ is $0.98\%$. \label{fig: interface_approximation}}
\end{figure}
 Figure~\ref{fig: interface_errors} shows the relative energy norms in both $u$ and $\bfq$. Notice that across all iterations, the errors attained in $\bfq$ are much smaller than the errors in $u$. This occurs because the discontinuity in the exact solution is located in $u$ and it is therefore more difficult to approximate.

\begin{figure}[!ht]
\centering
   \includegraphics[width =0.49\linewidth]{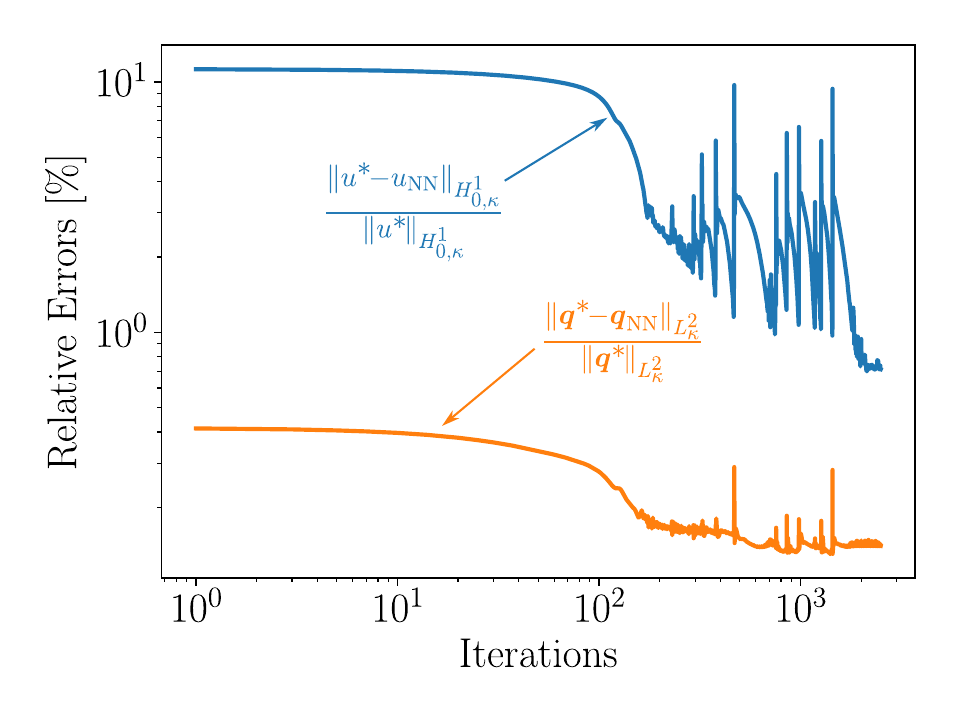}
    \caption{Relative errors in the energy norms of $u$ and $\bfq$. \label{fig: interface_errors}}
\end{figure}

\subsection{Two-dimensional experiments}
We consider the domain $\Omega = (0,1)^2$.  We  consider a neural network space with  $L=2$ hidden layers and $n_1=n_2 =32$. The least-squares problem and the energy-norm Poincaré constant estimation are performed in the same manner as in the 1D case. The Dirichlet lifting is $g_D = x(1-x)y(1-y)$. The initial learning rate is taken as $10^{-3}$, which decays over the last  $1\,000$ iterations at a rate of $0.995$. We employ a $100\times 100$ mesh for the stochastic integration. 
For the error evaluation, we approximate the norms using a tensor-product trapezoidal rule on a uniform Cartesian grid.  We consider two interface problems that differ in the component of the solution that exhibits less regularity: either the gradient $\nabla u^*$ or the flux $\bfq^*$.

\subsubsection{Transmission problem over a circular interface with discontinuous solution gradient}
We consider a problem given by a discontinuous coefficient defined across a circular interface:
\[
\kappa=
\begin{cases}
1, & |(x,y)- (1/2, 1/2)|\leq 1/4,\\[4pt]
3, & |(x,y)-(1/2, 1/2)|> 1/4.
\end{cases}
\]
For the exact solution, we select $u^* = \sin( 2 \pi x) \sin (2 \pi y) (|(x,y)-(1/2, 1/2)|^2- (1/4)^2 ) /\kappa$. With this setup, the flux $\bfq^*$ is everywhere continuous while the gradient $\nabla u^*$ exhibits a jump discontinuity at the circular interface.  We train a neural network for $25\,000$ iterations.  

Figure~\ref{fig:circular_interface_results} shows the final results.  Figures~\ref{fig: gradu_error_circ}--\ref{fig: q_error_circ} provide a direct comparison of \(u_{\text{NN}}\) and \(\bfq_{\text{NN}}\) through the same quantity $\kappa^{1/2} |\nabla u^*| = \kappa^{-1/2} |\bfq^*|$. The approximation of \(\nabla u^*\) is less accurate than that of \(\bfq^*\), as it exhibits a discontinuity, which requires a precise positioning of the breaking curves to properly approximate the interface. 
Figure~\ref{fig:circular_gradient} shows the directional derivatives of $u$ along various cross-sections orthogonal to the interface, as well as the breaking curves associated to the second hidden layer $\mathcal{S}^2 = \{\mathcal{S}^{(2)}_j\}_{j=1}^{n_2}$, where  $\mathcal{S}^{(2)}_j$ is defined as in \eqref{eq:breaking_surfaces}. The quasi-Gibbs phenomenon is negligible across all observed sections.

\begin{figure}[!ht]
    \centering
    \begin{subfigure}{0.49\linewidth}
         \includegraphics[width=\linewidth]{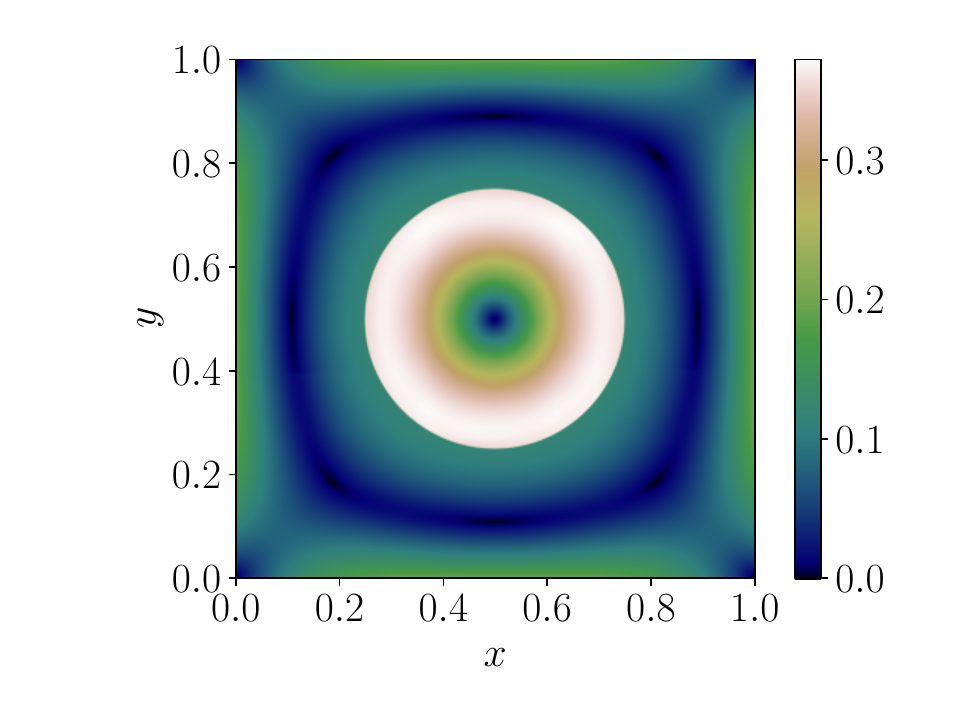}
         \caption{$|\nabla u_{\text{NN}}|$}
    \end{subfigure}
    \begin{subfigure}{0.49\linewidth}
         \includegraphics[width=\linewidth]{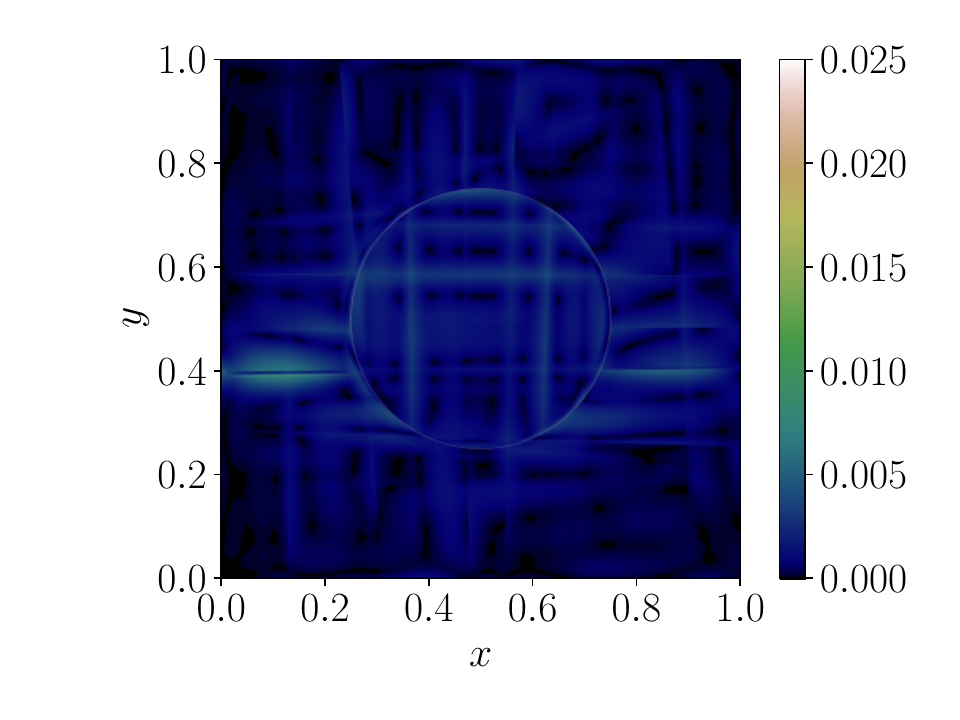}
          \caption{$|\kappa^{1/2}\nabla(u^*-u_{\text{NN}})|$ \label{fig: gradu_error_circ}}
    \end{subfigure}\\
    \begin{subfigure}{0.49\linewidth}
        \includegraphics[width=\linewidth]{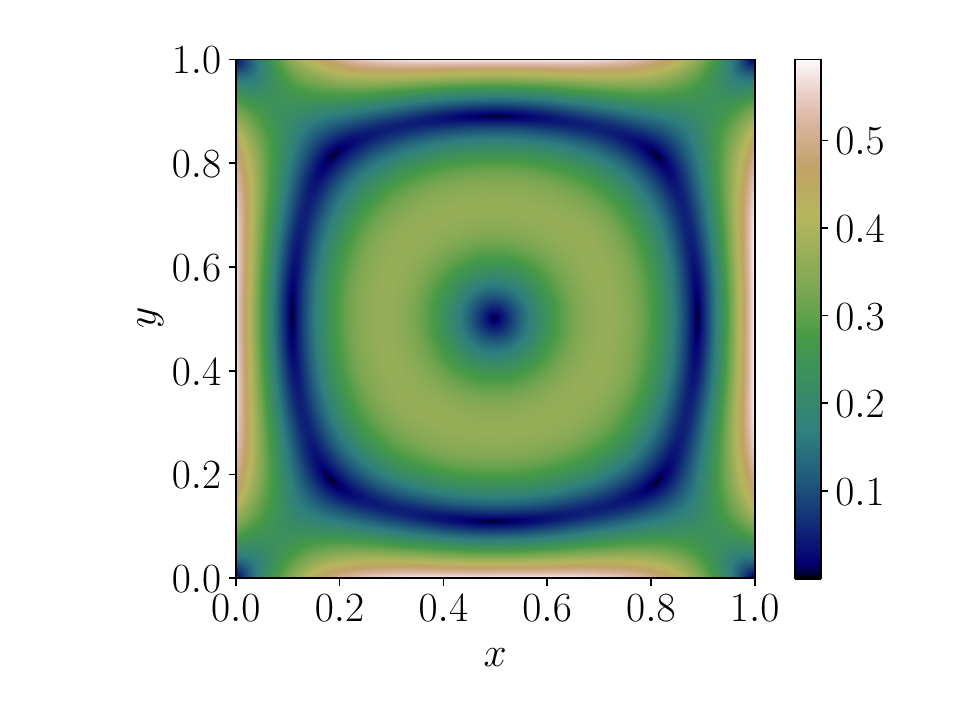}
        \caption{$|\bfq_{\text{NN}}|$ \label{fig: q_error_circ}}
    \end{subfigure}
     \begin{subfigure}{0.49\linewidth}
         \includegraphics[width=\linewidth]{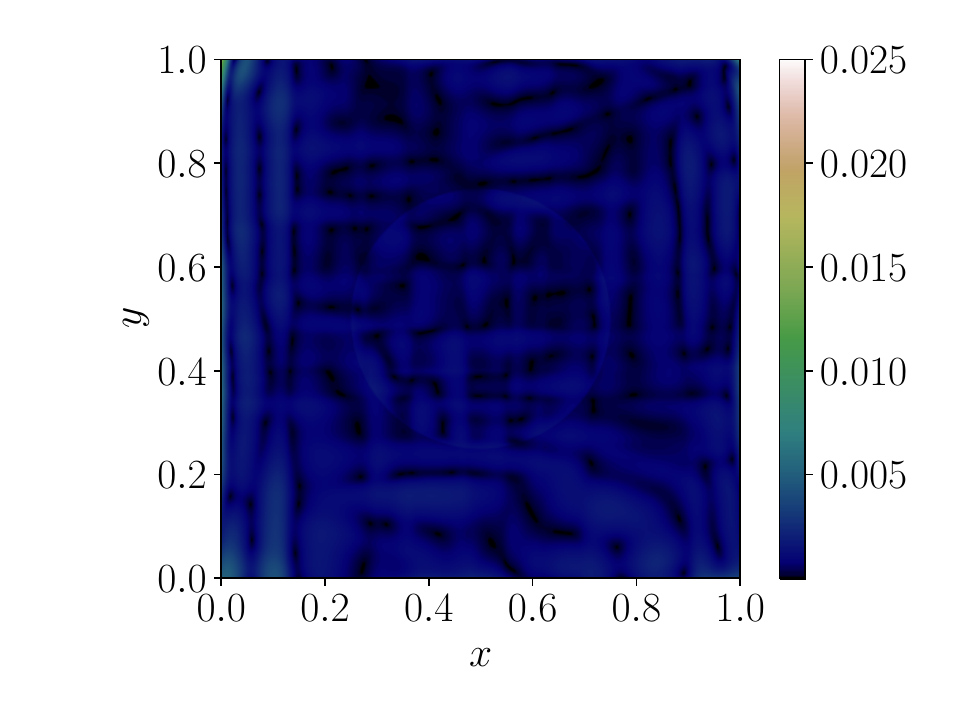}
          \caption{$|\kappa^{-1/2}(\bfq^*-\bfq_{\text{NN}})|$}
    \end{subfigure}
    \caption{Gradient and flux approximations for a problem with a discontinuity in the gradient. Following the energy norms, the errors are weighted to have a comparable measure of the error.  In panel (b), we capped the magnitude to 0.025 for proper visualization of the errors. The final relative error in $H_{0,\kappa}^1$ is $1.83\%$ and in $H(\Div,\kappa)$ is $0.32\%$.}
    \label{fig:circular_interface_results}
\end{figure}

\begin{figure}[!ht]
    \centering
    \begin{subfigure}{0.49\linewidth}
         \includegraphics[width=\linewidth]{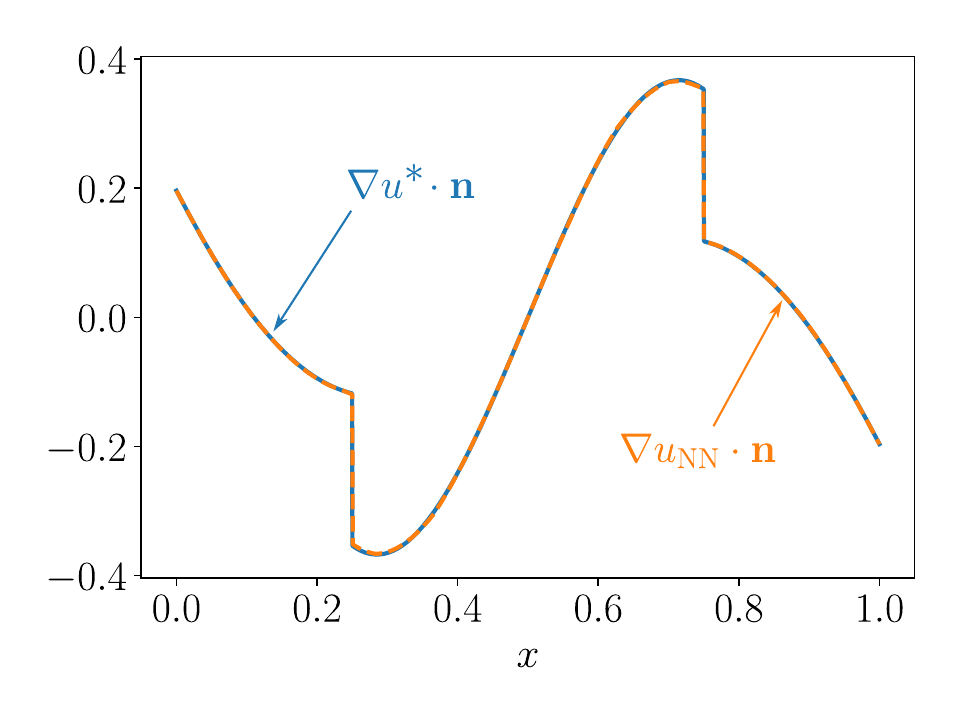}
         \caption{$\mathbf{n} = (1,0)$, at $y =0.5$}
    \end{subfigure}
    \begin{subfigure}{0.49\linewidth}
         \includegraphics[width=\linewidth]{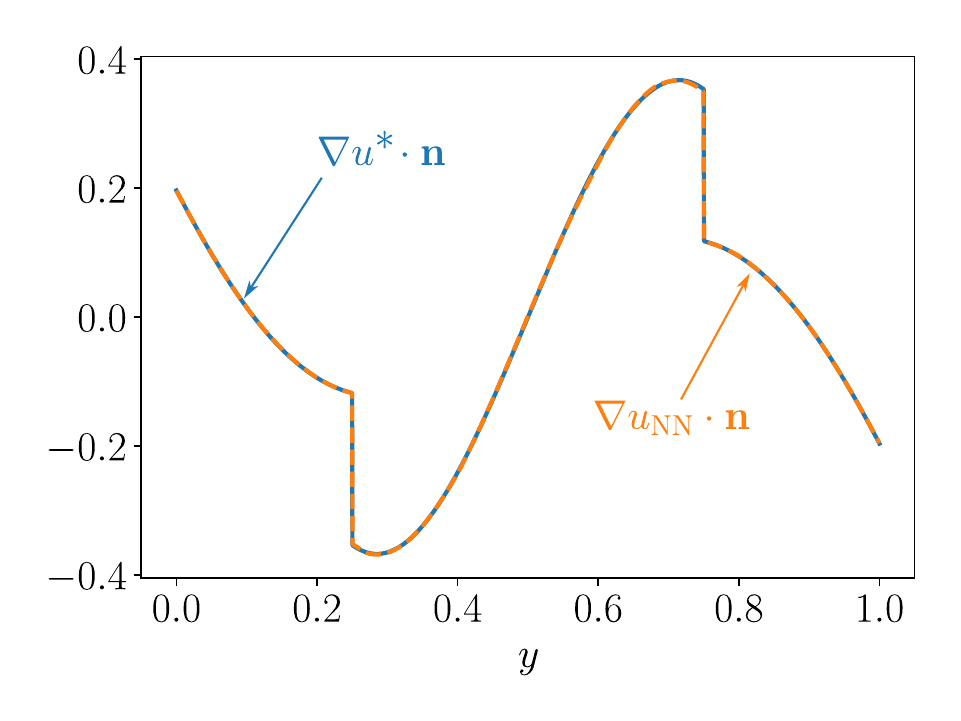}
          \caption{$\mathbf{n} = (0,1)$, at $x =0.5$ \label{fig: gradu_xy}}
    \end{subfigure}\\
    \begin{subfigure}{0.49\linewidth}
        \includegraphics[width=\linewidth]{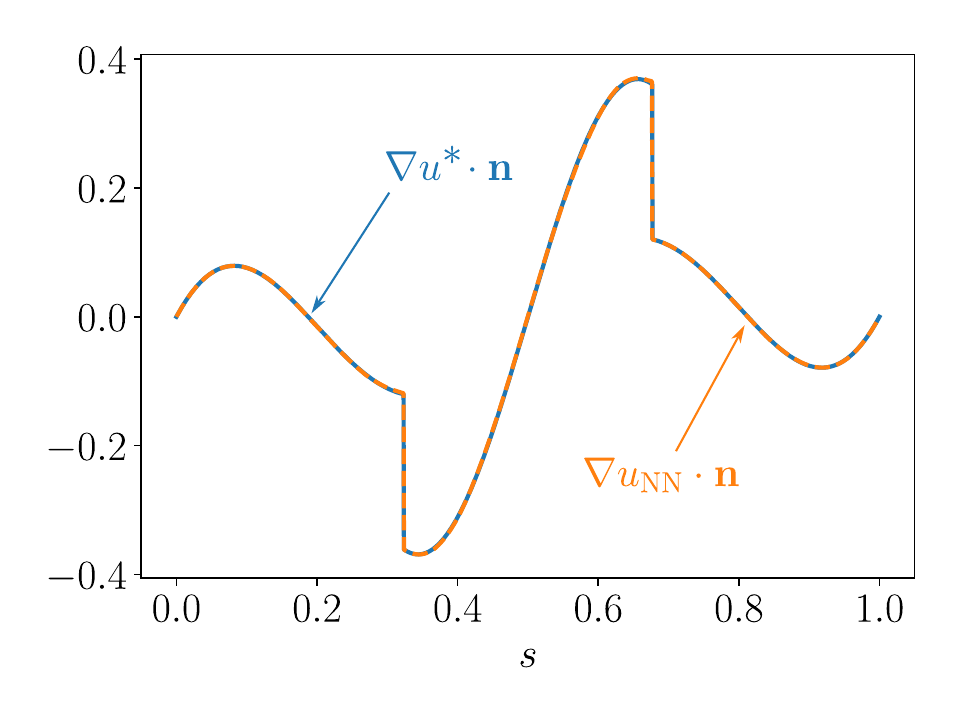}
        \caption{$\mathbf{n} = (1/{\sqrt2},1/{\sqrt2})$, at $ y=x$ \label{fig: grad_anti_diag}}
    \end{subfigure}
     \begin{subfigure}{0.49\linewidth}
         \includegraphics[width=\linewidth]{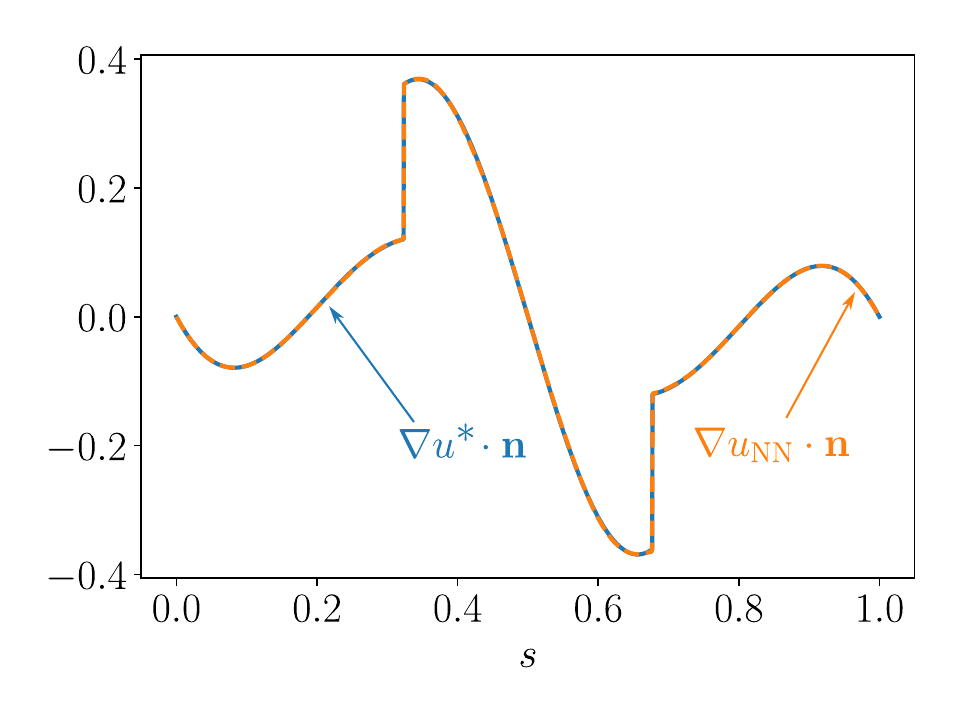}
          \caption{$\mathbf{n} = (-1/{\sqrt2},1/{\sqrt2})$, at $y=1-x$}
    \end{subfigure}
    \begin{subfigure}{\linewidth}
    \centering
         \includegraphics[width=0.5\linewidth]{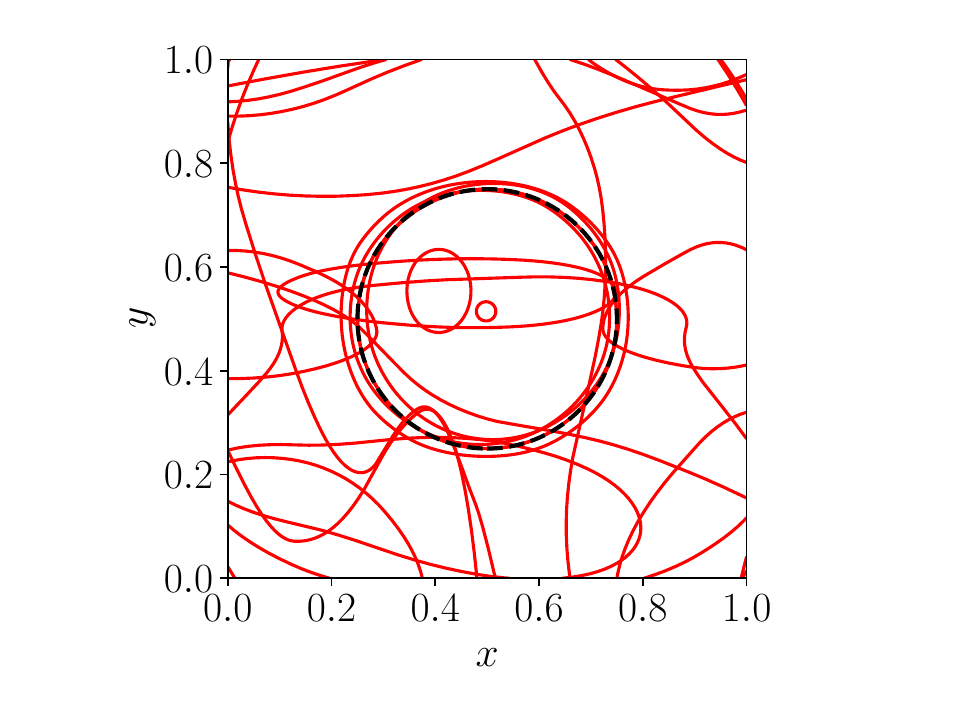}
          \caption{Breaking curves  $\mathcal{S}^{2}$\label{fig:breaking_curves_circ}}
    \end{subfigure}
    \caption{Directional derivative approximation on different cross-sections and the breaking curves $\mathcal{S}^{2}$. In panels (c) and (d), we use  $s\in [0,1]$  to denote the parameterization variable of the corresponding sections. Panel (e) shows in red the breaking curves of the last hidden layer, while the interface is represented by a dashed black line.}
    \label{fig:circular_gradient}
\end{figure}

\subsubsection{Transmission problem over a plane interface with discontinuous flux}
We now consider a problem with the following diffusion coefficient
\[
\kappa=
\begin{cases}
1, & (x, y)\in \bigl(0,\tfrac12\bigr) \times (0,1),\\[4pt]
3, & (x,y)\in \bigl(\tfrac12,1\bigr)\times (0,1).
\end{cases}
\]
For the exact solution, we select $u^* = (\cos (2\pi x)-1)\sin( \pi y)$. This setup induces a discontinuity in the component of the flux tangent to the interface  $\bfq^*_y$ rather than in the gradient $\nabla u^*$. Since $\partial_x u^*|_{x = 1/2} = 0$ the normal flux is continuous across the interface, ensuring that $\bfq^*$ remains in  $ H(\Div)$. We train a neural network for $10\, 000$ iterations. 

Figure~\ref{fig:plane_results} shows the final results. In contrast to the previous example, we now observe that  the approximation  \(\kappa^{-1/2}|\bfq^*|\) is less accurate than that of \(\kappa^{1/2}|\nabla u^*|\). This occurs because now the discontinuity appears in the tangential component of the flux $\bfq^*$ rather than on the gradient of the solution $\nabla u^*$. Figure~\ref{fig:plane_flux} further shows that the neural network has difficulties approximating the tangent component of the flux on different sections due to the discontinuities.

\begin{figure}[!ht]
    \centering
    \begin{subfigure}{0.49\linewidth}
         \includegraphics[width=\linewidth]{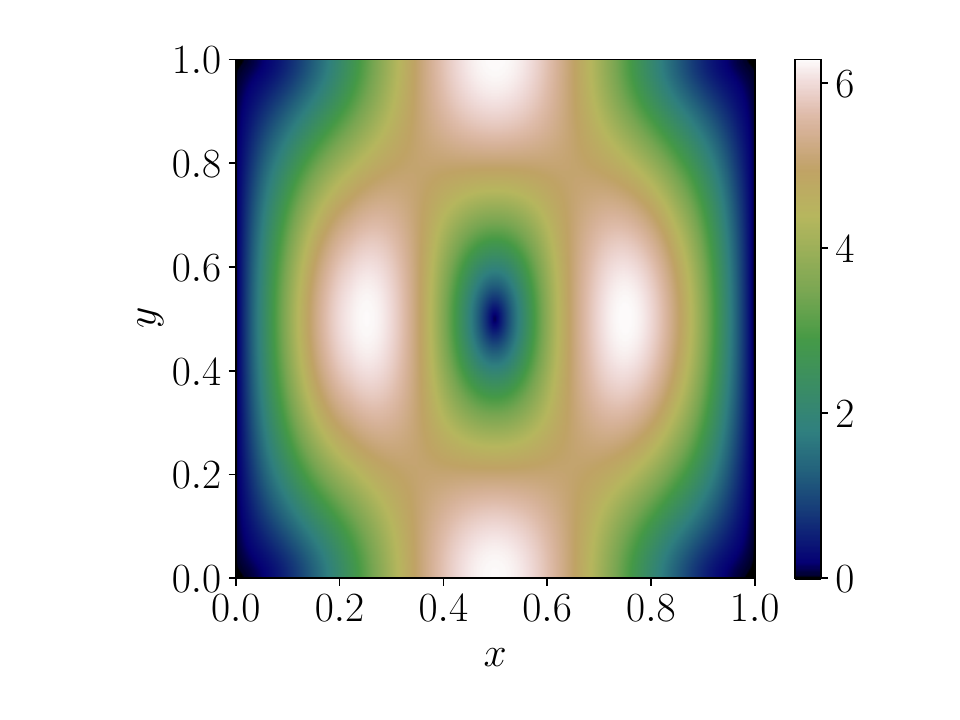}
         \caption{$|\nabla u_{\text{NN}}|$}
    \end{subfigure}
    \begin{subfigure}{0.49\linewidth}
         \includegraphics[width=\linewidth]{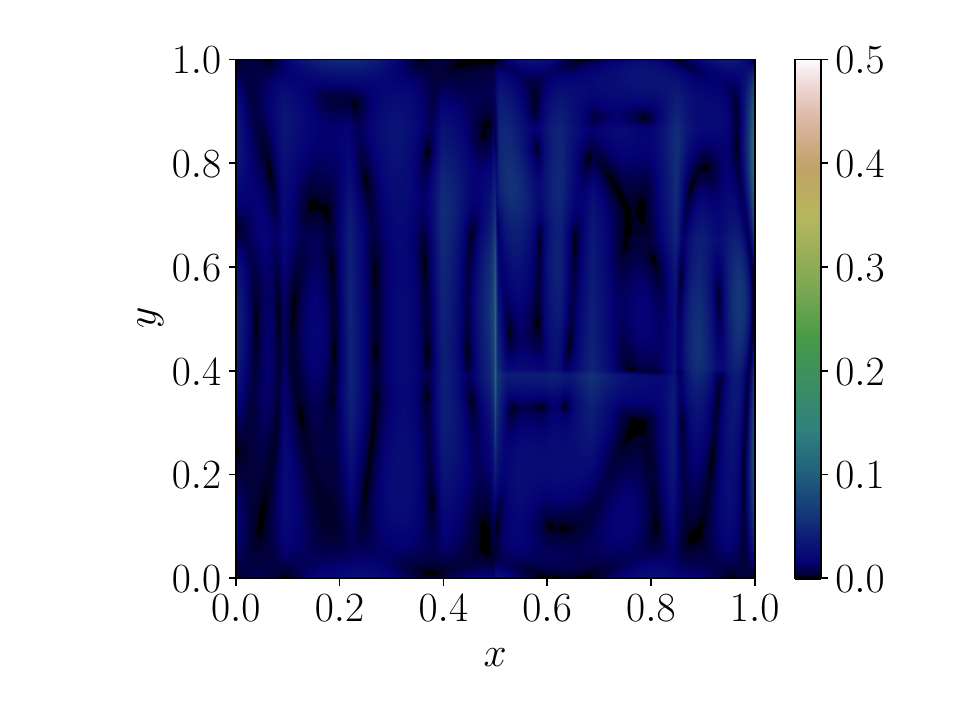}
          \caption{$|\kappa^{1/2}\nabla(u^*-u_{\text{NN}})|$}
    \end{subfigure}\\
    \begin{subfigure}{0.49\linewidth}
        \includegraphics[width=\linewidth]{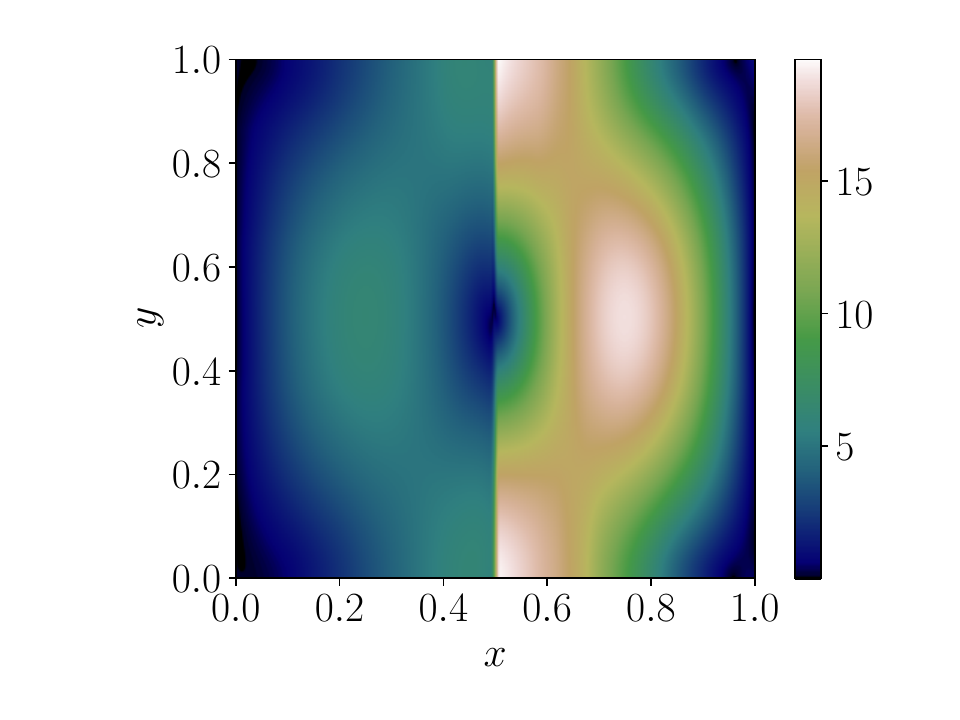}
        \caption{$|\bfq_{\text{NN}}|$}
    \end{subfigure}
     \begin{subfigure}{0.49\linewidth}
         \includegraphics[width=\linewidth]{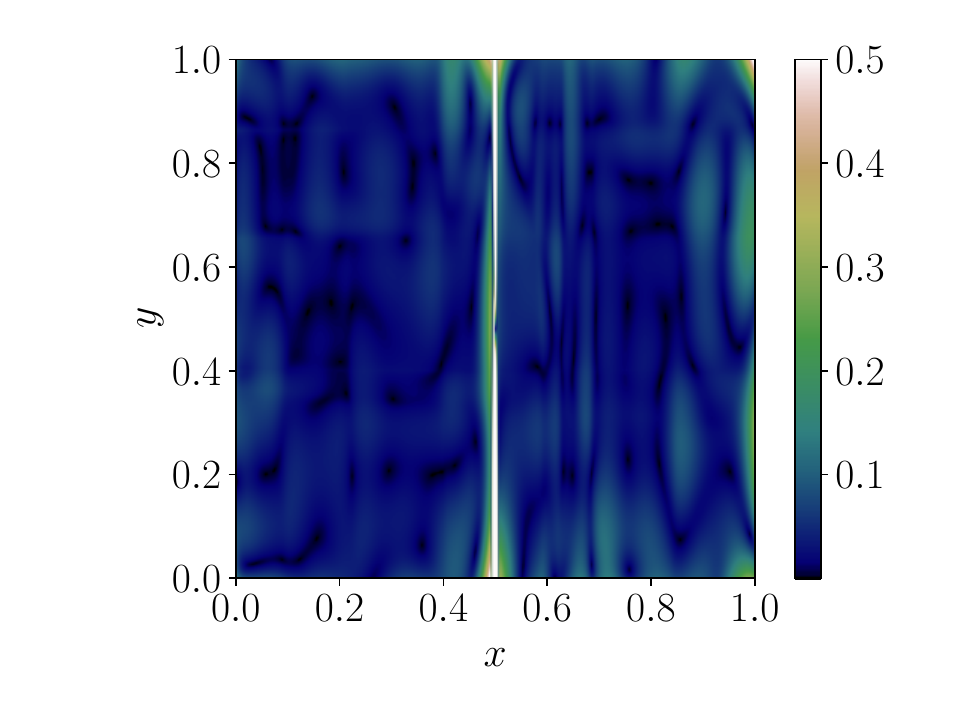}
          \caption{$|\kappa^{-1/2}(\bfq^*-\bfq_{\text{NN}})|$ \label{fig:flux_error}}
    \end{subfigure}
    \begin{subfigure}{0.49\linewidth}
        \includegraphics[width=\linewidth]{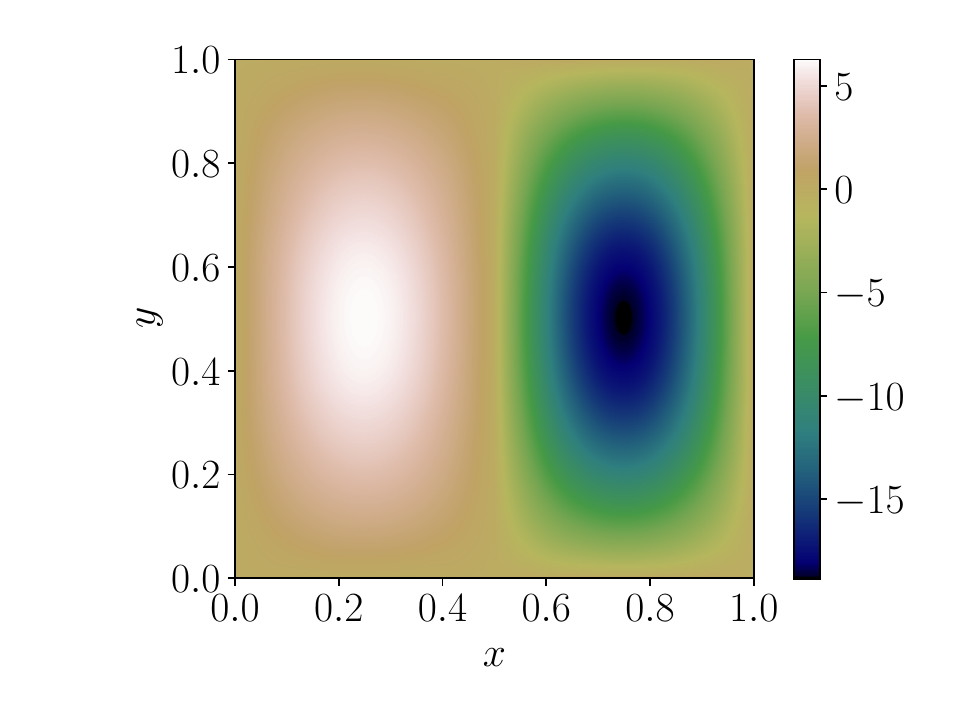}
        \caption{$\bfq_{\text{NN}}^{x}$}
    \end{subfigure}
     \begin{subfigure}{0.49\linewidth}
         \includegraphics[width=\linewidth]{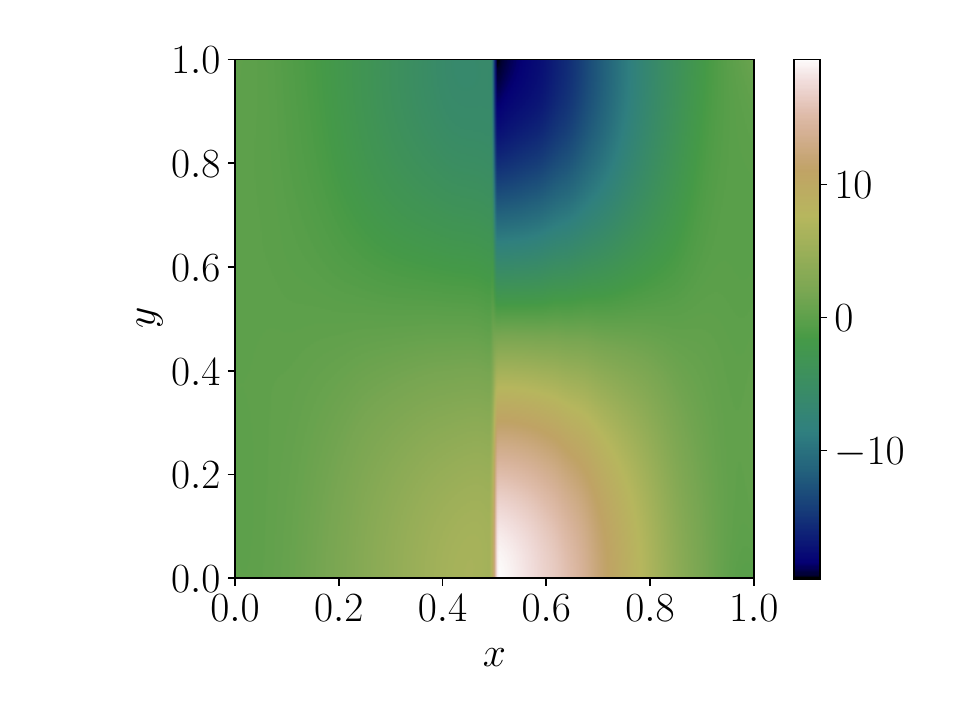}
          \caption{$\bfq_{\text{NN}}^{y}$ \label{fig:q_tangent}}
    \end{subfigure}
    \caption{Gradient and flux approximation for a problem with a discontinuity in the flux. Following the energy norms, the errors are weighted to have a comparable measure of the error. In panel (b), we capped the magnitude to 0.5 for proper visualization of the errors. The final relative error in $H_{0,\kappa}^1$ is $0.39\%$, and in $H(\Div, \kappa)$ is $2.02\%$.}
    \label{fig:plane_results}
\end{figure}

\begin{figure}[!ht]
    \centering
    \begin{subfigure}{0.49\linewidth}
         \includegraphics[width=\linewidth]{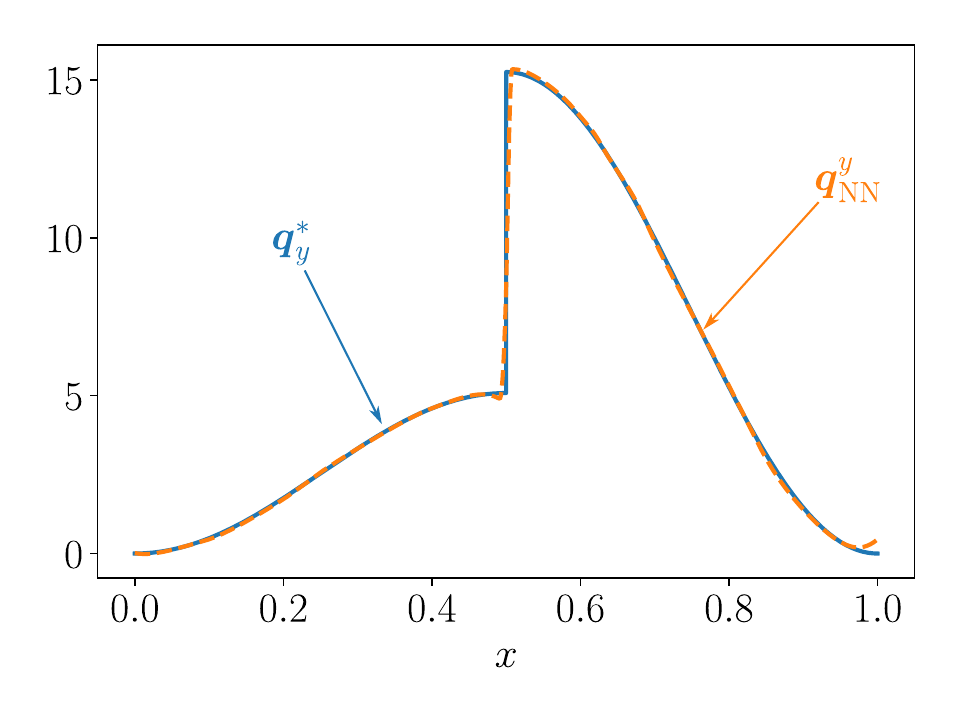}
         \caption{$\bfq^y$, at  $y=0.2$}
    \end{subfigure}
    \begin{subfigure}{0.49\linewidth}
        \includegraphics[width=\linewidth]{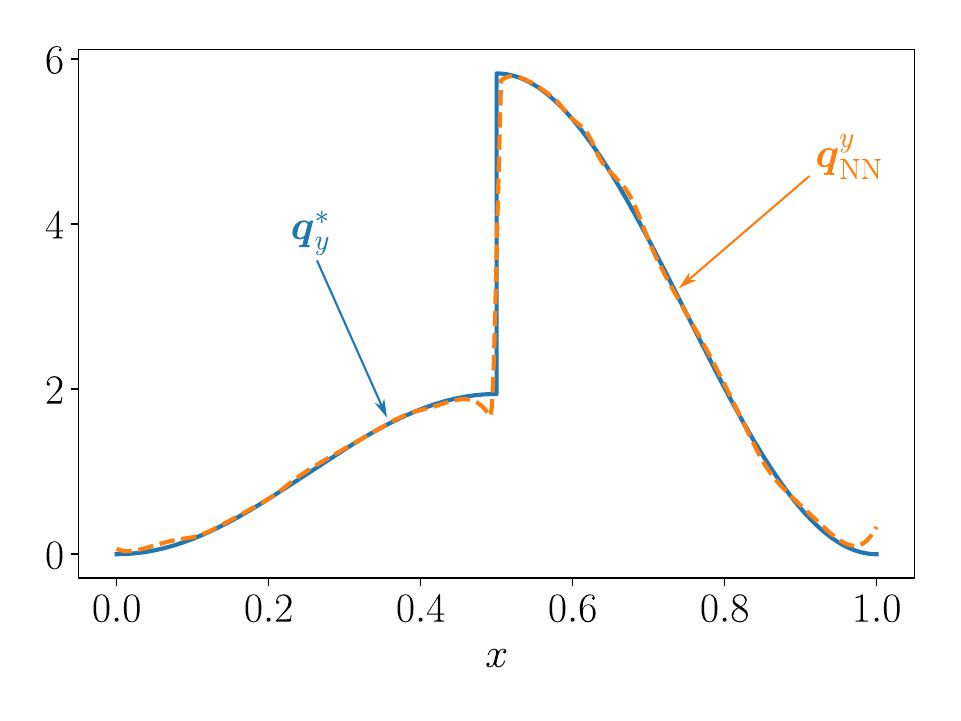}
           \caption{$\bfq^y$, at  $y=0.4$}
    \end{subfigure}\\
    \begin{subfigure}{0.49\linewidth}
       \includegraphics[width=\linewidth]{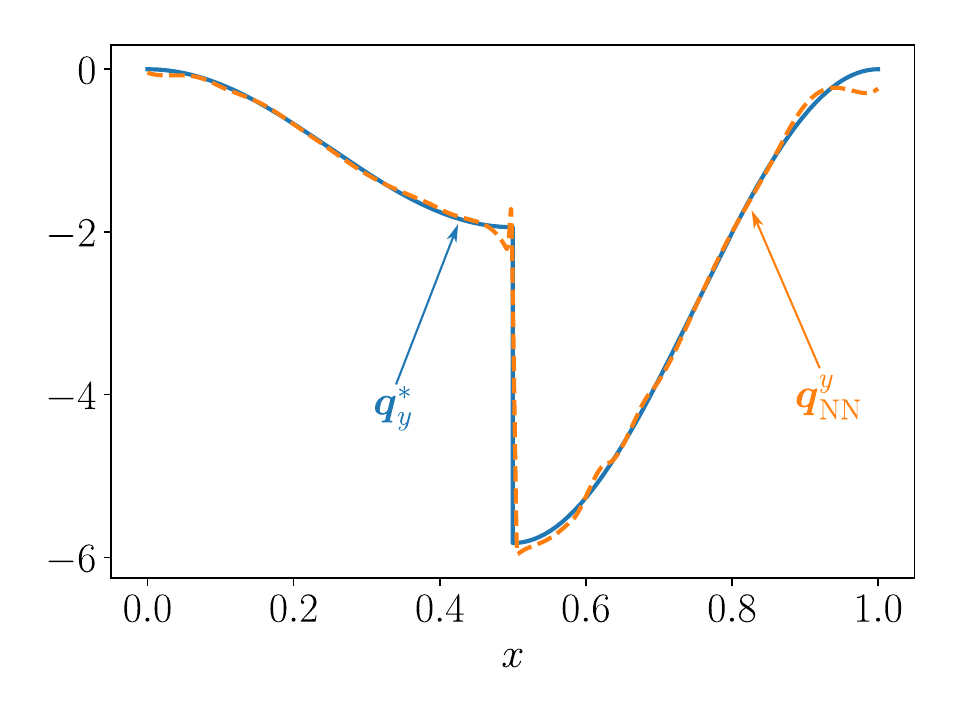}
         \caption{$\bfq^y$, at  $y=0.6$}
    \end{subfigure}
     \begin{subfigure}{0.49\linewidth}
        \includegraphics[width=\linewidth]{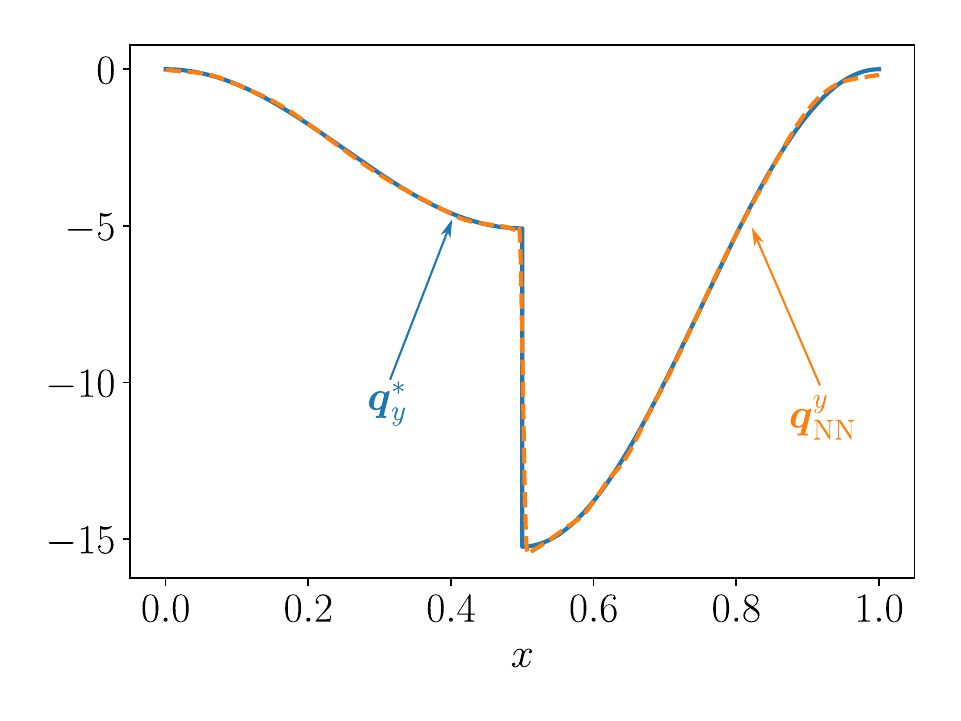}
           \caption{$\bfq^y$, at $y=0.8$}
    \end{subfigure}
    \begin{subfigure}{\linewidth}
        \centering
         \includegraphics[width=0.5\linewidth]{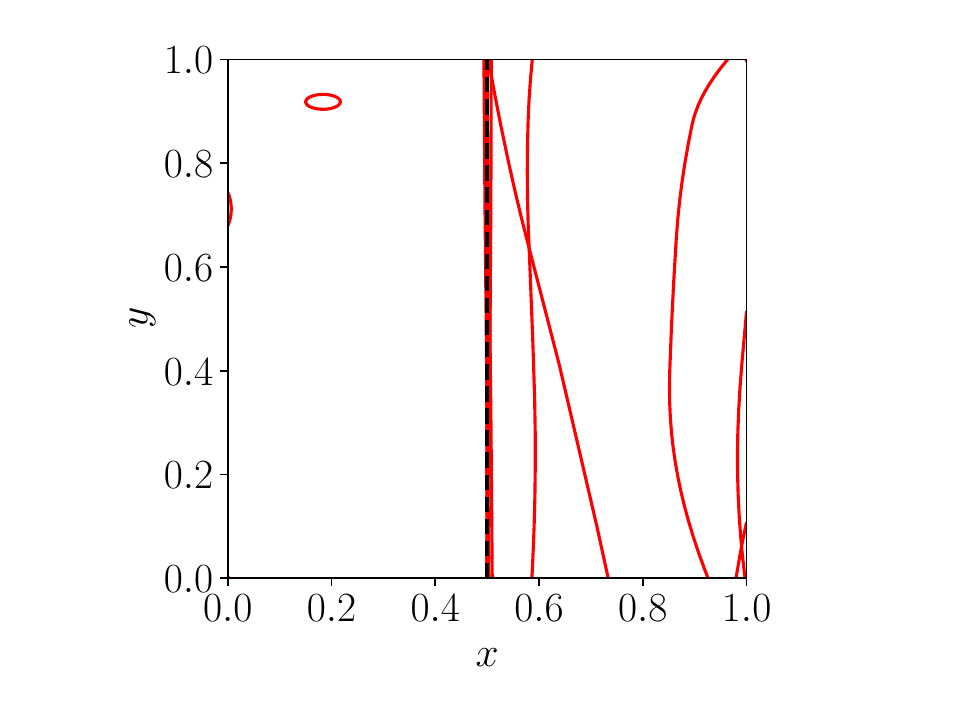}
        \caption{Breaking curves $\mathcal{S}^{2}$ \label{fig:breaking_curves_plane}}
    \end{subfigure}
    \caption{Approximation of the flux component tangent to the plane interface across different sections and breaking curves $\mathcal{S}^2$. Panel (e) shows in red the breaking curves of the last hidden layer, while the interface is represented by a dashed black line.}
    \label{fig:plane_flux}
\end{figure}

 These numerical examples demonstrate that discontinuities pose significant approximation challenges. In both cases, the discontinuous components exhibit larger relative errors than the smooth components of the solution.

\section{Conclusions}
In this work, we introduced a novel FOSLS functional and an associated robust energy norm for transmission problems. We also proposed a neural network training methodology that combines the optimality properties of least-squares solvers with the flexibility of gradient-based methods for searching over optimal trial spaces. Moreover, we showed that the gradient-based loss induced by the FOSLS functional yields a passive reduction in gradient variance, a property not satisfied by other standard weak formulations commonly used in the literature. We further observed that the neural approximation spaces accurately reproduce the associated Poincaré-type approximation constant, showing that the learned spanning functions capture the relevant structure of the solution space.

Our numerical experiments indicate that passive variance reduction plays an important role in stabilizing neural network training. We further illustrated the robustness of the proposed formulation for transmission problems, contrasting it with the standard FOSLS formulation. In addition, the use of the $\operatorname{ReQU}$ activation in neural networks was observed to reduce the quasi-Gibbs phenomenon compared  to smoother activation functions.  

The present study was restricted to problems without singularities, as these introduce additional analytical and numerical challenges. We leave as future work incorporating prior information about the singularity into the approximation space, for example, through tailored architectures following the ideas in \cite{taylor2025regularity}. In the experiments reported in this work, the numerical integration was performed on uniform  meshes. The use of adaptive grids for more accurate integration is another interesting direction for future investigation.

\section*{Acknowledgments}
Alejandro Duque-Salazar is supported by the University of the Basque Country (UPV/EHU) through a predoctoral contract for research training (Personal Investigador en Formación, PIF), under the 2024 call for the recruitment of research staff in training. Paulina Sep\'{u}lveda is supported by the grant N. FOVI240119 (Fomento a la Vinculación Internacional para instituciones de Investigación), ANID Chile. Carlos Uriarte is supported by the Postdoctoral Program for the Improvement of Doctoral Research Personnel of the Basque Government (Ref. No. POS-2024-1-0004) and by the Research Projects PID2023-146678OB-I00 and PID2023-146668OA-I00, both funded by MICIU/AEI /10.13039/501100011033. Jamie M. Taylor has received funding from the following Research Projects: PID2023-146678OB-I00 funded by MICIU/AEI /10.13039/501100011033. David Pardo has received funding from the following Research Projects/Grants: 
European Union’s Horizon Europe research and innovation programme under the Marie Sklodowska-Curie Action MSCA-DN-101119556 (IN-DEEP). 
PID2023-146678OB-I00 funded by MICIU/AEI /10.13039/501100011033 and by FEDER, EU;
BCAM Severo Ochoa accreditation of excellence CEX2021-001142-S funded by MICIU/AEI/10.13039/501100011033; 
Basque Government through the BERC 2022-2025 program;
RUL-ET(KK-2024/00086), funded by the Basque Government through ELKARTEK;
Consolidated Research Group MATHMODE (IT1866-26) of the UPV/EHU given by the Department of Education of the Basque Government; 
BCAM-IKUR-UPV/EHU, funded by the Basque Government IKUR Strategy and by the European Union NextGenerationEU/PRTR.

\clearpage

\bibliography{bib.bib}
\appendix

\section{Derivation of first eigenvalue constant for planar geometries}
\label{Appendix-poincare}

We consider the one-dimensional  problem with $\Omega = (0,1)$, $x_0\in \Omega$, and homogeneous Dirichlet boundary conditions, where
$$\kappa(x) =\left\{\begin{array}{ll}\kappa_1, & 0\leq x\leq x_0, \\
\kappa_2, &  x_0 \leq x\leq 1,
\end{array}\right.$$
with $\kappa_1,\kappa_2>0$. Since $\kappa$ is constant on each subinterval,  the eigenfunctions are given piecewise by
\[
u(x) =
\begin{cases}
A \sin\!\left(\sqrt{\frac{\lambda}{\kappa_1}}\, x\right), & 0 \le x \le x_0, \\
B \sin\!\left(\sqrt{\frac{\lambda}{\kappa_2}}\, (1-x)\right), & x_0 \le x \le 1.
\end{cases}
\]

Imposing continuity of $u$ and of the flux $-\kappa u'$ at $x=x_0$ yields
\[
A \sin\!\left(\sqrt{\frac{\lambda}{\kappa_1}}\, x_0\right)
=
B \sin\!\left(\sqrt{\frac{\lambda}{\kappa_2}}\, (1-x_0)\right),
\]
and
\[
\kappa_1 A \sqrt{\frac{\lambda}{\kappa_1}} \cos\!\left(\sqrt{\frac{\lambda}{\kappa_1}}\, x_0\right)
=
-\kappa_2 B \sqrt{\frac{\lambda}{\kappa_2}} \cos\!\left(\sqrt{\frac{\lambda}{\kappa_2}}\, (1-x_0)\right).
\]
Eliminating $A$ and $B$ leads to the transcendental equation:
$$\sqrt{\kappa_2}\tan\left(\sqrt{\frac{\lambda}{\kappa_1}} x_0\right)+
\sqrt{\kappa_1}\tan\left(\sqrt{\frac{\lambda}{\kappa_2}}(1 -x_0)\right)=0.
$$
This equation characterizes a discrete sequence of positive eigenvalues. In particular, the smallest positive eigenvalue $\lambda_1$ is uniquely defined. 

\end{document}